\documentclass{amsart}

\usepackage{amsmath,amsthm,amsfonts,eucal,eufrak}

\usepackage{latexsym}
\usepackage{amssymb}

\usepackage[active]{srcltx}
\usepackage[dvips]{epsfig}

\usepackage[usenames,dvipsnames]{xcolor}
\usepackage[pdfusetitle,
bookmarks=true,bookmarksnumbered=true,bookmarksopen=true,bookmarksopenlevel=1,
breaklinks=false,pdfborder={0 0
  0},backref=false,colorlinks=true,linkcolor=NavyBlue,citecolor=NavyBlue,urlcolor=black]
{hyperref}
\pdfpageheight\paperheight
\pdfpagewidth\paperwidth
\renewcommand{\Im}{\operatorname{Im}}
\renewcommand{\Re}{\operatorname{Re}}
\newcommand{\T}{\mathbb{T}}
\newcommand{\norm}[1]{\|#1\|}
\newcommand{\mnorm}[1]{\left\|#1\right\|}
\newcommand{\normhs}[1]{\|#1\|_{\mathrm{HS}}}

\newcommand{\C}{\mathbb{C}}
\renewcommand{\epsilon}{\varepsilon}
\renewcommand{\Im}{\operatorname{Im}}
\usepackage[capitalise]{cleveref}
\crefformat{equation}{(#2#1#3)}

\let\deg\relax
\DeclareMathOperator{\deg}{deg}
\DeclareMathOperator{\spec}{spec}
\DeclareMathOperator{\rank}{rank}

\newcommand{\fS}{\mathfrak{S}}

\newcommand{\oN}{\overline N}

\newcommand{\oell}{\overline \ell}

\newcommand{\cA}{{\mathcal{A}}}

\newcommand{\car}{\mathop{\rm{Car}}\nolimits}

\newcommand{\Car}{\mbox{\rm Car}}

\newcommand{\nn}{\nonumber}

\newcommand{\uw}{{\underline{w}}}
\newcommand{\cB}{{\mathcal{B}}}
\newcommand{\cD}{{\mathcal{D}}}

\newcommand{\cG}{{\mathcal{G}}}

\newcommand{\cN}{{\mathcal{N}}}
\newcommand{\cR}{{\mathcal{R}}}

\newcommand{\cT}{{\mathcal{T}}}
\newcommand{\cS}{{\mathcal{S}}}

\newcommand{\cP}{{\mathcal{P}}}

\newcommand{\IC}{{\mathbb{C}}}

\newcommand{\IZ}{{\mathbb{Z}}}
\newcommand{\uz}{{\underline{z}}}
\newcommand{\Res}{\mathop{\rm{Res}}}

\newcommand{\La}{\Lambda}

\renewcommand{\deg}{\mbox{\rm deg}}

\renewcommand{\mod}{{\rm mod\ }}

\newcommand{\be}{\begin{eqnarray}}
\newcommand{\ee}{\end{eqnarray}}

\newcommand{\supp}{\operatorname{supp}}

\newcommand{\R}{{\mathbb R}}
\newcommand{\tor}{{\mathbb T}}
\newcommand{\Z}{{\mathbb Z}}

\newcommand{\dist}{\mbox{\rm dist}}
\newcommand{\mes}{{\rm mes}}

\newcommand{\Proj}{{\rm Proj}}

\renewcommand{\mod}{{\rm{mod}\, }}

\newtheorem{theorem}{Theorem}[section]
\newtheorem{thma}{Theorem}

\newtheorem{thmb}{Theorem}

\newtheorem{thmc}{Theorem}

\newtheorem{thmd}{Theorem}

\newtheorem{thme}{Theorem}

\newtheorem{lemma}[theorem]{Lemma}
\newtheorem{cor}[theorem]{Corollary}
\newtheorem{prop}[theorem]{Proposition}

\theoremstyle{definition}
\newtheorem{defi}[theorem]{Definition}

\theoremstyle{remark}
\newtheorem{remark}[theorem]{Remark}

\numberwithin{equation}{section}

\begin{document}
\title[On multi-frequency quasi-periodic operators]
{On localization and the spectrum of multi-frequency quasi-periodic operators}

\date{}
\author{ Michael Goldstein, Wilhelm Schlag, Mircea Voda}

\address{Department of Mathematics, University of Toronto,
  Toronto, Ontario, Canada M5S 1A1}
\email{gold@math.toronto.edu}

\address{Department of Mathematics, The University of Chicago, 5734 S. University Ave., Chicago, IL 60637, U.S.A.}
\email{schlag@math.uchicago.edu}

\address{Department of Mathematics, The University of Chicago, 5734 S. University Ave., Chicago, IL 60637, U.S.A.}
\email{mvoda@uchicago.edu}

\thanks{The second author was partially supported by the NSF, DMS-1500696. The first author thanks the University of Chicago for its hospitality during the months of July and August of 2016.}

\maketitle

\begin{abstract}
  We study multi-frequency quasi-periodic Schr\"{o}dinger operators on $\Z$ in the regime of positive Lyapunov exponent and for general analytic potentials.
  Combining  Bourgain's semi-algebraic elimination of multiple resonances~\cite{Bou07} with the method of elimination of
  double resonances from \cite{GolSch11},  we establish exponential finite-volume localization as well as
    the separation between the eigenvalues.
   In a follow-up paper~\cite{GolSchVod16a} we develop the method further
  to show that for potentials given by large generic trigonometric polynomials the spectrum consists of a single interval,
  as conjectured by Chulaevski and Sinai~\cite{ChuSin89}.
\end{abstract}

\tableofcontents


\section{Introduction}\label{Sec:intro}

In their pioneering paper \cite{ChuSin89}, Chulaevsky and Sinai analyzed the  spectrum and localized eigenfunctions of Schr\"odinger operators on $\Z$
with a large quasi-periodic potential
given by evaluating a generic smooth function on $\tor^{2}$ along the orbit of an ergodic shift.
They conjectured that in contrast to the shift on the one-dimensional torus $\tor$, for which the spectrum is typically a Cantor set,
 for the two-dimensional shift the spectrum can be an interval. 

\smallskip

In the 25 years since \cite{ChuSin89},
the theory of quasi-periodic operators has been  developed extensively, see 
\cite{Bou05} and \cite{JitMar16}.
Techniques from complex and harmonic analysis, the theory of semi-algebraic sets,
and the theory of quasi-periodic cocyles played a key role in these developments.
For the one-dimensional shift the spectral theory of quasi-periodic Schr\"odinger co-cycles is almost
complete. Most of the results have been established {\em non-perturbatively}, i.e., either in the regime
of almost reducibility or in the regime of positive Lyapunov exponent, and
Avila's global theory, see \cite{Avi15}, gives a qualitative spectral picture, covering both regimes, for generic
potentials.
In the regime of positive Lyapunov exponent with generic
frequency the spectrum is a Cantor set which is Carleson-homogeneous, see \cite{GolSch11} and
\cite{DamGolSchVod15}. In this regime finite-volume exponential localization holds outside a set in phase space
which is  exponentially small in terms of the size of the interval. Moreover, one has a quantitative control
on the repulsion between all eigenvalues, see \cite{GolSch08}. For the case of the almost Mathieu operator
(corresponding to a cosine potential), both 
 localization and the spectrum have been studied in great detail under arithmetic conditions
on the frequency, see  Jitomirskaya~\cite{Jit99} for localization, and Puig~\cite{Pui04} and
Avila, Jitmoriskaya~\cite{AviJit09,AviJit10} for the Cantor structure of the spectrum.

\smallskip

The spectral theory of Schr\"odinger co-cycles over  multidimensional shifts  on $\tor^{d}$ turns out to be more intricate to analyze.
For $d=2$ exponential  localization was established in \cite{BouGol00} almost everywhere 
 in phase space.
For $d>2$ the same result was developed
in~\cite{Bou05}. However, these results were established directly in infinite volume, whereas localization in finite volume remained unknown. Regarding the structure of the spectrum, we note that Bourgain \cite{Bou02a} proved
existence of gaps for small potential and atypical frequencies. Numerical investigations by Puig, Sim\'o
\cite{PuiSim11} indicate that the spectral set may range from a Cantor
set, to a finite union of intervals, to one interval, depending on the largeness of the potential, which can be tuned
through a coupling constant.

\smallskip

In this paper we develop some basic features of the multi-frequency model, which are needed in the resolution of the Chulaevski-Sinai
conjecture. Heuristically speaking, gaps in the spectrum of the one-frequency operators  are created by resonances between an eigenvalue of one scale, and another shifted 
 eigenvalue of the same scale, see~\cite{Sin87,GolSch11}. In contrast to this, the heuristic principle underlying \cite{ChuSin89}  is that the graph of an eigenvalue
 on finite volume parametrized by the phase is too large to be destroyed along an entire horizontal section by the ``forbidden zones'' created by resonances. 
It is clear that some genericity assumption on the potential function is needed for this to be true, since $V(x,y)=v(x)$ has Cantor spectrum.
Implementing such an argument, however, appears to be very challenging for a number of reasons.
First, long chains of resonances   might occur rendering the
eigenvalue parametrization hard to handle. Second, the analytical techniques available in finite volume
are less favorable (mainly the large deviation theorems and everything that depends on them, see below) as compared to the case of one frequency.

\smallskip

Nevertheless,   this paper obtains  precisely finite volume localization and level repulsion estimates as in~\cite{GolSch08,GolSch11}
 for the case of multiple frequencies. This is made possible by introducing a new device into the analysis, namely Bourgain's semi-algebraic elimination technique from~\cite{Bou07}.
 We emphasize that we do not assume large disorder, so our arguments are nonperturbative and rely only on positive Lyapunov exponents.
 In a subsequent paper~\cite{GolSchVod16a}  the results obtained here are used in a spectral analysis to show that indeed \textit{level sets of eigenvalues cannot be completely destroyed} for large generic potentials.

\smallskip

To give a more detailed explanation of the difficulties we encounter in the finite interval localization problem
for the case of several frequencies, we need to recall some
standard definitions and some basic results.

Let V be a real-valued analytic function on the $ d $-dimensional torus $\mathbb{T}^d$ ($ \tor=\R/\Z $). We consider
the family of discrete Schr\"odinger operators defined by
\begin{equation}
  \label{eq:schr}
  [H(x,\omega)\psi](n)= -\psi(n+1)-\psi(n-1) + V(x+n\omega)\psi(n)
\end{equation}
with the frequency vector $ \omega\in \T^d $ satisfying
the standard Diophantine condition
\begin{equation}\label{vecdiophant}
  \|k\cdot\omega \|\ge\frac{a}{|k|^{b}}\mbox{\ \ for all nonzero $k\in\Z^d$},
\end{equation}
where $ a>0 $, $ b>d $ are some constants (here $ \norm{\cdot} $ denotes the distance to the nearest integer and
$ |\cdot| $ stands for the sup-norm on $ \Z^d $). It is well known that for any $ b>d $, a.e.~$\omega\in \T^d $
satisfies~\eqref{vecdiophant} with some $ a=a(\omega) $. We denote by $\tor^d(a,b)$ the set of $\omega$ which
obey \eqref{vecdiophant}.
We denote by $H_{[a,b]}(x,\omega)$ the operator
on the finite interval $[a,b]$ with Dirichlet boundary conditions and by
\begin{equation*}
  f_{[a,b]}(x,\omega,E):= \det (H_{[a,b]}(x,\omega)-E)
\end{equation*}
the Dirichlet determinants.

\smallskip

The development of finite volume localization starts with {\em large deviation theorems} (LDT)
 of the form
\begin{equation}\label{eq:LDEIntro}
\mes \left\{ x\in\tor^d : |\log |f_{[1,N]}(x,\omega,E)|-NL(\omega,E)| > N^{1-\tau} \right\}
  < \exp(-N^{\sigma}),
\end{equation}
where $L(\omega,E)$ stands for the Lyapunov exponent, see~\cite{GolSch08}.
This applies for arbitrary $ d\ge 1 $.
For the one-dimensional torus $\tor$,
 the estimate was sharpened in~\cite{GolSch08} so that it implies a $(\log N)^A$--estimate
for the local number of zeros of $f_{[1,N]}(z,\omega,E)$ when $\omega$ and $E$ are fixed and
the phase $z$ runs over a  complex neighborhood of the
torus.
This level of precision   allows for a Weierstrass preparation theorem factorization of the
determinant $f_{[1,N]}(\cdot,\omega,E)$ with a polynomial factor of degree $(\log N)^A$.
This low   degree is crucial as it allows, in combination with other tools developed in
\cite{GolSch11}, for an effective control over the {\em double resonances} of the problem.
The latter refers to the phases $x$ for which the (LDT) estimate \eqref{eq:LDEIntro} fails
twice:  at $x$ and at $x+n\omega$ with $n$ in the range $N<|n|<N^A$, $A\gg 1$. In terms of exponentially localized eigenfunctions
at   scale  $N$, it means that after shifting one of them  on $\Z$ by  $n$, this pair will be very close to forming an eigenfunction of a larger scale, say $N^C$
with $C$ large. That is the meaning of {\em resonance} in an inductive scheme and it is a key feature    in \cite{GolSch11} and
\cite{DamGolSchVod15}.

\smallskip

It is well-known that for $d>1$ the deviations in the large deviations theorem are much larger than
for $d=1$, because the discrepancy of an orbit of length $N$ of a shift on $\tor^d$ cannot be reduced beyond some fixed power of $N$, whereas in one dimension
it can be made essentially logarithmic in $N$. In particular, for $ d>1 $, the (LDT) estimate \eqref{eq:LDEIntro} cannot be improved beyond
$N^{1/2}$-deviations. 
One of the main ideas here is to gain control over the {\em local deviations}
of $\log |f_{[1,N]}(x,\omega,E)|$ rather than insisting  on the whole torus $\tor^d$.
To do so, we rely on  Bourgain's~\cite{Bou07} analysis of the structure
of an arbitrary set $\cR\subset \mathbb{Z}$ of shifts of a given semi-algebraic set
$\cA\subset\tor^d\times\tor^d\times \mathbb{R}$ with
controlled size and complexity,
such that
\begin{equation}\label{eq:bourgainIntro1}
  \bigcap_{n\in\cR} \{x\in\tor^d:(x+n\omega,\omega,E)\in \cA\}\neq \emptyset.
\end{equation}
The phases $x$ in \eqref{eq:bourgainIntro1}
are called multiple resonances. Bourgain's result states that after removing a small set of $\omega$'s
{\em uniformly in} $x,E$ the set
$\cR$ in \eqref{eq:bourgainIntro1} has a {\em lacunary structure}.  This structure defines
for any given phase $x$ and scale $L$ a suitable size $L\le N\le L^A$ for which
we show that the local deviations
of $\log |f_{[1,N]}(x,\omega,E)|$ behave almost as well as in the one-dimensional case.

The aforementioned statement from~\cite{Bou07} is a combination of two techniques: (i) a purely semi-algebraic analysis
of the set $\tilde\cA:=\{(x,y)\in \tor^{2d}\::\: \cA(x)\cap \cA(y)\neq \emptyset\}$, stating that this is a small semi-algebraic set
(ii) the {\em method of steep planes} from \cite{BouGol00,Bou05,Bou07} which implies that it is unlikely that $(x,x+n\omega)\in \tilde\cA$ for fixed $x$; this means that the measure
of $\omega\in\tor^d$ for which this occurs is small provided $n$ is large enough.  

We note that, while Bourgain's technique applies uniformly to all phases and energies, the scales determined by it are  sensitive to the choice of these parameters. 
In and of itself, the measure estimate on resonant $\omega$'s obtained in this fashion is too weak to be useful for finite volume localization and the separation of the eigenvalues. 
Indeed, it cannot be summed over the range $N\le n\le 2N$,  since the measure bound is never better than $N^{-1}$ for a single choice of such $n$.
In contrast, the technique from \cite{GolSch11}, which is based on  resultants and the Cartan bound on large negative values of
subharmonic functions, gives subexponential bounds which can be summed. However, one needs
good control on the degree of the polynomials going into the resultant and
a nondegeneracy condition on the resultant, in order to avoid that the resultant
is close to $0$ everywhere on a disk. For $ d=1 $ these issues are addressed by the sharpening of the (LDT), whereas
for $ d>1 $ we have to take advantage of the flexibility of the methods based on semi-algebraic sets.
We gain the needed control on the degrees of the polynomials by working at
the scales afforded by the lacunary structure of the set of multiple resonances and we obtain the nondegeneracy
condition by the aforementioned semialgebraic-steep-planes elimination method.  The combination of all of these techniques in effect allows us to work with {\em three successive scales} of our inductive procedure.

\medskip

We now state  the main results of the paper.
Given an interval $ \Lambda\subset \Z $ we use $ E_j^\Lambda(x,\omega) $, $ \psi_j^\Lambda(x,\omega) $ to
denote the eigenpairs of
$ H_\Lambda(x,\omega) $, with $ \psi_j^\Lambda(x,\omega) $ a unit vector. When $ \Lambda=[1,N] $ we
use the shorter notation $ H_N(x,\omega) $, $ E^{(N)}_j(x,\omega) $, $ \psi_j^{(N)}(x,\omega) $.
\begin{thma}\label{thm:A}
  Let $ \epsilon\in(0,1/5) $, $ \gamma>0 $. There exists $ \sigma=\sigma(a,b) $ and
  a set $ \Omega_N $,
  \begin{equation*}
    \mes(\Omega_N)<\exp(-(\log N)^{\epsilon\sigma}),
  \end{equation*}
  such that for
  $ \omega\in \T^d(a,b)\setminus \Omega_N  $ there exists a set $ \cB_{N,\omega} $,
  \begin{equation*}
    \mes(\cB_{N,\omega})<\exp(-\exp((\log N)^{\epsilon\sigma}))
  \end{equation*}
  and the following holds.
  For any $ N\ge N_0(V,a,b,\gamma,\epsilon) $, $ \omega\in \T^d(a,b)\setminus \Omega_N $,
  $ x\in \T^d\setminus \cB_{N,\omega} $,
  and any eigenvalue $ E=E_j^{(N)}(x,\omega) $, such that $ L(\omega,E)>\gamma $,
  there exists an interval
  \begin{equation*}
    I=I (x, \omega,E)\subset [1,N],\quad  |I|<\exp((\log N)^{5\epsilon}),
  \end{equation*}
  such that
  \begin{equation}\label{eq:thm-a-ev-decay}
    \left| \psi_j^{(N)}(x,\omega;m) \right|< \exp\left(-\frac{\gamma}{4}\dist(m,I)\right),
  \end{equation}
  provided  $ \dist(m,I)> \exp((\log N)^{2\epsilon\sigma})$.
\end{thma}

\begin{thmb}\label{thm:B}
  Let $ \epsilon\in(0,1/5) $, $ \gamma>0 $, and $ \Omega_N $, $ \cB_{N,\omega} $ as in \cref{thm:A}.
  For any $ N\ge N_0(V,a,b,\gamma,\epsilon) $, $ \omega\in \T^d(a,b)\setminus \Omega_N $,
  $ x\in \T^d\setminus \cB_{N,\omega} $,
  and any eigenvalue $ E=E_j^{(N)}(x,\omega) $, such that $ L(\omega,E)>\gamma $,
  we have
  \begin{equation*}
    |E_k^{(N)}(x,\omega)-E_{j}^{(N)}(x,\omega)|>\exp(-C(V)|I|)
  \end{equation*}
  for any $ k\neq j $, with $ I=I(x,\omega,E) $ as in \cref{thm:A}.
\end{thmb}

Similarly to \cite{GolSch11}, finite scale localization and separation of eigenvalues allows us to give an effective,
quantitative, and detailed description of the spectrum on the whole lattice $ \Z $ in terms of the spectrum on
finite volume. This is achieved for eigenpairs in \cref{thm:C} and for spectral sets in \cref{thm:D}. This kind of
results are crucial for carrying out a spectral analysis along the lines of \cite{GolSch11}.
In what follows, when we take the norm of vectors like $ \psi_j^{(N)} $, we will always use the $ \ell^2 $ norm.
\begin{thmc}\label{thm:C}
  Let $ \epsilon\in (0,1/5) $, $ \gamma>0 $.
  With $ \Omega_N $, $ \cB_{N,\omega} $ as in \cref{thm:A}, applied on $ [-N,N] $ instead of $ [1,N] $,
  let $ \hat\Omega_{N_0}=\bigcup_{N\ge N_0}\Omega_{N} $,
  $ \hat\cB_{N_0,\omega}=\bigcup_{N\ge N_0} \cB_{N,\omega} $. The following statements hold for any
  $ N_0\ge C(V,a,b,\gamma,\epsilon) $, $ \omega\in \T^d(a,b)\setminus \hat\Omega_{N_0} $,
  $ x\in\T^d\setminus \hat \cB_{N_0,\omega} $.

  \noindent{\normalfont (a)} Let $ N_k=N^{2^k} $. If $ L(\omega,E_j^{[-N,N]}(x,\omega))>2\gamma $ and
  $ I=I(x,\omega,E_j^{[-N,N]}(x,\omega))\subset [-N/2,N/2] $, then for each $ k\ge 1 $ there exist $ j_k $
  such that
  \begin{gather}
    \label{eq:jk}
    \begin{gathered}
      \left| E_{j_k}^{[-N_k,N_k]}(x,\omega)-E_{j}^{[-N,N]}(x,\omega) \right|
      <\exp\left( -\frac{\gamma}{10}N \right), \\
      \norm{\psi_{j_k}^{[-N_k,N_k]}(x,\omega)-\psi_{j}^{[-N,N]}(x,\omega)}
      <\exp\left( -\frac{\gamma}{10}N \right),
    \end{gathered}\\
    \label{eq:psik}
    \left| \psi_{j_k}^{[-N_k,N_k]}(x,\omega;n) \right|<\exp\left( -\frac{\gamma}{20}\dist(n,I) \right),\
    3N/4\le |n|\le N_k.
  \end{gather}
  Furthermore, for any $ k'\ge k\ge 0 $,
  \begin{equation}\label{eq:jk'k}
    \begin{gathered}
      \left| E_{j_{k'}}^{[-N_{k'},N_{k'}]}(x,\omega)-E_{j_k}^{[-N_k,N_k]}(x,\omega) \right|
      <\exp\left( -\frac{\gamma}{10}N_k \right),\\
      \norm{\psi_{j_{k'}}^{[-N_{k'},N_{k'}]}(x,\omega)-\psi_{j_k}^{[-N_k,N_k]}(x,\omega)}
      <\exp\left( -\frac{\gamma}{10}N_k \right).
    \end{gathered}
  \end{equation}
  In particular, the limits
  \begin{equation*}
    E(x,\omega):=\lim_{k\to \infty} E_{j_k}^{[-N_k,N_k]}(x,\omega),\quad
    \psi(x,\omega;n):= \lim_{k\to \infty}\psi_{j_k}^{[-N_k,N_k]}(x,\omega;n),\ n\in \Z
  \end{equation*}
  exist, $ \norm{\psi}=1 $, and
  \begin{equation}\label{eq:psi}
    \left| \psi(x,\omega;n) \right|<\exp\left( -\frac{\gamma}{20}\dist(n,I) \right),\
    3N/4\le |n|.
  \end{equation}

  \noindent{\normalfont (b)} If $ L(\omega,E)>3\gamma $ for all $ E\in [E',E''] $, then
  $ P([E',E''])H(x,\omega)P([E',E'']) $
  has purely pure point spectrum and all the eigenpairs, with unit eigenvector,
  can be obtained from {\normalfont (a)}.
\end{thmc}

By the minimality of the rationally-independent shift on $ \T^d $ the spectrum of $ H(x,\omega) $
does not depend on $ x $. We denote it by $ \cS_{\omega} $. Let
$$ \cS_{N,\omega}=\bigcup_{x\in\T^d} \spec H_{[-N,N]}(x,\omega). $$ We would like to say that $ \cS_{N,\omega} $
approximates $ \cS_\omega $ as $ N\to \infty $. However, it turns out that there can exist large segments in
$ \cS_{N,\omega}\setminus \cS_\omega $ that persist as $ N\to \infty $. This segments correspond to eigenvalues
that are localized near the edges of $ [-N,N] $. We will argue that if we focus only on the eigenvalues localized
away from the edges (as in \cref{thm:C}), we do get a finite scale approximation of the spectrum.
To this end we replace
$ \cS_{N,\omega} $ by a set defined as follows.
Let $ N\ge 1 $, $ s>1 $, $ k_0\ge 0 $ integers, $ \rho\in \R^{k_0+1} $. We define
\begin{equation}\label{eq:respec-C}
  \fS_{N,\omega}(s,k_0,\rho)
  =\bigcup_{x\in\tor^d}
  \left( \bigcap_{0\le k\le k_0}(\spec H_{[-N^{(k)},N^{(k)}]}(x,\omega))^{(\rho_{k})} \right),
  \quad N^{(k)}=N^{s^k}.
\end{equation}
Note that given a set $ S $ we let
\begin{equation}\label{eq:rho-neighborhood}
  S^{(\rho)}= \{ x: \dist(x,S)<\rho \}.
\end{equation}
The motivation behind \cref{eq:respec-C} is the fact that any eigenvalue $ E_j^{[-N,N]}(x,\omega) $ that is
localized
away from the edges of $ [-N,N] $ is an approximate eigenvalue on any larger interval $ [-N',N'] $, namely
$ E_j^{[-N,N]}(x,\omega)\in (\spec H_{[-N',N']}(x,\omega))^{(\rho)} $, $ \rho=\exp(-cN) $. To fully
justify the use of \cref{eq:respec-C} we would also need to argue that any energy in
$  \fS_{N,\omega}(s,k_0,\rho) $ is
at least close to an eigenvalue localized away from the edges. This leads to the following type of problem.
If $ \cB\subset \T^d $ is a set with small measure, does $ E_j^{[-N,N]}(\cB,\omega) $ also have small measure? This
is true for $ d=1 $, but it is generally false for $ d>1 $. To work around this type of issues we take advantage
again of Bourgain's elimination of multiple resonances. As a trade-off we have to be quite particular about the
choice of parameters $ s,k_0,\rho $. So, for example we have to take $ k_0=2^{2d+1}-1 $ instead of the natural choice
$ k_0=1 $ (which works only for $ d=1 $). However, such technicalities have no effect on the applicability of
the result.

\begin{thmd}\label{thm:D}
  Let $ \gamma>0 $, $ A\ge 1 $.
  For any $ N\ge N_0(V,a,b,\gamma,A) $, $ s\ge s_0(a,b,A) $, there exists a
  set $ \Omega_N=\Omega_N(s) $, $ \mes(\Omega_N)<N^{-A} $, such that the following holds with
  \begin{gather*}
    \fS_{N,\omega}=\fS_{N,\omega}(s,k_0,\rho_N),\ k_0=2^{2d+1}-1,\\
    \rho_{N,0}=\exp(-N^{c(a,b)}),\ \rho_{N,k}=\exp\left( -\frac{\gamma}{10}N \right),k=1,\ldots,k_0.
  \end{gather*}
  If $ \omega\in \T^d\setminus \Omega_N $ and $ L(\omega,E)>\gamma $ for all $ E\in [E',E''] $,  then
  \begin{gather}
    \label{eq:S-SN}
    \cS_\omega\cap [E',E'']\subset \fS_{N,\omega},\\
    \label{eq:SN-S}
    \mes \left( (\fS_{N,\omega}\cap [E',E''])\setminus  \cS_\omega \right)<\exp\left( -\frac{\gamma}{20}N \right).
  \end{gather}
\end{thmd}

Finally, as an application we prove the local homogeneity of the spectrum at a supercritical energy. In particular,
this implies that if the Lyapunov exponent is positive for all energies, then the spectrum is homogeneous (in the
sense of Carleson; cf.\ \cite{DamGolSchVod15}). It would be interesting to establish such a result for all
Diophantine frequency vectors.
\begin{thme}\label{thm:E}
  Let $ \gamma>0 $. For any $ N\ge N_0(V,a,b,\gamma) $ there exists
  a set $ \Omega_N $, $ \mes(\Omega_N)<N^{-1}$, and a constant $ \delta_0=\delta_0(a,b,N) $ such that
  the following holds.
  If $ \omega\in \T^d(a,b)\setminus \Omega_N $,
  $ E\in \cS_\omega $, and $ L(\omega,E)>\gamma $, then
  \begin{equation*}
    \mes(\cS_\omega\cap (E-\delta,E+\delta))>\frac{\delta}{2},
  \end{equation*}
  for any $ \delta\in (0,\delta_0] $.
\end{thme}

\section{Basic Tools}\label{sec:basic-tools}

In this section we discuss the basic results we will employ throughout the paper. Many of them can be traced back
to \cite{GolSch01} and \cite{GolSch08}, and we will refer to these papers for some of the proofs. Some of the
results
were originally derived only for the single frequency case with a strong Diophantine condition, however they easily
extend to our more general context. For the convenience of the reader we will sketch the proofs of such results.

We start by recalling some basic information related to the transfer matrix and the Lyapunov exponent.
If $ \psi $ is a solution of the difference equation
\begin{equation}
\label{eq:hamiltonC}
-\psi(n+1) - \psi(n-1) + V(x+n\omega)\psi(n) = E\psi(n)\ ,\quad n \in \IZ\,
\end{equation}
then for any $ a<b $ we have
\begin{equation}\label{eq:transfer}
  \begin{bmatrix}
	\psi(b+1)\\ \psi(b)
  \end{bmatrix}
  = M_{[a,b]}(x,\omega,E) \begin{bmatrix}
	\psi(a) \\ \psi (a-1)
  \end{bmatrix}
\end{equation}
where the transfer matrix is given by
\begin{equation*}
  M_{[a,b]}(x,\omega,E)= \prod_{n=b}^a \begin{bmatrix}
	V(x+n\omega)-E & -1 \\
    1 & 0
  \end{bmatrix}.
\end{equation*}
The transfer matrix is related to the Dirichlet determinants
\begin{equation*}
  f_{[a,b]}(x,\omega,E):= \det (H_{[a,b]}(x,\omega)-E)
\end{equation*}
through the following  formula
\begin{equation}\label{eq:Mndetbasic}
  M_{[a,b]}(x,\omega,E) = \begin{bmatrix}
    f_{[a, b]} (x,\omega,E) & - f_{[a+1, b]} (x,\omega,E)\\
    f_{[a, b-1]} (x,\omega,E) & - f_{[a+1, b-1]}(x,\omega,E)
 \end{bmatrix}.
\end{equation}

We let $ M_N=M_{[1,N]} $ and
\begin{equation*}
  L_N(\omega,E)
  = \frac{1}{N} \int_{\tor^d} \log\|M_N(x,\omega,E)\|\,dx.
\end{equation*}
The sequence $ L_N $ is subadditive and the {\em Lyapunov exponent} is defined as
\begin{equation*}
  L(\omega,E)=\lim_{N\to\infty} L_N(\omega,E)= \inf_{N} L_N(\omega,E).
\end{equation*}
Note that
\begin{equation}\label{eq:monodr1}
  0\le \log \|M_N(x,\omega,E)\|\le C(V,|E|)N
\end{equation}
and therefore
\begin{equation}\label{eq:LEupperb1}
  0\le L_N(\omega,E)\le C(V,|E|).
\end{equation}
The lower bound in \cref{eq:monodr1} is due to the fact that $ \det M_N(x,\omega,E)=1 $.

It is known that $ V $ extends to be complex analytic on a set
\begin{equation*}
  \cA_\rho:=\{x+iy: x\in \tor^d,  y\in \R^d, |y|<\rho\},
\end{equation*}
with $ \rho=\rho(V) $. We use $ |\cdot| $ to denote the Euclidean norm.
Let
\begin{equation}\label{eq:lapyomega}
\begin{split}
 L_N(y,\omega,E) &= \frac{1}{N}\int_{\mathbb{T}^d}  \log \|M_N(x+iy,\omega,E)\|\,dx, \\
L(y,\omega,E) &= \lim_{N\to \infty} L_N(y,\omega,E).
\end{split}
\end{equation}
For $y=0$ we reserve the notation $L(\omega,E)$. Most of  the results in this section
{\em do not use} the fact that $V$ assumes only real values on the torus $\mathbb{T}^d$ and therefore, they
also hold on $ \T^d+iy $, $ |y|<\rho $, with $ L(y,\omega,E) $ instead of $ L(\omega,E) $. In particular, we care
that this applies to all the results up to and including the uniform upper estimates in \cref{sec:uniform}.

\begin{remark}
  We briefly comment on the use of constants. Unless stated otherwise, the constants denoted by
  $ C $ might have different values each time they are used.
  They will be allowed to depend on $ \gamma,\omega,V,E,d $. In most cases we leave the dependence on $ d $
  implicit.
  The dependence on $ V $ will be through $ \rho(V) $ and the sup norm of $ V $ on $ \cA_\rho $,
  $ \norm{V}_\infty $. The dependence on $ \omega $
  will be only on the parameters $ a,b $ of the Diophantine condition. Constants depending on $ E $ can be
  chosen uniformly for $ E $ in a bounded set. We let $ a\lesssim b $ denote $ a\le Cb $ with some positive $ C $
  and  and $ a\ll b $ denote $ a\le C^{-1} b $ with a sufficiently large positive $ C $. Finally, $ a\simeq b $
  stands for $ a\lesssim b $ and $ b\lesssim a $. It
  will be clear from the context what the implicit constants are
  allowed to depend on.
\end{remark}

\subsection{Large Deviations Estimates}

The following result, called Large Deviations Theorem (LDT) is
the most basic tool in the localization theory. We refer to \cite{BouGol00}, \cite{GolSch01}
for two different approaches to its proof.
\begin{theorem}\label{thm:LDT}
  Assume $ \omega\in \tor^d(a,b) $, $ E\in \C $.
  There exist $\sigma =\sigma(a,b)$, $\tau=\tau(a,b)$,
  $ \sigma,\tau\in (0,1) $,
  such that for $N\ge N_0(V,a,b,|E|)$ one has
  \begin{equation*}
    \mes \left\{ x\in\tor^d : |\log\|M_N(x,\omega,E)\|-NL_N(\omega,E)| > N^{1-\tau} \right\}
    < \exp(-N^{\sigma}).
  \end{equation*}
\end{theorem}

In \cite{GolSch08} it was shown (see \cite[Prop.~2.11]{GolSch08}) that in the the regime of positive
Lyapunov exponent, the large deviations estimate
extends to the entries of the transfer matrix.

\begin{theorem}\label{thm:DirLDT}
  Assume $ \omega\in \tor^d(a,b) $, $ E\in \C $, and $ L(\omega,E)> \gamma >0 $.
  There exist $\sigma =\sigma(a,b)$, $\tau=\tau(a,b)$,
  $ \sigma,\tau\in (0,1) $,
  such that for $N\ge N_0(V,a,b,|E|,\gamma)$ one has
  \begin{equation*}
    \mes \left\{ x\in\tor^d : |\log |f_N(x,\omega,E)|-NL_N(\omega,E)| > N^{1-\tau} \right\}
    < \exp(-N^{\sigma}).
  \end{equation*}
\end{theorem}

The constants $ \sigma,\tau $ in the (LDT) for determinants depend on the $ \sigma,\tau $ from the (LDT) for
the transfer matrix. Since the sharpness of these constants plays no role in our work, we choose them to be
the same (by making the constants in \cref{thm:LDT} smaller). We will refer to either of the deviations estimates
as (LDT).

\subsection{The Avalanche Principle}

The following statement, known as the Avalanche Principle (AP), is another
 basic tool in the theory
of quasi-periodic Schr\"{o}dinger operators.
\begin{prop}[{\cite[Prop.~2.2]{GolSch01}}]
\label{prop:AP}
Let $A_1,\ldots,A_n$ be a sequence of  $2\times 2$--matrices whose determinants satisfy
\begin{equation}
\label{eq:detsmall}
\max\limits_{1\le j\le n}|\det A_j|\le 1.
\end{equation}
Suppose that
\be
&&\min_{1\le j\le n}\|A_j\|\ge\mu>n\mbox{\ \ \ and}\label{large}\\
   &&\max_{1\le j<n}[\log\|A_{j+1}\|+\log\|A_j\|-\log\|A_{j+1}A_{j}\|]<\frac12\log\mu\label{diff}.
\ee
Then
\begin{equation}
\Bigl|\log\|A_n\ldots A_1\|+\sum_{j=2}^{n-1} \log\|A_j\|-\sum_{j=1}^{n-1}\log\|A_{j+1}A_{j}\|\Bigr|
< C\frac{n}{\mu}
\label{eq:AP}
\end{equation}
with some absolute constant $C$.
\end{prop}

\subsection{Estimates for the Lyapunov Exponent}

Using (LDT) and (AP) one gets the following estimate on the rate of
convergence for the finite scale Lyapunov exponent.
\begin{prop}[{\cite[Lem.~10.1]{GolSch01}}]\label{prop:uniform}
  Assume $ \omega\in \tor^d(a,b) $, $ E\in \C $, and
  $L(\omega,E)>\gamma>0$.  Then for any $ N\ge 2 $,
  \[ 0\le L_N(\omega,E)-L(\omega,E)< C\frac{(\log
      N)^{1/\sigma}}{N},\]
  where $C=C(V,a,b,|E|,\gamma)$ and $ \sigma $ is as in (LDT).
\end{prop}

We will need some estimates on the modulus of continuity of the Lyapunov
exponent. A first rough estimate can be obtained from the next lemma.

\begin{lemma}\label{lem:Eomdiff}
  Let $ N\ge 1 $, $ (z_i,w_i,E_i)\in \C^d\times \C^d\times \C $, $ i=1,2 $, such that
  \begin{equation*}
    |\Im z_i|,\ N|\Im w_i|<\rho(V).
  \end{equation*}
  Then
  \begin{multline}\label{eq:modcont1}
    \bigl\|M_N(z_1, w_1, E_1) - M_N(z_2, w_2,E_2)\bigr\| \\
    \le
    \Big(C(V)  +    |E_1| + |E_2|\Big)^N
    \left(|z_1-z_2|+|w_1 - w_2| +  |E_1 - E_2|\right).
  \end{multline}
  In particular we have
  \begin{multline}\label{eq:modcont111}
    \big|\log\bigl\|M_N(z_1, w_1, E_1)\bigr\| -
    \log\bigl\|M_N(z_2, w_2,E_2)\bigr\|\big| \\
    \le
    \Big(C(V)  + |E_1| + |E_2|\Big)^N \left(|z_1-z_2|+|w_1 - w_2| +
      |E_1 - E_2|\right),	
  \end{multline}
  provided the right-hand side is less than $ 1/2 $.
\end{lemma}
\begin{proof}
  Let
  \begin{equation*}
    A_{i,n}= \begin{bmatrix}
      V(z_i+nw_i)-E_i & -1 \\
      1 & 0
    \end{bmatrix},\ i=1,2.
  \end{equation*}
  Then
  \begin{multline}\label{eq:M_N1vs2}
    M_N(z_1,w_1,E_2)-M_N(z_2,w_2,E_2)=\sum_{n=1}^N  A_{2,N}\ldots A_{2,n+1}(A_{1,n}
    -A_{2,n})A_{1,n-1}\ldots A_{1,1}\\
    =\sum_{n=1}^N M_{N-n}(z_2,w_2,E_2)(A_{1,n}-A_{2,n})M_{n-1}(z_1,w_1,E_1).
  \end{multline}
  Now \cref{eq:modcont1} follows by using the Mean Value Theorem and the fact that we clearly have
  \begin{equation*}
    \norm{A_{i,n}}\le C(V)+|E_i|.
  \end{equation*}

  The second
  inequality follows from \cref{eq:modcont1} and the fact that
  $ |\log x|\le 2|x-1| $, provided $ |x-1|\le 1/2 $. Indeed, we have
  \begin{multline*}
	\big|\log\bigl\|M_N(z_1, w_1, E_1)\bigr\| -
    \log\bigl\|M_N(z_2, w_2,E_2)\bigr\|\big|
    \lesssim \left| \frac{\norm{M_N(z_1,w_1,E_1)}}{\norm{M_N(z_2, w_2,E_2)}}-1 \right|\\
    \le\frac{\norm{M_N(z_1, w_1,E_1)-M_N(z_2, w_2,E_2)}}{\norm{M_N(z_2, w_2,E_2)}}
  \end{multline*}
  provided the right-hand side is less than $ 1/2 $. The conclusion follows by recalling
  that
  \begin{equation*}
    \norm{M_N(z,w,E)}\ge 1  \text{ for any }  z,w,E.
  \end{equation*}
\end{proof}

We refine the result of the previous lemma using (AP) and then we
deduce our estimate on the modulus of continuity for
$ L_N(\omega,E) $.

\begin{lemma}\label{lem:compL}
  Assume $ \omega_0\in \tor^d(a,b) $, $ E_0\in \C $, and
  $L(\omega_0,E_0)>\gamma>0$. Let $ \sigma $ be as in (LDT) and let $A$
  be a constant such that $ \sigma A\ge 1 $.  Then for all
  $N\ge N_0(V,a,b,|E_0|,\gamma,A)$,
  \begin{equation}\label{eq:3.35}
    \left | \log \|M_N(z, w,E)\|
      -\log \|M_N(x_0, \omega_0,E_0)\|\right|
    < \exp \left(-(\log N)^A\right)
  \end{equation}
  for any $x_0 \in\tor^d\setminus \cB_{N,\omega_0, E_0}$,
  $\mes (\cB_{N,\omega_0, E_0}) < \exp\left(-(\log N)^{\sigma A}\right)$,
  and $ (z,w,E)\in \C^d\times \C^d\times \C  $ such that
  \begin{equation}\label{eq:compL}
    |z-x_0|,\ |w - \omega_0|,\ |E - E_0| < \exp\left(-(\log N)^{4A}\right).
  \end{equation}
\end{lemma}
\begin{proof}
  Let $\ell\simeq (\log N)^{A+1}$, $n=[N/\ell]$,
  $$A_j=A_j(x,\omega,E)=M_\ell(x+(j-1)\ell\omega,E).$$
  To simplify notation we assume that $
  N=n\ell
  $.  It is easy to see how to adjust the argument for general $N$.

  Using (LDT) and Proposition~\ref{prop:uniform} we get that
  \begin{equation}\label{eq:APcondA1}
    \min_{1\le j\le n}\|A_j(x_0,\omega_0,E_0)\|
    \ge \exp(\ell L_\ell(\omega_0,E_0)-\ell^{1-\tau})\ge \exp(\ell \gamma/2)>n,
  \end{equation}
  \begin{multline*}
    \max_{0\le
      j<n}\Bigl[\log\|A_{j+1}(x_0,\omega_0,E_0)\|+\log\|A_j(x_0,\omega_0,E_0)\|\\
    -\log\|A_{j+1}(x_0,\omega_0,E_0)A_{j}(x_0,\omega_0,E_0)\|\Bigr] \le
    \frac{\gamma}{16}\ell,
  \end{multline*}
  for all $  x_0 $ outside a set $\cB_{N,\omega_0,E_0} \subset \tor^d$ with
  \begin{equation*}
    \mes\left(\cB_{N,\omega_0,E_0}\right) \lesssim n \exp(-\ell^\sigma)< \exp(-(\log N)^{\sigma A})
  \end{equation*}
  (the last inequality holds provided that $
  \sigma A\ge 1 $). Take $x_0\in  \tor^d\setminus \cB_{N,E_0,\omega_0}$. We apply (AP) with
  $\mu=\exp(\ell\gamma/2)$ to get
  \begin{multline}\label{eq:apome}
    \Bigl|\log\|M_N(x_0,\omega_0,E_0)\|+\sum_{j=2}^{n-1}\log\|A_{j}(x_0,\omega_0,E_0)\|\\
    -\sum_{j=1}^{n-1}\log\|A_{j+1}(x_0,\omega_0,E_0)A_{j}(x_0,\omega_0,E_0)\|\Bigr|<
    Cn\exp(-\gamma\ell /2).
  \end{multline}
  Take $ z,w,E $ satisfying \eqref{eq:compL}. Using Lemma~\ref{lem:Eomdiff} it
  follows that
  \begin{equation}\label{eq:APcondA1-0}
    \min_{1\le j\le n}\|A_j(z,w,E)\|
    \ge \exp(\ell \gamma/4)>n,
  \end{equation}
  \begin{multline*}
    \max_{0\le
      j<n}\Bigl[\log\|A_{j+1}(z,w,E)\|+\log\|A_j(z,w,E)\|\\
    -\log\|A_{j+1}(z,w,E)A_{j}(z,w,E)\|\Bigr] \le
    \frac{\gamma}{8}\ell.
  \end{multline*}
  We  apply (AP) again, this time with $  \mu=\exp(\gamma\ell/4) $, to get
  \begin{multline}\label{eq:apome1}
    \Bigl|\log\|M_N(z,w,E)\|+\sum_{j=2}^{n-1}\log\|A_{j}(z,w,E)\|\\
    -\sum_{j=1}^{n-1}\log\|A_{j+1}(z,w,E)A_{j}(z,w,E)\|\Bigr|
    <Cn\exp(-\gamma\ell /4).
  \end{multline}
  Subtracting \eqref{eq:apome} from \eqref{eq:apome1}
  and applying Lemma~\ref{lem:Eomdiff}, term-wise, yields
  \eqref{eq:3.35}.
\end{proof}

\begin{prop}\label{prop:L_comp}
  Assume $ \omega_0\in \tor^d(a,b) $, $ E_0\in \C $, and
  $L(\omega_0,E_0)>\gamma>0$. Let $ \sigma $ be as in (LDT) and let $A$
  be a constant such that $ \sigma A\ge 1 $.  Then for all
  $N\ge N_0(V,a,b,|E_0|,\gamma,A)$,
  \begin{equation*}
    \bigl| L_N(\omega,E) - L_N(\omega_0, E_0)\bigr | < \exp \left(-(\log N)^{\sigma A}\right)
  \end{equation*}
  provided
  $|\omega - \omega_0|,\ |E - E_0| \le \exp \left(-(\log  N)^{4A}\right)$.
\end{prop}
\begin{proof}
  We have
  \begin{equation*}
    \bigl| L_N(\omega,E) - L_N(\omega_0, E_0)\bigr |
    \le \frac{1}{N}\int_{\tor^d}|\log\norm{M_N(x,\omega,E)}-\log\norm{M_N(x,\omega_0,E_0)}|\,dx
  \end{equation*}
  and the conclusion follows from \cref{lem:compL}.
\end{proof}

As a consequence we obtain the log-H\"older continuity of the Lyapunov exponent.

\begin{prop}\label{prop:log-Hoelder}
  Assume $ \omega_0\in \tor^d(a,b) $, $ E_0\in \C $, and  $L(\omega_0,E_0)>\gamma>0$. There exists
  $ \epsilon_0=\epsilon_0(V,a,b,|E_0|,\gamma) $ such that if $ |E-E_0|<\epsilon_0 $,
  then  $ L(\omega_0,E)> \gamma/2 $ and
  \begin{equation*}
    |L(\omega_0,E)-L(\omega_0,E_0)|\le
    \exp \left( \frac{1}{2} \left(-\log |E-E_0| \right)^{\sigma/4} \right),
  \end{equation*}
  with $ \sigma $ as in (LDT). Furthermore, if $ \omega\in \T^d(a,b)\cap(\omega_0-\epsilon_0,\omega_0+\epsilon_0) $,
  then $ L(\omega,E_0)>\gamma/2 $ and
  \begin{equation*}
    |L(\omega,E_0)-L(\omega_0,E_0)|\le
    \exp \left( \frac{1}{2} \left(-\log |\omega-\omega_0| \right)^{\sigma/4} \right).
  \end{equation*}
\end{prop}
\begin{proof}
  From \cref{prop:uniform} and \cref{prop:L_comp} we have that
  \begin{multline*}
    |L(\omega_0,E)-L(\omega_0,E_0)|< C\frac{(\log N)^{1/\sigma}}{N}\le \exp(-\frac{1}{2}\log (N+1))\\
    \le \exp \left( \frac{1}{2} \left(-\log |E-E_0| \right)^{\sigma/4} \right)
    < \frac{\gamma}{2},
  \end{multline*}
  provided
  \begin{equation*}
    \exp(-(\log (N+1))^{4/\sigma})\le |E-E_0|\le\exp(-(\log N)^{4/\sigma})
  \end{equation*}
  and $ N\ge N_0(V,a,b,|E_0|,\gamma) $. The first statement follows by setting
  $$ \epsilon_0=\exp(-(\log N_0)^{4/\sigma}). $$ The second statement follows in the same way (note that we need
  $ \omega\in \T^d(a,b) $ to use \cref{prop:uniform}).
\end{proof}

\begin{remark}
  The above result is essentially proved in \cite[Prop. 10.2]{GolSch01}, however the existence of the interval
  $ (E_0-\epsilon_0,E_0+\epsilon_0) $ is not covered explicitly there.
\end{remark}

We are also interested in the modulus of continuity of
$ L_N(y,\omega,E) $ with respect to $ y $. We will show that $ L_N $
is in fact Lipschitz in $ y $. The proof will be based on the
following fact.

\begin{lemma}[{\cite[Lem.~4.1]{GolSch08}}]
  \label{lem:liprad}
  Let $1>\rho>0$ and suppose $u$ is subharmonic on
  \[
    \cA_\rho:=\{x+iy: x\in \tor, |y|<\rho\},
  \]
  such that $\sup_{ \cA_\rho} u\le 1$ and $\int_{\tor}
  u(x)\,dx\ge0$.
  Then for any $y,y'$ so that
  $-\frac{\rho}{2} < y,y' < \frac{\rho}{2}$ one has
  \begin{equation*}
    \left| \int_\tor u(x+iy)\,dx-\int_\tor u(x+iy')\,dx \right|\le C_{\rho}|y-y'|.
  \end{equation*}
\end{lemma}

\begin{cor}\label{cor:liplap}
  Let $ \omega\in \T^d $, $ E\in \C $. There exists  $C=C(V,|E|)$ such that
  \[
    |L_N(y,\omega,E) - L_N(\omega,E)| \le C\sum_{i=1}^d |y_i| \text{\ \ \ for all\ \
      \ }|y| < \rho(V),
  \]
  uniformly in $N$. In particular, the same bound holds with $ L $ instead of $ L_N $.
\end{cor}
\begin{proof}
  If $d=1$, then the statement follows directly from
  Lemma~\ref{lem:liprad} applied to  $ u(z)= \frac{1}{CN}\log \norm{M_N(z,\omega,E)} $ (with $ C $
  as in \cref{eq:monodr1}).  Let us verify the
  statement for $d=2$.
  Let
  \begin{equation*}
    v_{y_2}(x_1+iy_1):=\int_\T \frac{1}{CN}\log\norm{M_N(x_1+iy_1,x_2+iy_2,\omega,E)}\,dx_2.
  \end{equation*}
  By \cref{lem:liprad}
  \begin{equation*}
    |v_{y_2}(x_1+iy_1)-v_{0}(x_1+iy_1)|\le C|y_2|.
  \end{equation*}
  Since $ v_{y_2} $ is subharmonic for any fixed $ y_2 $, \cref{lem:liprad} also implies
  \begin{equation*}
    \left| \int_\tor v_{y_2}(x_1+iy_1)\,dx_1-\int_\tor v_{y_2}(x_1)\,dx_1 \right|\le C|y_1|.
  \end{equation*}
  So we have
  \begin{multline*}
    |L_N(\omega,E)-L_N(y,\omega,E)|
    = \left| \int_\tor v_0(x_1)\,dx_1-\int_\tor v_{y_2}(x_1+iy_1)\,dx_1 \right|\\
    \le \left| \int_\tor (v_0(x_1)-v_0(x_1+iy_1))\,dx_1 \right|
    + \left| \int_\tor (v_0(x_1+iy_1)-v_{y_2}(x_1+iy_1))\,dx_1\right| \\
    \le C(|y_1|+|y_2|).
  \end{multline*}
  For general $d$ the  proof is completely similar with help of induction over $d$.
\end{proof}

\subsection{Uniform Upper Estimates}\label{sec:uniform}

The upper bound from \cref{eq:monodr1} can be improved by invoking
 the sub-mean value property for subharmonic functions.

\begin{prop}\label{prop:logupper}
  Assume $ \omega\in \tor^d(a,b) $, $ E\in \C $, and $ L(\omega,E)> \gamma >0 $. Then for all $N\ge 1$,
  \begin{equation}\label{eq:uniform-estimate}
    \sup\limits_{x\in\tor^d} \log \| M_N (x,\omega,E) \| \le
    NL_N(\omega,E)+ C N^{1-\tau} \ ,
  \end{equation}
  with $C=C(V,a,b,|E|,\gamma)$ and $ \tau $ as in (LDT).
\end{prop}
\begin{proof}
  We only check \cref{eq:uniform-estimate} for $ N $ large enough, for smaller $ N $ we simply choose
  $ C $ large enough.
  By (LDT), 
  \begin{equation}\label{eq:LDEAP1}
    \mes \left\{ x\in\tor^d :
      |\log\|M_N(x+iy,\omega,E)\|-NL_N(y,\omega,E)| > N^{1-\tau} \right\}<
    \exp(-N^{\sigma}),
  \end{equation}
  and using  Corollary~\ref{cor:liplap} we have
  \begin{equation}\label{eq:LDEAP12}
    \mes \left\{ x\in\tor^d :
      |\log\|M_N(x+iy,\omega,E)\|-NL_N(\omega,E)| > 2N^{1-\tau} \right\}<
    \exp(-N^{\sigma}),
  \end{equation}
  provided $ |y|\le 1/N $.
  Due to the sub-mean value property for subharmonic functions we have
  \begin{equation}\label{eq:submeanMN}
    \log\|M_N(x,\omega,E)\| \le (\pi r^2)^{-d}\int_{\cP}\log\|M_N(\xi+iy,\omega,E)\|\,d\xi dy,
  \end{equation}
  where $\cP=\prod\cD(x_{j},r)$, $x=(x_1,\dots,x_d)$,  $r=N^{-1}$.
  Denote by $\cB_y\subset \mathbb{T}^d$ the set in \eqref{eq:LDEAP12}.  Let
  \[
    \cB=\{(\xi,y)\in [0,1]^d\times (-r,r)^d:\xi\in\cB_y\}.
  \]
  Due to \eqref{eq:LDEAP12} we have
  \begin{equation}\label{eq:submeanMN1}
    \begin{split}
      (\pi
      r^2)^{-d}\int_{\cP\setminus\cB}\log\|M_N(\xi+iy,\omega,E)\|\,d\xi
      dy\le NL_N(\omega,E)+2N^{1-\tau}.
    \end{split}
  \end{equation}
  On the other hand, due to   \eqref{eq:monodr1}
  \begin{equation}\label{eq:submeanMN2}
    \begin{split}
      (\pi r^2)^{-d}\int_{\cP\cap\cB}\log\|M_N(\xi+iy,\omega,E)\|\,d\xi dy\le
      (\pi r^2)^{-d}CN\mes (\cB)<\exp(-N^{\sigma}/2).
    \end{split}
  \end{equation}
  The conclusion follows by combining \eqref{eq:submeanMN}, \eqref{eq:submeanMN1}, and
  \eqref{eq:submeanMN2}.
\end{proof}

For the purpose of Cartan's estimate (see \cref{sec:Cartan}) we need  to extend the uniform estimate
to complex neighborhoods of $ (x,\omega,E) $. It is crucial that the size of the neighborhood in the next
result is much larger than what can be obtained by simply applying \cref{lem:Eomdiff}.
\begin{cor}\label{cor:logupperc}
  Assume $ \omega_0\in \tor^d(a,b) $, $ E_0\in \C$, and $ L(\omega_0,E_0)> \gamma >0 $.
  Let $ \sigma,\tau $ as in (LDT).
  For all $ N\ge N_0(V,a,b,|E_0|,\gamma) $ and
  $ (y,w,E)\in \R^d\times \C^d\times \C $ such that
  \begin{equation*}
    |y|<1/N,\ |w-\omega_0|,\ |E-E_0|< \exp(-(\log N)^{8/\sigma}),
  \end{equation*}
  we have
  \begin{equation}\nn
    \sup\limits_{x\in\tor^d} \log \| M_N (x+iy,w,E) \| \le
    NL_N(\omega_0,E_0)+ C N^{1-\tau},
  \end{equation}
  with $C=C(V,a,b,|E_0|,\gamma)$.  In
  particular, we also have
  \begin{equation}\nn
    \sup\limits_{x\in\tor^d} \log | f_N (x+iy,w,E)| \le
    NL_N(\omega_0,E_0)+ C N^{1-\tau}.
  \end{equation}
\end{cor}
\begin{proof}
  Take $ y,w,E $ satisfying the assumptions.
  Due to \cref{prop:logupper} and \cref{lem:compL} (with $ A=2/\sigma $; both results are applied on
  $ \T^d+iy $), we have that
  \begin{equation*}
    \log\norm{M_N(x+iy,w,E)}\le NL_N(y,\omega_0,E_0)+CN^{1-\tau}
  \end{equation*}
  for any $ x\in \T^d\setminus \cB_{y} $, $ \mes(\cB_y)<\exp(-(\log N)^2) $.
  Now we can employ the sub-mean value property for subharmonic functions just as in the proof of
  \cref{prop:logupper} and the conclusion follows.
\end{proof}

Another consequence of the uniform upper estimate is an improvement of the stability estimate from
\cref{lem:Eomdiff}.
\begin{cor}\label{cor:4.6}
  Assume $ \omega_0\in \tor^d(a,b) $, $ E_0\in \C$, and $ L(\omega_0,E_0)> \gamma >0 $.
  Let $ \sigma,\tau $ as in (LDT).
  For all $ N\ge N_0(V,a,b,|E_0|,\gamma) $ and
  $ (z_i,w_i,E_i)\in \C^d\times \C^d\times \C $ such that
  \begin{equation*}
    |\Im z_i|<1/N,\ |w_i-\omega_0|,\ |E_i-E_0|< \exp(-(\log N)^{8/\sigma}),\ i=1,2
  \end{equation*}
  we have
  \begin{multline*}
	\left\| M_N \left(z_1, w_1,E_1\right) - M_N \left( z_2,w_2,E_2\right) \right\|\\
    \le  \left(|z_1-z_2| + |w_1-w_2|+|E_1 - E_2| \right)
    \exp \left(NL \left(\omega_0, E_0\right) + C N^{1-\tau}\right),
  \end{multline*}
  with $C=C(V,a,b,|E_0|,\gamma)$.
  In  particular,
  \begin{multline}\label{eq:modcont222}
    \big|\log\bigl\|M_N(z_1, w_1, E_1)\bigr\| -
    \log\bigl\|M_N(z_2,w_2,E_2)\bigr\|\big| \\
    \le\left(|z_1 - z_2|+|w_1-w_2|+|E_1-E_2|\right)
     \exp\left(NL(\omega_0, E_0\right) + C N^{1-\tau}\bigr),	
  \end{multline}
  \begin{multline} \label{eq:fN_comp}
    \left| \log | f_N(z_1,w_1,E_1)|
      -\log | f_N(z_2, w_2,E_2)| \right|\\
    \le
    \left(|z_1-z_2|+|w_1-w_2|+|E_1-E_2|\right)
    {\exp\left(NL(\omega_0, E_0\right) + C N^{1-\tau}\bigr)
      \over \max_i| f_N(z_i,w_i,E_i)|},
  \end{multline}
  provided the right-hand sides of~\eqref{eq:modcont222} and~\eqref{eq:fN_comp} are less than
  $1/2$.
\end{cor}
\begin{proof}
  The proof is completely analogous to that of \cref{lem:Eomdiff}. The only difference is that now we can
  use \cref{cor:logupperc} to bound the $ M_{N-n}, M_{n-1} $ factors in \cref{eq:M_N1vs2}.
\end{proof}

In a similar manner one can treat the stability of $ M_N,f_N $ under
a change of the potential. Let $ \tilde M_N $ and $ \tilde f_N $
denote the transfer matrices and the Dirichlet determinants associated
to the operator having $ \tilde V $ as a potential instead of $ V $.

\begin{lemma}\label{lem:stability-in-V}
  Assume $ \omega\in \tor^d(a,b) $, $ E\in \C$, and $ L(\omega,E)> \gamma >0 $.
  Then for any $ x\in \T^d $ and  $ N\ge 1 $ we have
  \begin{equation*}
    \left\| M_N \left(x, \omega,E\right) - \tilde M_N \left(x,\omega,E\right) \right\|
    \le \norm{V-\tilde V}_\infty
    \exp \left(NL \left(\omega, E\right) + C N^{1-\tau}\right),
  \end{equation*}
  with $ C=C(V,a,b,|E|,\gamma) $.
  In particular,
  \begin{equation*}
    \left| \log \| M_N(x,\omega,E)\|
      -\log \|\tilde M_N(x,\omega,E)\| \right|\le
    \norm{V-\tilde V}_\infty
    \exp\left(NL(\omega, E\right) + C  N^{1-\tau}\bigr),
  \end{equation*}
  \begin{equation*}
    \left| \log | f_N(x,\omega,E)|
      -\log |\tilde f_N(x,\omega,E)| \right|\le
    \norm{V-\tilde V}_\infty
    {\exp\left(NL(\omega, E\right) + C
      N^{1-\tau}\bigr)\over \max \left(\bigl| f_N\bigl(x,\omega,E\bigr)\bigr|,
        \bigl| \tilde f_N\bigl(x,\omega,E\bigr)\bigr|\right)}\ ,
  \end{equation*}
  provided the right-hand sides  are less than  $1/2$.
\end{lemma}

\subsection{Cartan's Estimate}\label{sec:Cartan}

We adopt the definition of Cartan sets from \cite{GolSch08}.
\begin{defi}\label{defi:cartansets} Let $H \gg 1$.  For an arbitrary subset $\cB \subset \cD(z_0,
1)\subset \IC$ we say that $\cB \in \car_1(H, K)$ if $\cB\subset
\bigcup\limits^{j_0}_{j=1} \cD(z_j, r_j)$ with $j_0 \le K$, and
\begin{equation}
\sum_j\, r_j < e^{-H}\ .
\end{equation}
If $d$ is a positive integer greater than one and $\cB \subset
\prod\limits_{i=1}^d \cD(z_{i,0}, 1)\subset \IC^d$ then we define
inductively that $\cB\in \car_d(H, K)$ if for any $1 \le j \le d$ there
exists $\cB_j \subset \cD(z_{j,0}, 1)\subset \IC, \cB_j \in \car_1(H,
K)$ so that $\cB_z^{(j)} \in \car_{d-1}(H, K)$ for any $z \in \IC
\setminus \cB_j$,  here $\cB_z^{(j)} = \left\{(z_1, \dots, z_d) \in \cB:
z_j = z\right\}$.
\end{defi}

The above definition appears naturally from the proof of the following
generalization of the usual Cartan estimate (see \cite[Lecture 11]{Lev96}) to several variables.

\begin{lemma}[{\cite[Lem.~2.15]{GolSch08}}]\label{lem:high_cart}
 Let $\varphi(z_1, \dots, z_d)$ be an analytic function defined
in a polydisk $\cP = \prod\limits^d_{j=1} \cD(z_{j,0}, 1)$, $z_{j,0} \in
\IC$.  Let $M \ge \sup\limits_{\uz\in\cP} \log |\varphi(\uz)|$,  $m \le \log
\bigl |\varphi(\uz_0)\bigr |$, $\uz_0 = (z_{1,0},\dots, z_{d,0})$.  Given $H
\gg 1$ there exists a set $\cB \subset \cP$,  $\cB \in
\car_d\left(H^{1/d}, K\right)$, $K = C_d H(M - m)$,  such that
\begin{equation}\label{eq:cart_bd}
  \log \bigl | \varphi(z)\bigr | > M-C_d H(M-m)
\end{equation}
for any $z \in \prod^d_{j=1} \cD(z_{j,0}, 1/6)\setminus \cB$.
Furthermore, when $ d=1 $ we can take
$ K=C(M-m) $ and keep only the disks of $ \cB $ containing a zero of $\phi$ in them.
\end{lemma}

We note that the definition of the Cartan sets gives implicit information
about their measure.
\begin{lemma}\label{lem:Cartan-measure}
  If  $ \cB\in \Car_d(H,K) $ then
  \begin{equation*}
    \mes_{\C^d}(\cB)\le C(d) e^{-H} \text{ and } \mes_{\R^d}(\cB\cap \R^d)\le C(d) e^{-H}.
  \end{equation*}
\end{lemma}
\begin{proof}
  The case $ d=1 $ follows immediately from the definition of
  $ \Car_1 $. The case $ d>1 $ follows by induction, using Fubini and
  the definition of $ \Car_d $.
\end{proof}

Cartan's estimate allows us to argue that if the (LDT) estimate fails for $ f_N(x,\omega,E) $, then $ x $ and $ E $
are close to zeros of $ f_N $.
\begin{lemma}\label{lem:LDT-fail}
  Assume $ x\in \T^d $, $ \omega\in \tor^d(a,b) $, $ E\in \C $, and $ L(\omega,E)> \gamma >0 $. Let
  $ H\gg 1 $ and $ \tau,\sigma $ as in (LDT).
  There exists $ C_0=C_0(V,a,b,|E|,\gamma) $, such that for any $ N\ge N_0(V,a,b,|E|,\gamma) $, if
  \begin{equation*}
    \log|f_N(x,\omega,E)|\le NL_N(\omega,E)-C_0 H N^{1-\tau},
  \end{equation*}
  then there exists $ z\in \C^d $,
  $ |z-x|\lesssim \exp(-(H+N^\sigma/d)) $ such that $ f_N(z,\omega,E)=0 $. Furthermore,
  \begin{equation*}
    \norm{(H_N(x,\omega)-E)^{-1}}\ge c(V) \exp(H+N^\sigma/d).
  \end{equation*}
\end{lemma}
\begin{proof}
  By (LDT) there exists $ x_0 $, $ |x-x_0|\lesssim \exp(-N^{\sigma}/d) $ such that
  \begin{equation*}
    \log|f_N(x_0,\omega,E)|> NL_N(\omega,E)-N^{1-\tau}.
  \end{equation*}
  Let
  \begin{equation*}
    \phi(\zeta)=f_N \left( x_0+\frac{ C \exp(-N^{\sigma}/d)\zeta}{|x-x_0|}(x-x_0),\omega,E  \right).
  \end{equation*}
  Let $ \zeta_x $ be such  that $ \phi(\zeta_x)=f_N(x,\omega,E) $. Our choice of scaling is such that
  $ \zeta_x\in \cD(0,1/7) $.
  Using the uniform upper estimate from \cref{cor:logupperc}, we can apply Cartan's estimate to get
  \begin{equation}\label{eq:Cartan}
    \log|\phi(\zeta)|> NL_N(\omega,E)-CHN^{1-\tau}
  \end{equation}
  for all $ \zeta\in \cD(0,1/6)\setminus \cB $, with $ \cB\in \car_1(H,CHN^{1-\tau}) $. It follows that
  $ \zeta_x\in \cD(\zeta_j,r_j)\subset \cB $, with $ r_j<\exp(-H) $. By the second part of \cref{lem:high_cart},
  there exists $ \zeta'\in \cD(\zeta_j,r_j) $ such that $ \phi(\zeta')=0 $. The first statement follows with
  \begin{equation*}
    z=x_0+\frac{C \exp(-N^{\sigma}/d)\zeta'}{|x-x_0|}(x-x_0).
  \end{equation*}

  The second statement follows from the facts that
  \begin{equation*}
    \norm{H_N(z,\omega)-H_N(x,\omega)}\le C(V) |z-x|
  \end{equation*}
  and that if
  \begin{equation*}
    \norm{(H_N(x,\omega)-E)^{-1}}\norm{H_N(z,\omega)-H_N(x,\omega)}<1,
  \end{equation*}
  then $ H_N(z,\omega)-E $ would  be invertible.
\end{proof}

We will need the following immediate consequence of the previous lemma.
We refer to this result as the {\em spectral form} of
(LDT).

\begin{cor}\label{cor:4.6zeros}
  Assume $ x\in \T^d $, $ \omega\in \tor^d(a,b) $, $ E\in \C $, and $ L(\omega,E)> \gamma >0 $. Let
  $ \tau,\sigma $ as in (LDT). If $ N\ge N_0(V,a,b,|E|,\gamma) $ and
  \begin{equation*}
    \norm{(H_N(x,\omega)-E)^{-1}}\le \exp(N^{\sigma/2}),
  \end{equation*}
  then
  \begin{equation*}
    \log|f_N(x,\omega,E)|> NL_N(\omega,E)-N^{1-\tau/2}.
  \end{equation*}
\end{cor}

\subsection{Poisson's Formula}

Recall that for any solution $ \psi $ of the difference equation
\eqref{eq:hamiltonC}, Poisson's formula reads
\begin{equation}\label{eq:poissonC}
  \psi(m) = \cG_{[a, b]} (x,\omega,E;m, a)\psi(a-1) + \cG_{[a, b]} (x,\omega,E;m,b)\psi(b+1),\quad m \in [a, b],
\end{equation}
where $\cG_{[a,b]} (x,\omega,E) = \left(H_{[a,b]}(x,\omega) -E\right)^{-1}$ is the
Green's function.  In particular, if $\psi$ is a solution of equation
\eqref{eq:hamiltonC}, which satisfies a zero boundary condition at the
left or the right edge, i.e.,
\begin{equation}\label{dircondagain}
  \psi(a-1) = 0 \text{ or }\psi(b+1) = 0,
\end{equation}
then
\begin{equation}\label{eq:poissondirichlet1}
  \psi(m) =\cG_{[a,b]}(x,\omega,E;m, b)\psi(b+1) \text{ or } \psi(m)=\cG_{[a,b]}(x,\omega,E;m,a) \psi(a-1).
\end{equation}

Poisson's formula gives us a way to show how the decay of Green's
function on an interval can be deduced, via a covering argument, from
its decay on smaller subintervals. We let
\begin{equation*}
  \delta_{m,n}= \begin{cases}
	0 &, n\neq m\\
    1 &, n=m
  \end{cases}.
\end{equation*}
Our {\em covering lemma} is as follows.
\begin{lemma}\label{lem:Poissoncover}
  Let $ x,\omega\in \mathbb{T}^d$, $ E\in\mathbb{R} $, and
  $ [a,b]\subset \mathbb{Z} $. If for any $ m\in[a,b] $, there exists an interval
  $ I_m=[a_m,b_m]\subset[a,b] $ containing $m$ such that
  \begin{equation}\label{eq:Poissoncover-condition}
    (1- \delta_{a,a_m}) \left| \cG_{I_m}(x,\omega,E;m,a_m) \right|
    +(1- \delta_{b,b_m}) \left| \cG_{I_m}(x,\omega,E;m,b_m) \right|
    <1,
  \end{equation}
  then $ E\notin \spec H_{[a,b]}(x,\omega) $.
\end{lemma}
\begin{proof}
  Assume to the contrary that $ E\in \spec H_{[a,b]}(x,\omega) $ and
  let $ \psi $ be a corresponding eigenvector.  Let $ m\in [a,b] $ be
  such that $ |\psi(m)|=\max_n |\psi(n)| $. The hypothesis together
  with the Poisson formula \eqref{eq:poissonC} gives us that
  $ |\psi(m)|<\max(|\psi(a_m)|,|\psi(b_m)|) $ if $ a_m\neq a $ and
  $ b_m\neq b $, $ |\psi(m)|<|\psi(b_m)| $ if $ a_m=a $, and
  $ |\psi(m)|<|\psi(a_m)| $ if $ b_m=b $. In either case we reach a
  contradiction, so we must have
  $ E\notin \spec H_{[a,b]}(x,\omega) $.
\end{proof}

\begin{remark}
  We comment on the use of \cref{lem:Poissoncover}. By stability
  considerations, the condition \cref{eq:Poissoncover-condition} will
  hold for energies $ E $ in some interval
  $ (E_0-\delta,E_0+\delta) $.  The conclusion will then be that
  $ (E_0-\delta,E_0+\delta)\cap \spec H_{[a,b]}(x,\omega)=\emptyset $
  and therefore $ \norm{\cG_{[a,b]}(x,\omega,E_0)}<\frac{1}{\delta} $. One
  would then apply the spectral form of (LDT) and the next lemma to get
  the decay of Green's function on $ [a,b] $.
\end{remark}

The condition \cref{eq:Poissoncover-condition} and its stability in
$ E $ will be obtained from the following lemma in conjuction with
(LDT) for determinants.
\begin{lemma}\label{lem:Green}
  Assume $ x_0\in \tor^d $, $ \omega_0\in \tor^d(a,b) $,
  $ E_0\in \C $, and $ L(\omega_0,E_0)> \gamma >0 $.  Let $ K\in \R $ and $ \tau $
  be as in (LDT). There exists $ C_0=C_0(V,a,b,|E_0|,\gamma) $ such that if $ N\ge N_0(V,a,b,|E_0|,\gamma) $ and
  \begin{equation}\label{eq:fN_unter}
    \log \big | f_N(x_0, \omega_0,E_0) \big | > NL_N(\omega_0,E_0) - K,
  \end{equation}
  then for any $ (x,\omega,E)\in \T^d\times \T^d\times \C $ with
  $|x - x_0|,\ |\omega - \omega_0|,\ |E - E_0| < \exp(-(K+C_0N^{1-\tau}))$ we have
  \begin{align} \label{eq:Gjk}
    \big | \cG_{[1, N]} (x,\omega,E;j,k)\big | & \le \exp\left(-
               \frac{\gamma}{2}|k - j|+K+2C_0N^{1-\tau}\right), \\
    \big \| \cG_{[1, N]} (x, \omega,E) \big \| & \le
                                                 \exp(K+3C_0N^{1-\tau}). \label{eq:normG}
  \end{align}
\end{lemma}
\begin{proof}
  Take $|x - x_0|,\ |\omega - \omega_0|,\ |E - E_0| < \exp(-(K+CN^{1-\tau}))$ with $ C $ large enough.
  Using Corollary~\ref{cor:4.6} we have
  \begin{multline*}\label{eq:fN_untera1}
    \log \big | f_N(x, \omega,E) \big | \ge \log \big |
    f_N(x_0,\omega_0,E_0)\big |-1
    \ge NL_N(\omega_0,E_0) - K-1\\
    \ge NL_N(\omega,E)-K-2\ge NL(\omega,E)-K-2.
  \end{multline*}
  By Cramer's rule and the uniform upper bound of
  \cref{prop:logupper},
  \begin{multline*}
	\big | \cG_{[1, N]}(x, \omega,E; j, k) \big |
    = \frac{|f_{j-1} (x, \omega, E)|\cdot|f_{N - k}(x+k\omega,\omega, E)|}{|f_N(x,\omega,E)|}\\
    \le \exp \left( - (k-j) L(\omega, E) + CN^{1-\tau}+K+2\right)\le
    \exp(-\frac{\gamma}{2}(k-j)+K+2CN^{1-\tau})
  \end{multline*}
  (we assumed $ j\le k $ and we also used \cref{prop:log-Hoelder}).
  The estimate \eqref{eq:normG} follows from \eqref{eq:Gjk}.
\end{proof}

Finally, we give a very important application of the covering lemma in
combination with the spectral form of (LDT) and with Lemma~\ref{lem:Green}.
We call it the {\em covering form}  of (LDT).

\begin{lemma}\label{lem:Greencoverap1}
  Assume $ N\ge 1 $, $ x_0\in \tor^d $, $ \omega_0\in \tor^d(a,b) $, $ E_0\in \R $, and
  $ L(\omega_0,E_0)> \gamma >0 $.  Let $ \tau,\sigma $ be as in (LDT).
  Suppose that for each point $ m\in [1,N] $ there exists an interval $ I_m\subset [1,N] $
  such that:
  \begin{enumerate}
    \renewcommand{\theenumi}{\roman{enumi}}
  \item $ \dist(m,[1,N]\setminus I_m)\ge |I_m|/100 $,
  \item $ \ell_0(V,a,b,|E_0|,\gamma)\le |I_m| $,
  \item $ \log|f_{I_m}(x_0,\omega_0,E_0)|> |I_m|L_{|I_m|}(\omega_0,E_0)-|I_m|^{1-\tau/4} $.
  \end{enumerate}
  Then for any $ (x,\omega,E)\in \T^d\times \T^d\times \C $ such that
  \begin{equation*}
    |x-x_0|,\ N|\omega-\omega_0|,\ |E-E_0|< \exp(- 2 \max_m |I_m|^{1-\tau/4}),
  \end{equation*}
  we have
  \begin{equation*}
    \dist(E,\spec H_N(x,\omega))\ge \exp(- 2 \max_m |I_m|^{1-\tau/4}).
  \end{equation*}
  Furthermore, if we also have $ \omega\in \T^d(a,b) $ and $ \max_m |I_m| \le N^{\sigma/2} $, then
  \begin{equation*}
    \log \bigl | f_N\bigl(x,\omega, E\bigr)\bigr|> NL_N(\omega,E)- N^{1-\tau/2}.
  \end{equation*}
\end{lemma}
\begin{proof}
  Due to (iii) and \cref{lem:Green} we
  have that
  \begin{equation*}
    \left| \cG_{I_m}(x,\omega,E;m,k) \right|\le
    \exp \left( -\frac{\gamma}{2}|m-k|+\frac{3}{2}|I_m|^{1-\tau/4} \right),
  \end{equation*}
  provided
  \begin{equation*}
    |x-x_0|,N|\omega-\omega_0|,|E-E_0|<\exp\left(-\frac{3}{2}|I_m|^{1-\tau/4}\right).
  \end{equation*}
  This and assumptions (i) and (ii)
  guarantee that the assumptions of \cref{lem:Poissoncover} are satisfied, and therefore
  $ E\notin \spec H_N(x,\omega) $, whenever $$ |E-E_0|< \exp(-\frac{3}{2}\max_m|I_m|^{1-\tau/4}).$$ Therefore, if
  $ |E-E_0|<\exp(-\frac{3}{2}\max_m |I_m|^{1-\tau/4})/2 $, then
  \begin{equation*}
    \dist(E,\spec H_N(x,\omega))\ge  \exp\left(-\frac{3}{2}\max_m|I_m|^{1-\tau/4}\right)/2.
  \end{equation*}
  This proves the first statement.

  If we have  $ \max_m |I_m| \le N^{\sigma/2} $, then
  \begin{equation*}
    \dist(E,\spec H_N(x,\omega))\ge  \exp(-N^{\sigma/2})
  \end{equation*}
  and the second statement follows from \cref{cor:4.6zeros}.
\end{proof}

\subsection{Wegner's Estimate}
We will need an estimate for the probability that there exists an eigenvalue of $ H_N(x,\omega) $ in some given
interval $ (E-\epsilon,E+\epsilon) $. From \cref{lem:Green} and (LDT) it follows immediately that if
$ \epsilon=\exp(-CN^{1-\tau}) $, then this probability is less than $ \exp(-N^\sigma) $. However, we will need to
have control over intervals which are much larger. This is readily achieved by using the covering form of (LDT).

\begin{prop}\label{prop:Wegner}
  Assume $ \omega_0\in \tor^d(a,b) $, $ E_0\in \C $, and
  $ L(\omega_0,E_0)> \gamma >0 $.  Let $ \sigma $ be as in (LDT).
  Let $ \ell,N $ be integers such that $ (2\log N)^{1/\sigma}\le \ell\le N $.
  Then for any
  $N\ge N_0(V,a,b,|E_0|,\gamma)$ there exists a set $ \cB_{N,\omega_0,E_0} $,
  $ \mes(\cB_{N,\omega_0,E_0})<\exp(-\ell^\sigma/2) $ such that
  for any $ x\in \T^d\setminus\cB_{N,\omega_0,E_0} $ and any $ (\omega,E)\in \T^d\times \C $, 
  $N|\omega-\omega_0|,\ |E-E_0|< \exp(-\ell)$,  we have
  \begin{equation*}
    \dist(E,\spec H_N(x,\omega))\ge \exp(-\ell)
  \end{equation*}
\end{prop}
\begin{proof}
  Take $ \omega,E $ satisfying the assumptions. Let $ \cB_{N,\omega_0,E_0} $ be the set of $ x $ such that
  \begin{equation*}
    \log|f_\ell(x+(m-1)\omega_0,\omega_0,E_0)|> \ell L_\ell(\omega_0,E_0)-\ell^{1-\tau},\ m\in [1,N].
  \end{equation*}
  By (LDT), we have
  \begin{equation*}
    \mes(\cB_{N,\omega_0,E_0})< N\exp(-\ell^{\sigma})\le \exp(-\ell^\sigma/2).
  \end{equation*}
  By the covering form of (LDT), for any $ x\notin \cB_{N,\omega_0,E_0} $ we have $$\dist(E,\spec H_N(x,\omega))\ge \exp(-\ell),$$
  thus concluding the proof.
\end{proof}

An important consequence of Wegner's estimate is that the graphs of the eigenvalues cannot be too flat.

\begin{lemma}\label{lem:Wegner-graphs}
  Let $ \omega\in \tor^d(a,b) $, $ \gamma>0 $,
  $ \sigma $ be as in (LDT), and $ \ell,N $ be integers such that $ (2\log N)^{1/\sigma}\le \ell\le N $.
  For any  $N\ge N_0(V,a,b,\gamma)$  we have that if $ S\subseteq \T^d $ is connected and
  \begin{equation*}
    \mes(S)\ge \exp(-\ell^{\sigma/2}),
  \end{equation*}
  then
  \begin{equation*}
    \mes(E_j^{(N)}(S,\omega))\ge \exp(-\ell)
  \end{equation*}
  for any $ j\in \{ 1,\ldots,N \} $ such that $ L(\omega,E)>\gamma $ on $ E_j^{(N)}(S,\omega) $.
\end{lemma}
\begin{proof}
  Since the eigenvalues are continuous in phase, the sets $ E_j^{(N)}(S,\omega) $ are intervals.
  Assume that $ L(\omega,E)>\gamma $ on $ E_j^{(N)}(S,\omega) $ and let $ E $ be the middle point of the interval
  $ E_j^{(N)}(S,\omega) $. Then
  \begin{equation*}
    S\subset \{ x\in \T^d: \dist(E,\spec H_N(x,\omega))<\mes(E_j^{(N)}(S,\omega)) \}.
  \end{equation*}
  By \cref{prop:Wegner}, we need to have $ \mes(E_j^{(N)}(S,\omega))\ge \exp(-\ell) $, otherwise the above would
  imply $ \mes(S)<\exp(-\ell^{\sigma}/2) $.
\end{proof}

\subsection{Weierstrass' Preparation Theorem}
We  also need  to discuss shortly a version of  Weierstrass' preparation theorem for an analytic function
$f(z, w_1, \dots, w_d)$ defined in a polydisk
\begin{equation}
\cP = \cD(z_0, R_0) \times \prod^d_{j=1} \cD(w_{j,0}, R_0),\quad z_0,\ w_{j, 0} \in \IC,\quad
 R_0 > 0\ .
\end{equation}

\begin{lemma}
\label{lem:weier}
Assume that $f(\cdot, w_1, \dots, w_d)$ has no zeros on some circle
\begin{equation*}
  \left\{z:|z-z_0| = r \right\}, 0 < r < R_0/2,
\end{equation*}
for any $\uw = (w_1, \dots, w_d) \in
\cP = \prod\limits^d_{j=1} \cD(w_{j, 0}, r_{j,0})$ where $0<r_{j,0}<R_0$.
Then there exist a polynomial $P(z, \uw)
= z^k +a_{k-1} (\uw) z^{k-1} + \cdots + a_0 (\uw)$ with $a_j(\uw)$ analytic in
$\cP$ and an analytic function $g(z, \uw), (z, \uw) \in \cD(z_0, r) \times \cP$
so that the following properties hold:
\begin{enumerate}
\item[(a)] $f(z, \uw) = P(z, \uw) g(z, \uw)$ for any $(z, \uw) \in \cD(z_0, r) \times
\cP$,

\item[(b)] $g(z, \uw) \ne 0$ for any $(z, \uw) \in \cD(z_0, r) \times \cP$,

\item[(c)] for any $\uw \in \cP$, $P(\cdot, \uw)$ has no zeros in
  $\IC \setminus \cD(z_0,r)$.
\end{enumerate}
\end{lemma}
\begin{proof} By the usual Weierstrass argument, one notes that
$$
b_p(\uw) := \sum^k_{j=1} \zeta^p_j (\uw) = {1\over 2\pi i} \oint\limits_{|z-z_0|=r} z^p \
{\partial _z f(z, \uw) \over f(z, \uw)}\, dz
$$
are analytic in $\uw \in \cP$.  Here $\zeta_j (\uw)$ are the zeros of $f(\cdot, \uw)$
in $\cD(z_0, r)$.  Since the coefficients $a_j(\uw)$ are linear combinations of the $b_p$,
they are analytic in $\uw$.  Analyticity
of $g$ follows by standard arguments.
\end{proof}

\subsection{Resultants}\label{sec:resultants}
Let us recall the definition of the resultant of two polynomials and deduce a key property that
will be needed in conjunction with Weierstrass' Preparation Theorem.
Let $f(z) = z^k + a_{k-1} z^{k-1} + \cdots + a_0$,
$g(z) = z^m + b_{m-1} z^{m-1} + \cdots + b_0$ be polynomials with complex coefficients.
Let $\zeta_i$, $1 \le i \le k$ and $\eta_j$,
$1 \le j \le m$ be the zeros of $f(z)$ and $g(z)$ respectively.  The
resultant of $f$ and $g$ is defined as follows:
\begin{equation}\label{eq:resdef}
  \Res(f, g) = \prod_{i,j} (\zeta_i - \eta_j)=\prod_i g(\zeta_i)=(-1)^k\prod_j f(\eta_j).
\end{equation}
The resultant can also be expressed explicitly in terms of the
coefficients (see \cite[p.~200]{Lan02}):
\begin{equation*}\label{eq:resdef1}
  \Res(f, g) = \left|\begin{array}{ll}
                       {\overbrace{\begin{array}{lll}
                                     1 & 0 & \cdots\\
                                     a_{k-1}  & 1 & \cdots\\
                                     a_{k-2} & a_{k-1} & \cdots\\
                                     \cdots & \cdots & \cdots\\
                                     \cdots & \cdots & \cdots\\
                                     \cdots & \cdots & \cdots\\
                                     a_0 & a_1 &  \\
                                     0 & a_0 &  \end{array}}^m} &
                                                                  {\overbrace{\begin{array}{llll}
                                                                                1 & 0 & \cdots & 0\\
                                                                                b_{m-1} & 1 & \cdots & \cdots\\
                                                                                b_{m-2} & b_{m-1} &\cdots & \cdots\\
                                                                                \cdots & \cdots & \cdots & \cdots\\
                                                                                \cdots & \cdots & \cdots & \cdots\\
                                                                                \cdots & \cdots & \cdots& \cdots\\
                                                                                  &&&\\
                                                                                  &&&\\
                                                                              \end{array}}^k}
                     \end{array}\right|
                 \end{equation*}

\begin{lemma}\label{lem:resbasicprop1}
  Let $ f,g $ be polynomials as above and set
  \[
    s=\max(k,m),\quad r=\max (\max |\zeta_i|,\max    |\eta_j|).
  \]
  Let $ \delta\in(0,1) $.    If $       |\Res(f, g)|>\delta $
  and $ r\le 1/2 $, then
  \[
    \max  (|f(z)|,|g(z)|)> \left( \frac{\delta}{2} \right)^s\quad\text{for all }z.
  \]
\end{lemma}
\begin{proof}
  We argue by contradiction.
  Suppose there exists $ z $ such that
  \begin{equation*}
    \max  (|f(z)|,|g(z)|)\le   \left( \frac{\delta}{2} \right)^s.
  \end{equation*}
  Then there exist $ \zeta_{i_0} $, $ \eta_{j_0} $ such that $ |z-\zeta_{i_0}|,|z-\eta_{j_0}|\le \delta/2 $.
  We have that $ |\zeta_{i_0}-\eta_{j_0}|\le \delta $ and by the assumption that $ r\le 1/2 $,
  $ |\zeta_i-\eta_j|\le 1 $. Therefore, $ |\Res(f,g)|\le \delta $ and we arrived at a contradiction.
\end{proof}

\subsection{Semialgebraic Sets}


Recall that a set $ \cS\subset\R^n $ is called semialgebraic if it is
a finite union of sets defined by a finite number of polynomial
equalities and inequalities. More precisely, a semialgebraic set
$ \cS\subset\R^n $ is given by an expression
\begin{equation*}
  \cS=\cup_j \cap_{\ell\in L_j} \{ P_\ell s_{j\ell} 0  \},
\end{equation*}
where $ \{P_1,\ldots,P_s\} $ is a collection of polynomials of $ n $
variables,
\begin{equation*}
  L_j\subset\{1,\ldots,s\}\text{ and }s_{j\ell}\in\{>,<,=\}.
\end{equation*}
If the degrees of the polynomials are
bounded by $ d $, then we say that the degree of $ \cS $ is bounded by
$ sd $. We refer to \cite[Ch.~9]{Bou05} for more information on
semialgebraic sets.

Semialgebraic sets will be introduced by approximating the potential $ V $ with a polynomial $ \tilde V $. More
precisely, given $ N\ge 1 $, by truncating $ V $'s Fourier series and the Taylor series of
the trigonometric functions, one can obtain a polynomial $ \tilde V $ of degree $ \lesssim N^4 $ such that
\begin{equation}\label{eq:V-tilde}
  \norm{V-\tilde V}_\infty\lesssim\exp(-N^2)
\end{equation}
(see \cite[Ch.~10]{Bou05} for some details; the sup norm is taken over $ \T^d $).
The precise bound on the degree will not be important, as long as the
bound is polynomial in $ N $. If we let $\tilde H_N(x, \omega)$ be the operator with
this truncated potential $\tilde V$ and $ \tilde E_j^{(N)}(x,\omega) $ its eigenvalues, then
\begin{equation}\label{eq:E-tilde}
  |E_j^{(N)}(x,\omega)-\tilde E_j^{(N)}(x,\omega)|
  \le\norm{H_N(x,\omega)-\tilde H_N(x,\omega)}\lesssim \norm{V-\tilde V}_\infty.
\end{equation}
This and \cref{lem:stability-in-V} will ensure that our estimates at scale $ N $
are stable under the change of potential.

For the purpose of semialgebraic approximation we also need to consider the set
\begin{equation*}
  \tor^d_N(a,b)= \left\{ \omega\in \tor^d: \norm{k\cdot\omega}\ge \frac{a}{|k|^b}, \text{ for all } k\in \Z^d,
    0<|k|\le N \right\}.
\end{equation*}

\begin{remark}\label{rem:scale-N-sa-refinement}
  All the scale $ N $ results presented so far
  that require the frequency $ \omega $ to be Diophantine, also work for
  $ \omega\in \T^d_N(a,b) $. Furthermore, whenever the positivity of the Lyapunov exponent is required,
  the positivity of $ L_N $ suffices. This is simply because the proofs of the results at scale $ N $ never
  depend on what happens at larger scales. In particular, this observation applies to (LDT), Wegner's estimate,
  and the uniform upper bound from \cref{prop:logupper}.
\end{remark}

The only remaining issue is having a semialgebraic approximation
for $ L_N(\omega,E) $. We deal with this in Lemma \ref{lem:averaging}, in the same way as in \cite{BouGol00},
though with a different proof.
The proof of Lemma~\ref{lem:averaging} will be based on the following result.
\begin{lemma}[{\cite[Cor.~9.7]{Bou05}}]\label{lem:sublinear-count}
    Let $ \cS\subset [0,1]^d $ be semialgebraic of degree $ B $ and
    $ \mes (\cS)<\eta $. Let  $ J $ be an integer
    such that
    \begin{equation*}
        \log B \ll \log J < \log \frac{1}{\eta}.
    \end{equation*}
    Then, for any $ x_0\in\tor^d $ and $ \omega\in\tor_J^d(a,b) $,
    \begin{equation*}
        \# \{ j=1,\ldots,J : x_0+j\omega\in \cS (\mod \Z^d)  \}<J^{1-\delta}
    \end{equation*}
    for some $ \delta=\delta(\omega) $.
\end{lemma}

To apply Lemma~\ref{lem:sublinear-count} we need the following lemma.
\begin{lemma}\label{lem:semialgebraic-phase}
    Let $ N\ge 1 $, $ \omega\in \tor_N^d(a,b) $, $ E\in \R $ , $ \sigma,\tau $ as in (LDT), and
    \begin{equation*}
        \cB_N := \{ x\in \tor^d : |\log\norm{M_N(x,\omega,E)}- NL_N(\omega,E)|\ge 4N^{1-\tau}  \}.
    \end{equation*}
    There exists a semialgebraic set $ \cS_N $ such that $ \cB_N\subset \cS_N $, $ \deg (\cS_N) \le N^{C}$,
    and $ \mes (\cS_N)< \exp(-N^\sigma) $ provided $ N\ge N_0(V,a,b,|E|) $.
\end{lemma}
\begin{proof}
  Take $ \tilde V $ as in \cref{eq:V-tilde} and let
  \begin{equation*}
    \cS_N:= \left\{ x\in \tor^d : |\log\normhs{\tilde M_N(x,\omega,E)}- NL_N(\omega,E)|\ge 2N^{1-\tau}  \right\},
  \end{equation*}
  where $ \normhs{\cdot} $ denotes the Hilbert-Schmidt norm. Clearly $ \deg(\cS_N)\le N^C $, and
  $ \cB_N\subset\cS_N $ follows from \cref{lem:stability-in-V}. The measure estimate follows from (LDT) and the
  fact that, by \cref{lem:stability-in-V},
  \begin{equation*}
    \cS_N\subset \{ x\in \tor^d : |\log\norm{M_N(x,\omega,E)}- NL_N(\omega,E)|>N^{1-\tau}  \}.
  \end{equation*}
\end{proof}

We can now prove the result that we use to handle $ L_N(\omega,E) $.
\begin{lemma}\label{lem:averaging}
  Let  $\sigma, \tau $ as in (LDT).
  Then  for any $ N\ge N_0(V,a,b,|E|) $,
  \begin{equation*}
    C(a,b)\log N\le \log J <  N^\sigma,
  \end{equation*}
  $ x\in \T^d $,
  $ \omega\in \T_J^d(a,b) $, $ E\in \R $,
  we have
  \begin{equation*}
    \left| \frac{1}{J}\sum_{j=1}^J\log \norm{M_N(x+j\omega,\omega,E)}-NL_N(\omega,E) \right|
    \le 5 N^{1-\tau}.
  \end{equation*}
\end{lemma}
\begin{proof}
    Let $ \cB_N $ and $ \cS_N $ be as in Lemma \ref{lem:semialgebraic-phase}.
    For any $ x\in \tor^d \setminus \cS_N $ we have
    \begin{equation*}
        NL_N(\omega,E)-4N^{1-\tau}\le \log \norm{M_N(x,\omega,E)} \le NL_N(\omega,E)+4N^{1-\tau},
    \end{equation*}
    whereas for $ x\in \cS_N $ we only have
    \begin{equation*}
      0\le \log \norm{M_N(x,\omega,E)} \le C(V,|E|)N.
    \end{equation*}
    Take $ J $ satisfying the assumptions.
    From the above and Lemma \ref{lem:sublinear-count} we get
    \begin{multline*}
        5N^{1-\tau}\ge \frac{J-J^{1-\delta}}{J}(4N^{1-\tau})+\frac{J^{1-\delta}}{J}(CN) \\
        \ge \frac{1}{J} \sum_{j=1}^J \log \norm{M_N(x+j\omega,\omega,E)}-NL_N(\omega,E) \\
        \ge \frac{J-J^{1-\delta}}{J}(-4N^{1-\tau})+\frac{J^{1-\delta}}{J}(-NL_N(\omega,E))
        \ge -5N^{1-\tau}.
    \end{multline*}
    This concludes the proof.
\end{proof}

Finally, we discuss a particular result on semialgebraic approximation needed for our applications.

\begin{lemma}\label{lem:semialgebraic-all}
  Let  $ N\ge 1 $, $ \gamma>0 $, and $ \sigma,\tau $ as in (LDT).
  Let $ \cB_N $ be the set of $ (x,\omega,E)\in \T^d\times\T^d(a,b)\times \R $ such
  that $ L(\omega,E)>\gamma $ and
  \begin{equation*}
    \log|f_N(x,\omega,E)|\le NL_N(\omega,E)- N^{1-\tau/2}.
  \end{equation*}
  Then there exists a semialgebraic set $ \cS_N $ such that $ \cB_N\subset \cS_N $,
  $ \deg (\cS_N) \le N^{C(a,b)}$, and
  $ \mes (\cS_N) < \exp(-N^\sigma/2) $, provided $ N\ge N_0(V,a,b,\gamma) $.
  Furthermore, $$ \mes(\cS_N(\omega,E))<\exp(-N^\sigma) ,
  \quad \cS_N(\omega,E) = \{ x : (x,\omega,E)\in \cS_N \}. $$
\end{lemma}
\begin{proof}
  Note that by the spectral form of (LDT) (see \cref{cor:4.6zeros}), the definition of
  $ \cB_N $ implicitly restricts $ E $ to be in a bounded interval $ [-C(V),C(V)] $ containing
  $ \cup_{x,\omega}\spec H_N(x,\omega) $.
  Take $ \tilde V $ as in \cref{eq:V-tilde} and let $ \cS_N $ be the set
  of
  \begin{equation*}
    (x,\omega,E)\in \T^d\times \T_{J}^d(a,b)\times [-C(V),C(V)],\ J=N^{C(a,b)}
  \end{equation*}
  such that
  \begin{gather*}
    \frac{1}{NJ}\sum_{j=1}^J\log \normhs{\tilde M_N(x+j\omega,\omega,E)}\ge \frac{\gamma}{2}\\
    \log|\tilde f_N(x,\omega,E)|\le \frac{1}{J}\sum_{j=1}^J\log \normhs{\tilde M_N(x+j\omega,\omega,E)}
    -N^{1-\tau/2}/2.
  \end{gather*}
  Clearly $ \deg(\cS_N)\le N^C $, and
  $ \cB_N\subset\cS_N $ follows from \cref{lem:stability-in-V}, \cref{lem:averaging}, and \cref{prop:uniform}.
  For the measure estimate note that for $ (x,\omega,E) \in \cS_N $ we have $ L_N(\omega,E)\ge \gamma/4 $ and
  \begin{equation*}
    \log|f_N(x,\omega,E)|\le NL_N(\omega,E)-N^{1-\tau/2}/4.
  \end{equation*}
  By (LDT) (recall \cref{rem:scale-N-sa-refinement}), $ \mes(\cS_N(\omega,E))<\exp(-N^\sigma) $
  and therefore
  \begin{equation*}
    \mes(\cS_N)<2C(V)\exp(-N^\sigma)<\exp(-N^\sigma/2).
  \end{equation*}
\end{proof}

\subsection{Perturbation Theory}

In this section we state some standard facts on interlacing eigenvalues and basic perturbation
theory. They will be needed in \cref{sec:NDRWegner}.

\begin{lemma}[{\cite[Thm. 4.3.15]{HorJoh85}}]\label{lem:interlacing}
  Let $ H $ be a $ n\times n $ Hermitian matrix, let $ r $ be an
  integer with $ 1\le r\le n $, and let $ H_r $ be any $ r\times r $
  principal submatrix of $ H $ (obtained by deleting $ n-r $ rows and
  the corresponding columns from $ H $). For each integer $ k $ such
  that $ 1\le k\le r $ we have
  \begin{equation*}
    E_k(H)\le E_k(H_r)\le E_{k+n-r}(H)
  \end{equation*}
  ($ E_k(H) $ denotes $ H $'s $ k $-th eigenvalue; the eigenvalues are
  arranged in increasing order).
\end{lemma}

The following statements are standard facts from basic perturbation
theory.  We give the proofs for completeness.

\begin{lemma}\label{lem:P-Q}
  If $ P,Q $ are two arbitrary projectors on $ \C^n $ and
  $ \norm{P-Q}<1 $, then $ \rank P=\rank Q $.
\end{lemma}
\begin{proof}
  Without loss of generality, suppose to the contrary, that
  $ \rank P > \rank Q $.  Let $ \mathfrak{M}=P(\C^n) $,
  $ \mathfrak{N}=(I-Q)(\C^n) $.  Then
  \begin{equation*}
    \dim(\mathfrak M)+\dim(\mathfrak N)=\rank P+n-\rank Q>n
  \end{equation*}
  and therefore, there exists $ h\in \mathfrak M\cap \mathfrak N $
  with $ \norm{h}=1 $. Then we reach a contradiction because
  \begin{equation*}
    \norm{Ph-Qh}=\norm{h-0}=1.
  \end{equation*}
\end{proof}

\begin{lemma}
  \label{lem:evcount10}
  Let $H,H_0$ be $n\times n$ matrices, $H_0$ is Hermitian,
  $E_0\in\mathbb{R}$, $r_0>0$. Assume the number of
  eigenvalues of $H_0$ in $(E_0-r_0,E_0+r_0)$ is at most $ K $ and
  \[
    \|H-H_0\|\le \frac{r_0}{32(K+1)^2}.
  \]
  Then there exist $r_0/2< r< r_0$, which depends only on $H_0$,
  such that $ H $ and $ H_0 $ have the same number of eigenvalues in the disk
  $\cD(E_0,r)$.
  Moreover, neither $H$ nor $H_0$ have eigenvalues in  the region
  \begin{equation}\label{eq:region}
    r-\frac{r_0}{8(K+1)}\le |\zeta-E_0|\le r+\frac{r_0}{8(K+1)}.
  \end{equation}
\end{lemma}
\begin{proof}
  Clearly, there exists $r_0/2<r<r_0$ such that $H_0$ has no
  eigenvalues in the domain
  \begin{equation*}
    r-\frac{r_0}{4(K+1)}\le |\zeta-E_0|\le r+\frac{r_0}{4(K+1)}.
  \end{equation*}
  Since $H_0$ is
  Hermitian,
  \[
    \|(H_0-\zeta)^{-1}\|\le \frac{4(K+1)}{r_0},\quad \text{for any
      $|\zeta-E_0|=r$}.
  \]
  Recall the basic resolvent expansion estimate:
  \[
    \|(A+B)^{-1}-A^{-1}\|\le \|A^{-1}\|\frac{\norm{A^{-1}B}}{1-\norm{A^{-1}B}}
    \le 2\norm{A^{-1}}^2\norm{B},  \]
    provided $
    \|A^{-1}\|\|B\|<1/2.$
  This implies
  \[
    \|(H-\zeta)^{-1}-(H_0-\zeta)^{-1}\|\le 2\norm{(H_0-\zeta)^{-1}}^2\norm{H-H_0}\le \frac{1}{r_0}
    ,\quad
    \text{for any $|\zeta-E_0|=r$}.
  \]
  Consider the Riesz projectors
  \[
    P=\frac{1}{2\pi i}\oint_{|\zeta-E_0|=r}(H-\zeta)^{-1}d\zeta, \quad
    P_0=\frac{1}{2\pi i}\oint_{|\zeta-E_0|=r}(H_0-\zeta)^{-1}d\zeta.
  \]
  Then
  \[
    \|P-P_0\|\le
    \frac{r_0^{-1}}{2\pi}\oint_{|\zeta-E_0|=r}d|\zeta|=\frac{r}{r_0}<1.
  \]
  The previous lemma implies $ \rank P= \rank P_0 $. The first statement follows
  by recalling that $ P,P_0 $
  project onto the sum of the generalized eigenspaces of the eigenvalues in
  $ \cD(E_0,r) $.

  If $ \zeta $ is in the region \cref{eq:region}, then
  \begin{equation*}
    \norm{(H_0-\zeta)^{-1}}\norm{H-H_0}\le \frac{8(K+1)}{r_0}\frac{r_0}{32(K+1)^2}<1
  \end{equation*}
  and therefore $ H-\zeta $ is invertible and the second statement follows.
\end{proof}

\subsection{Stabilization of Eigenvalues and Eigenvectors}

In this section we present some basic results on the relation between the eigenvalues and eigenvectors at
different scales.

\begin{lemma}\label{lem:eigenvalue-stabilization}
	Let $ x,\omega\in \T^d $. For any intervals $ \Lambda_0=[a_0,b_0]\subset \Lambda\subset \Z $ and any
	$ j_0 $, we have
	\begin{equation*}
		\dist\left(E_{j_0}^{\Lambda_0}(x,\omega),\spec H_\Lambda(x,\omega)\right)
		\le \left| \psi_{j_0}^{\Lambda_0}(x,\omega;a_0) \right|
			+\left| \psi_{j_0}^{\Lambda_0}(x,\omega;b_0) \right|.
	\end{equation*}
\end{lemma}
\begin{proof}
	Let $ \psi_0 $ be the extension, with zero entries, of $ \psi_{j_0}^{\Lambda_0}(x,\omega) $ to
	$ \Lambda $. Since $ \norm{\psi_0}=1 $, the conclusion follows from the fact that we have
	\begin{align*}
      \|(H_\Lambda(x,\omega)-E_{j_0}^{\Lambda_0}(x,\omega))^{-1}\|^{-1} &\le
		\norm{(H_\Lambda(x,\omega)-E_{j_0}^{\Lambda_0}(x,\omega))\psi_0} \\
		&\le \left| \psi_{j_0}^{\Lambda_0}(x,\omega;a_0) \right|
			+\left| \psi_{j_0}^{\Lambda_0}(x,\omega;b_0) \right|.
	\end{align*}
\end{proof}

Recall the following simple general statement on the finite interval approximation of the spectrum
\begin{lemma}\label{lem:elemspec1}
  If for some $ x,\omega\in \T^d $, $ E\in \R $, $ \rho>0 $, there exist sequences
  $N'_k\to -\infty$, $N''_k\to +\infty$ such that
  \[
    \dist(E,\spec H_{[N_k',N_k'']}(x,\omega))\ge \rho,
  \]
  then
  \begin{equation*}
    \dist(E,\cS_\omega)\ge \rho.
  \end{equation*}
\end{lemma}
\begin{proof}
  Take arbitrary $ \phi\in \ell^2(\Z) $ with finite support.
  For any $k$ large enough so that $\supp \phi\subset (N'_k,N_k'')$
  we have
  \[
    (H(x,\omega)-E)\phi(n)=(H_{[N'_k,N''_k]}(x,\omega)-E)\phi(n),\quad n\in [N'_k,N''_k].
  \]
  Due to the hypothesis,
  \begin{align*}
    \|(H(x,\omega)-E)\phi\|&=\|(H_{[N'_k,N''_k]}(x,\omega)-E)\phi\|\\
&    \ge \|(H_{[N'_k,N''_k]}(x,\omega)-E)^{-1}\|^{-1}\|\phi\|\ge \rho\|\phi\|.
  \end{align*}
  Since this holds for any finite support $\phi$ it also holds for any $\phi \in \ell^2(\mathbb{Z})$,
  and the conclusion follows.
\end{proof}

The following standard result is the basis for the stabilization of eigenvectors.

\begin{lemma}\label{lem:eigenvector-stability} Let
  $A $ be a $ N\times N $ Hermitian matrix.  Let
  $E,\epsilon \in \mathbb{R}$, $ \epsilon>0 $ and suppose there exists
  $\phi\in \mathbb{\R}^N$, $\|\phi\|=1$, such that
  \begin{equation}\label{eq:2contraction1}
    \begin{split}
      \|(A-E)\phi\|< \varepsilon.
    \end{split}
  \end{equation}
  Then the following statements hold.

  \vspace{0.5em}
  \noindent {\normalfont (a)}
  There exists a normalized eigenvector $\psi$ of $A$ with an eigenvalue $E_0$
   such that
  \begin{equation}\label{eq:2contraction1aaaM}
    \begin{split}
    E_0\in (E-\varepsilon\sqrt2,E+\varepsilon\sqrt 2),\\
      |\langle \phi,\psi\rangle|\ge (2N)^{-1/2}.
    \end{split}
  \end{equation}

  \vspace{0.5em}
  \noindent {\normalfont (b)}  If in addition there exists $ \eta>\epsilon $ such that the subspace of the
  eigenvectors of $A$ with eigenvalues falling into the interval
  $(E-\eta,E+\eta)$ is at most of dimension one,
  then there exists a normalized eigenvector $\psi$ of $A$ with an eigenvalue
  $E_0\in (E-\varepsilon,E+\varepsilon)$, such that
  \begin{equation}\label{eq:2contraction1aaa}
    \begin{split}
      \|\phi-\psi\|<\sqrt{2}\eta^{-1}\varepsilon.
    \end{split}
  \end{equation}
\end{lemma}
\begin{proof} (a) Let $\psi_j$, $j=1,\dots, N$, be an
  orthonormal basis of eigenvectors of $A$, $A\psi_j=E_j\psi_j$.  Then
  \begin{equation*}
    \sum_{|\langle\phi,\psi_j\rangle|\ge (2N)^{-1/2}}|\langle \phi,\psi_j\rangle|^2=
    \norm{\phi}^2-\sum_{|\langle\phi,\psi_j\rangle|<(2N)^{-1/2}}|\langle \phi,\psi_j\rangle|^2> \frac{1}{2}
  \end{equation*}
  and
  \begin{multline*}
    \varepsilon^2>\norm{(A-E)\phi}^2=\sum_j|\langle \phi,\psi_j\rangle |^2(E_j-E)^2\ge
    \sum_{|\langle\phi,\psi_j\rangle|\ge (2N)^{-1/2}}|\langle \phi,\psi_j\rangle |^2(E_j-E)^2\\
    \ge\min_{|\langle \phi,\psi_j\rangle |\ge (2N)^{-1/2}}
    (E_j-E)^2\sum_{|\langle \phi,\psi_j\rangle |\ge (2N)^{-1/2}}|\langle \phi,\psi_j\rangle |^2
       >\frac{1}{2}\min_{|\langle \phi,\psi_j\rangle |\ge (2N)^{-1/2}} (E_j-E)^2.	
  \end{multline*}
 This finishes the proof of (a). To prove (b) note that
  \begin{equation*}
    \begin{split}
      \varepsilon^2>\sum_{j}|\langle \phi,\psi_j\rangle |^2(E_j-E)^2\ge \min_j (E_j-E)^2.
    \end{split}
  \end{equation*}
  This implies that there is $E_k\in (E-\varepsilon,E+\varepsilon)$.
  Due to our assumptions, for any $j\neq k$ one has
  $E_j\notin (E-\eta,E+\eta)$.  So
  \begin{equation}\label{eq:3resolventcontr1}
    \begin{split}
      \varepsilon^2>\sum_{j\neq k}|\langle \phi,\psi_j\rangle |^2(E_j-E)^2
      \ge \eta^2\sum_{j\neq k}|\langle \phi,\psi_j\rangle |^2\\
    \end{split}
  \end{equation}
  Thus,
  \begin{equation*}
    1-|\langle \phi,\psi_k \rangle|^2=
    \|\phi-\langle \phi,\psi_k\rangle \psi_k\|^2=\sum_{j\neq k}|\langle \phi,\psi_j\rangle|^2
    \le \eta^{-2}\varepsilon^2.
  \end{equation*}
  The conclusion now follows from the  fact that $ \norm{\phi-\psi_k}^2=2(1-\Re \langle \phi,\psi_k \rangle) $.
  Note that we can replace $ \psi_k $ by $ e^{i\theta}\psi_k $ to ensure
  $ \Re \langle \phi,\psi_k\rangle= |\langle \phi,\psi_k\rangle| $.
\end{proof}

We will need the following corollary of the first part in \cref{lem:eigenvector-stability}.
\begin{cor}\label{cor:localglob1}
  Let $ x,\omega\in \T^d $ and
  $ [a,b]\subset [c,d]$. Let $\varphi$ be a normalized eigenvector
  of $H_{[a,b]}(x,\omega)$, with $H_{[a,b]}(x,\omega)\varphi=E\varphi$. Set
  \[
    \varepsilon^2=|\varphi(a)|^2+|\varphi(b)|^2, N=d-c+1, \underline{N}=b-a+1.
  \]
  There exists a normalized eigenvector $\psi$ of $H_{[c,d]}(x,\omega)$ with an eigenvalue $E_0$
  such that
  \begin{equation}\label{eq:2contraction1aaaM5}
    \begin{split}
      E_0\in (E-\varepsilon\sqrt2,E+\varepsilon\sqrt 2),\\
      \max_{n\in [a,b]} |\psi(n)|\ge (2N\underline{N})^{-1/2}.
    \end{split}
  \end{equation}
\end{cor}
\begin{proof} Set
\begin{equation*}
		\phi(n) =\begin{cases}
			\varphi(n) &, n\in[a,b]\\
			0 &, n\in [c,d]\setminus [a,b].
		\end{cases}
	\end{equation*}
Note that
\begin{equation*}
		\|(H_{[c,d]}(x,\omega)-E)\phi\|^2=|\varphi(a)|^2+|\varphi(b)|^2=\varepsilon^2.
	\end{equation*}
    Due to the first part in \cref{lem:eigenvalue-stabilization} there exists an eigenvector $\psi_0$ of
    $H_{[c,d]}(x,\omega)$ with an eigenvalue $E_0$
   such that
  \begin{equation}\label{eq:2contraction1aaaM6}
    \begin{split}
    E_0\in (E-\varepsilon\sqrt2,E+\varepsilon\sqrt 2),\\
      |\langle \phi,\psi\rangle |\ge (2N)^{-1/2}.
    \end{split}
  \end{equation}
  Since  $|\langle \phi,\psi\rangle |^2\le \|\psi\vert_{[a,b]}\|^2$, \cref{eq:2contraction1aaaM6} implies
  \eqref{eq:2contraction1aaaM5}.
\end{proof}

\section{Localization and Separation of Eigenvalues  on a Finite Interval}
\label{sec:localization-separation}

In this section we discuss finite interval localization and separation of eigenvalues
independently of the details of elimination of resonances. Even in this setting we have
to deal with issues,  absent in \cite{GolSch08}, stemming from the largeness of the deviation in (LDT).
We start by deriving the localization of
Dirichlet eigenfunctions at a {\em given scale} $ N $ assuming ``no
long range double resonances'' at a much {\em smaller scale}
$\ell$. The precise meaning of ``no long range double resonances'' is
given by the condition \cref{eq:fellN1N2nores} from the next
proposition.

\begin{prop}\label{prop:NDRloc1}
  Assume $ x_0\in \tor^d $, $ \omega_0\in \tor^d(a,b) $,
  $ E_0\in \R $, and $ L(\omega_0,E_0)> \gamma >0 $.  Let
  $ \tau,\sigma $ be as in (LDT) and $ \ell,N $ be integers such that
  $\ell_0(V,a,b,|E_0|,\gamma)\le \ell \le N^{\sigma/2} $.  Assume that
  there exists an interval $ I=[N',N'']\subset [1,N] $ such that
  \begin{equation}\label{eq:fellN1N2nores}
    \log \big | f_\ell(x_0+(m-1)\omega_0, \omega_0,E_0) \big |>  \ell L_\ell(\omega_0,E_0) - \ell^{1-\tau/4}
    \text{ for  any $m\in [1,N-\ell+1]\setminus I$}.
  \end{equation}
  Then for any $ (x,\omega)\in \T^d\times \T^d(a,b) $, $|x - x_0|,\ |\omega - \omega_0|< \exp(-\ell)$ and
  any eigenvalue
  $|E_j^{(N)} (x, \omega)-E_0|< \exp\left(-\ell\right)$, the
  corresponding eigenfunction obeys
  \begin{equation}\label{eq:NDRlocalization1}
    \big | \psi_j^{(N)} (x, \omega;n) \big | < \exp\left(- \frac{\gamma}{4} \dist(n,I)\right),
  \end{equation}
  provided $ \dist(n,I)\ge \ell^{2/\sigma} $.
\end{prop}
\begin{proof}
  Take $ x,\omega,E=E_j^{(N)}(x,\omega) $ satisfying the assumptions, and $ n\in [1,N]\setminus I $ such that
  $ d=\dist(n,I)\ge \ell^{2/\sigma} $. Assume $ n\in [1,N'] $. Let
  $ J=[n-d,n+d]\cap [1,N]=[a,N'] $. Note that
  $ |J|\ge \ell^{2/\sigma} $. By the covering form of (LDT) (see \cref{lem:Greencoverap1}) we have
  \begin{equation}\label{eq:LDE-J}
    \log|f_J(x,\omega,E)|\ge |J|L_{|J|}(\omega,E)-|J|^{1-\tau/2}.
  \end{equation}
  Let $ \psi=\psi_j^{(N)}(x,\omega) $. By Poisson's formula,
  \begin{equation*}
    \psi(n)= \begin{cases}
      \cG_{J}(x,\omega,E;n,a)\psi(a-1)+\cG_{J}(x,\omega,E;n,N')\psi(N'+1) &, a>1\\
      \cG_{J}(x,\omega,E;n,N')\psi(N'+1) &, a=1
    \end{cases}
  \end{equation*}
  (recall the Dirichlet boundary condition $ \psi(0)=0 $).  Now,
  \cref{eq:LDE-J} and \cref{lem:Green} imply
  \begin{equation*}
    \left| \psi(n) \right|\le 2 \exp \left( -\frac{\gamma}{2} d+C|J|^{1-\tau/2} \right)
    < \exp \left( -\frac{\gamma}{4} d \right)
  \end{equation*}
  (recall that $ \psi $ is normalized and therefore $ |\psi(k)|\le 1 $
  for any $ k\in [1,N] $).  The case $ n\in[N'',N] $ is completely
  analogous.
\end{proof}

The second issue we study in this section is a quantitative
estimate for the {\em separation of the eigenvalues}. More precisely, we consider the separation of a localized
eigenvalue $ E_j^{(N)} (x, \omega) $ (as in \cref{prop:NDRloc1}) from all the other eigenvalues of
$ H_N(x,\omega) $.
We recall the following basic observation regarding the relation between the Dirichlet
determinants and the solutions to the difference equation \cref{eq:hamiltonC}. If $ \psi $ is a solution of
\cref{eq:hamiltonC} that satisfies $ \psi(0)=0 $, $ \psi(1)=1 $, then by \cref{eq:transfer} and
\cref{eq:Mndetbasic} we have $ \psi(n)=f_{[1,n-1]}(x,\omega,E) $, $ n\ge 1 $ (we convene that $ f_{[1,0]}=1 $).
In particular, if $ E\in \spec H_N(x,\omega) $, then $ (f_{[1,n-1]}(x,\omega,E))_{n\in [1,N]} $
is a corresponding eigenfunction. The idea of our method is as follows. The eigenfunctions
corresponding to different eigenvalues are orthogonal. As we just saw,
each eigenfunction can be expressed in terms of  Dirichlet determinants,
evaluated at the corresponding eigenvalue. We show
that the determinants evaluated at close energies are close
themselves. That puts a limitation on how close  two different
eigenvalues can be. We {\em cannot} use the estimate from
Corollary~\ref{cor:4.6} because it is too imprecise and would only give separation of eigenvalues
by $ \exp(-CN^{1-\tau}) $. Instead, we use Harnack's inequality for harmonic functions.

\begin{lemma}\label{lem:Harnack-app}
  Assume $ x_0\in \tor^d $, $ \omega_0\in \tor^d(a,b) $,
  $ E_0\in \R $, and $ L(\omega_0,E_0)> \gamma >0 $.  Let
  $ \tau,\sigma $ be as in (LDT) and $ \ell,N $ be integers such that
  $\ell_0(V,a,b,|E_0|,\gamma)\le \ell \le N^{\sigma} $.  Assume that
  there exists an interval $ I=[N',N'']\subset [1,N] $ such that \cref{eq:fellN1N2nores} holds.
  Then for any $ (x,\omega)\in \T^d\times \T^d(a,b) $, $|x - x_0|,\ |\omega - \omega_0|< \exp(-\ell)$ and
  any $ E_i\in \R $,  $|E_i-E_0|< \exp\left(-\ell\right)$, $ i=1,2 $, we have
  \begin{equation*}
    |f_n(x,\omega,E_1)-f_n(x,\omega,E_2)|\le n\exp(\ell)|E_1-E_2|\max(|f_n(x,\omega,E_1)|,|f_n(x,\omega,E_2)|)
  \end{equation*}
  for any $ \ell^{2/\sigma}\le n\le N' $.
\end{lemma}
\begin{proof}
  Let $ n\in [\ell^{2/\sigma},N'] $. Fix $ x,\omega $ satisfying the assumptions.
  From  \cref{eq:fellN1N2nores} and the covering form of (LDT)
  (see \cref{lem:Greencoverap1}) it follows that
  \begin{equation}\label{eq:non-vanish}
    \log|f_n(x,\omega,E)|\ge nL_n(\omega,E)-n^{1-\tau/2}\text{ for any } E\in \cD(E_0,\exp(-C\ell^{1-\tau})).
  \end{equation}
  It follows that
  \begin{equation*}
    u(E)= C(V,|E_0|)n-\log|f_n(x,\omega,E)|
  \end{equation*}
  is harmonic and positive on $ \cD(E_0,\exp(-C\ell^{1-\tau})) $ (recall \cref{eq:monodr1}).
  We take $ r=|E_1-E_2| $, $ R= \exp(-C\ell^{1-\tau})$ (note that $ r/R\ll 1/2 $) and using
  Harnack's inequality we get
  \begin{equation*}
    \left( 1-2 \frac{r}{R} \right)u(E_2)\le \frac{R-r}{R+r}u(E_2)\le u(E_1)
    \le \frac{R+r}{R-r} u(E_2)\le \left( 1+4\frac{r}{R} \right)u(E_2).
  \end{equation*}
  It follows that
  \begin{equation*}
    |\log|f_n(x,\omega,E_1)|-\log|f_n(x,\omega,E_2)||\le \frac{Cn}{R}|E_1-E_2|.
  \end{equation*}
  By the Mean Value Theorem,
  \begin{equation*}
    ||f_n(x,\omega,E_1)|-|f_n(x,\omega,E_2)||\le \frac{Cn}{R}|E_1-E_2|\max(|f_n(x,\omega,E_1)|,|f_n(x,\omega,E_2)|)
  \end{equation*}
  and the conclusion follows from the fact that $ f_n(x,\omega,E) $ does not vanish (due to \cref{eq:non-vanish})
  and hence it has constant sign for $ E\in \R $.
\end{proof}

\begin{prop}\label{prop:Ej_NDRsep}
  Assume $ x_0\in \tor^d $, $ \omega_0\in \tor^d(a,b) $,
  $ E_0\in \R $, and $ L(\omega_0,E_0)> \gamma >0 $.  Let
  $ \tau,\sigma $ be as in (LDT) and $ \ell,N $ be integers such that
  $\ell_0(V,a,b,|E_0|,\gamma)\le \ell \le N^{\sigma} $.  Assume that
  there exists an interval $ I=[N',N'']\subset [1,N] $ such that \cref{eq:fellN1N2nores} holds
  and $ |I|\ge \ell^{2/\sigma}+\log N $.
  Then for any $ (x,\omega)\in \T^d\times \T^d(a,b) $, $|x - x_0|,\ |\omega - \omega_0|< \exp(-\ell)$ and
  any eigenvalue such that $|E_j^{(N)} (x, \omega)-E_0|< \exp\left(-\ell\right)/2$, we have
  \begin{equation}\label{eq:separation}
    |E_j^{(N)}(x,\omega)-E_k^{(N)}(x,\omega)|\ge \exp(-C|I|) \text{ for any }k\neq j,
  \end{equation}
  with $ C=C(V,|E_0|) $.
\end{prop}
\begin{proof}
  We argue by contradiction. Let $ E_1=E_j^{(N)}(x,\omega) $, $ E_2=E_k^{(N)}(x,\omega) $ and assume
  \begin{equation*}
    |E_1-E_2|<\exp(-C|I|).
  \end{equation*}
  Note that we therefore have $ |E_2-E_0|< \exp(-\ell) $, so \cref{prop:NDRloc1} can be applied to both
  eigenvalues.
  Let $ \psi_i(n)=f_{n-1}(x,\omega,E_i) $. Since $ \psi_1 $, $ \psi_2 $ are eigenfunctions corresponding to
  different eigenvalues we have that they are orthogonal and therefore
  \begin{equation*}
    \norm{\psi_1-\psi_2}^2=\norm{\psi_1}^2+\norm{\psi_2}^2.
  \end{equation*}
  Let $ \tilde I= \{ n\in  [1,N] : \dist(n,I)<\ell^{2/\sigma} \} $. By \cref{prop:NDRloc1},
  \begin{equation*}
    \sum_{n\notin \tilde I} |\psi_i(n)|^2
    \le \norm{\psi_i}^2 \sum_{n\notin \tilde I} \exp \left( -\frac{\gamma}{2}\dist(n,I) \right)
    \le \norm{\psi_i}^2 \exp \left( -\frac{\gamma}{4} \ell^{2/\sigma} \right)
  \end{equation*}
  (we used the fact that $ \ell $ is taken to be large enough).
  Therefore
  \begin{equation}\label{eq:sum-I}
    \sum_{n\notin \tilde I} |\psi_1(n)-\psi_2(n)|^2
    \le 2 \exp \left( -\frac{\gamma}{4} \ell^{2/\sigma} \right)(\norm{\psi_1}^2+\norm{\psi_2}^2).
  \end{equation}

  Let
  \begin{equation*}
    m= \begin{cases}
      N'-\ell^{2/\sigma} &, N'\ge 2\ell^{2/\sigma}+1\\
      1 &, N'<2\ell^{2/\sigma}
    \end{cases}.
  \end{equation*}
  By \cref{lem:Harnack-app} we have
  \begin{equation}\label{eq:psi1-psi2}
    |\psi_1(k)-\psi_2(k)|\le N\exp(\ell)|E_1-E_2|\max(|\psi_1(k)|,|\psi_2(k)|) \text{ for }k\in \{ m-1,m \}
  \end{equation}
  (note that the estimates hold trivially when $ m=1 $). For any $ n\in \tilde I $ we have
  \begin{multline*}
	|\psi_1(n)-\psi_2(n)|\le
    \mnorm{
      \begin{bmatrix}
        \psi_1(n+1)\\ \psi_1(n)
      \end{bmatrix}-
      \begin{bmatrix}
        \psi_2(n+1)\\ \psi_2(n)
      \end{bmatrix}
    } \\= \mnorm{
      M_{[m,n]}(x,\omega,E_1)\begin{bmatrix}
        \psi_1(m)\\ \psi_1(m-1)
      \end{bmatrix}-
      M_{[m,n]}(x,\omega,E_2)
      \begin{bmatrix}
        \psi_2(m)\\ \psi_2(m-1)
      \end{bmatrix}
    }\\ \le
    \mnorm{
      M_{[m,n]}(x,\omega,E_1)\begin{bmatrix}
          \psi_1(m)-\psi_2(m)\\ \psi_1(m-1)-\psi_2(m-1)
        \end{bmatrix}
    }\\+\mnorm{
      (M_{[m,n]}(x,\omega,E_1)-M_{[m,n]}(x,\omega,E_2))\begin{bmatrix}
        \psi_2(m)\\ \psi_2(m-1)
      \end{bmatrix}
    }.
  \end{multline*}
  Therefore, using \cref{eq:psi1-psi2} and \cref{lem:Eomdiff} we get
  \begin{equation*}
    |\psi_1(n)-\psi_2(n)|\le N\exp(C(V,|E_0|)(|I|+\ell^{2/\sigma}))|E_1-E_2|\max(\norm{\psi_1},\norm{\psi_2})
    \text{ for }n\in \tilde I
  \end{equation*}
  and
  \begin{equation}\label{eq:sum-non-I}
    \sum_{n\in \tilde I}|\psi_1(n)-\psi_2(n)|^2\le \exp(-C|I|)
    (\norm{\psi_1}^2+\norm{\psi_2}^2),
  \end{equation}
  provided the constant from \cref{eq:separation} is chosen large enough
  (recall that $ |I|\ge \ell^{2/\sigma}+\log N $). By \cref{eq:sum-I} and \cref{eq:sum-non-I},
  \begin{equation*}
    \norm{\psi_1-\psi_2}^2\ll \norm{\psi_1}^2+\norm{\psi_2}^2
  \end{equation*}
  contradicting the fact that $ \psi_1,\psi_2 $ are orthogonal.
\end{proof}

\section{(NDR) Condition}\label{sec:NDR-condition}

In this section we introduce the main new ingredient for the proof of the finite scale localization: the existence
of intervals satisfying the following ``no double resonances'' (NDR) condition.

\begin{defi}\label{def:NDR}
  Let $ \sigma,\tau $ be as in (LDT). We say that an interval
  $ \Lambda\subset \Z $ is $ (K,\ell,C) $-(NDR) with respect to
  $ x_0,\omega_0,E_0 $, if there exists $ \underline \Lambda\subset \Lambda $,
  $ |\underline \Lambda|\le K $, such that
  \begin{equation*}
    \log|f_\ell(x_0+(n-1)\omega_0,\omega_0,E_0)|> \ell L_\ell(\omega_0,E_0)-C\ell^{1-\tau/3}
    \text{ for all }n\in \Lambda\setminus \underline \Lambda.
  \end{equation*}
  Furthermore, we require that the connected components of
  $ \Lambda\setminus \underline \Lambda $ have length greater than
  $ \ell^{2/\sigma} $. If $ C=1 $ we say that $ \Lambda $ is $ (K,\ell) $-(NDR).
\end{defi}

\begin{remark}
  We shift by $ (n-1)\omega_0 $, instead of $ n\omega_0 $, to make sure that the intervals on which the estimate
  holds cover $ \Lambda\setminus \underline \Lambda $. Note that
  \begin{equation*}
    f_\ell(x_0+(n-1)\omega_0,\omega_0,E_0)=
  f_{[n,n+\ell-1]}(x_0,\omega_0,E_0).
  \end{equation*}
  The assumption on the connected components of
  $ \Lambda\setminus \underline \Lambda $ is just a matter of convenience (it facilitates the application of the
  covering form of (LDT)).  When this condition is not satisfied, one can simply choose a larger
  $ \underline \Lambda $. It goes without saying that for applications we will want $ K $ to be small relative
  to $ |\Lambda| $ (to be more precise, we will need that $ K\le |\Lambda|^\epsilon $ with some
  sufficiently small  $ \epsilon >0 $).
\end{remark}

\begin{remark}\label{rem:NDR-stability}
  The only reason we consider (NDR) intervals with $ C\neq 1 $ is the following stability property, that will be
  used in \cref{prop:NDR-elimination}.
  If $ C<1 $, $ \Lambda $ is $ (K,\ell,C) $-(NDR) with respect to $ x_0,\omega_0,E_0 $
  and $ L(\omega_0,E_0)>\gamma>0 $, then,
  by \cref{cor:4.6},  $ \Lambda $ is $ (K,\ell) $-(NDR) with respect to $ x,\omega,E $, provided
  \begin{equation*}
    |x-x_0|,\ |\omega-\omega_0|,\ |E-E_0|<\exp(-C'\ell^{1-\tau/3}).
  \end{equation*}
  Of course, $ \ell $ also needs to be large enough.
\end{remark}

The existence of (NDR) intervals (with small enough $ K $) follows from the work of Bourgain about localization
on $ \Z^d $ \cite{Bou07}. More specifically, we make use of his result on elimination of multiple resonances. We
recall the relevant abstract lemmas from \cite{Bou07} and then we apply them to our concrete situation.

\begin{lemma}[{\cite[Lem. 1.18]{Bou07}}]\label{lem:Bourgain-elimination}
  Let $ \cA\subset [0,1]^{q+r} $ be semialgebraic of degree $ B $
  and such that for each $ t\in [0,1]^r $, $ \mes_q(\cA(t))<\eta
  $. Then
  \begin{equation*}
    \{ (x_1,\ldots,x_{2^r}) : \cA(x_1)\cap \ldots \cap \cA(x_{2^r})\neq \emptyset  \}\subset [0,1]^{q2^r}
  \end{equation*}
  is semialgebraic of degree at most $ B^C $ and measure at most
  \begin{equation*}
    \eta_r=B^C \eta^{q^{-r}2^{-r(r-1)/2}}\text{ with } C=C(r).
  \end{equation*}
\end{lemma}
\begin{remark}
  Note that in the previous lemma $ \cA(t)= \{ x\in[0,1]^q : (x,t)\in \cA \} $
  and $ \cA(x)= \{ t\in [0,1]^r: (x,t)\in \cA \} $.
\end{remark}
\begin{lemma}[{\cite[Lem. 1.20]{Bou07}}]\label{lem:Bourgain-lacunary}
  Let $ \cA\subset [0,1]^{rq} $ be a semialgebraic set of degree
  $ B $ and $ \mes_{rq} (\cA)<\eta $.
  Let $ \cN_1,\ldots,\cN_{q-1}\subset \Z $ be finite sets with the
  property that
  \begin{equation*}
    |n_i|>(B|n_{i-1}|)^C,\text{ if } n_i\in \cN_i\text{ and }n_{i-1}\in\cN_{i-1},\ 2\le i\le q-1,
  \end{equation*}
  where $ C=C(q,r) $. Assume also
  \begin{equation*}
    \max_{n\in \cN_{q-1}} |n|^C<\frac{1}{\eta}.
  \end{equation*}
  Then
  \begin{equation*}
    \mes \{ \omega\in [0,1]^r : (\omega,n_1\omega,\ldots,n_{q-1}\omega)\in \cA
    \text{ for some } n_i\in \cN_i\}<B^C \left(\min_{n\in\cN_1} |n|\right)^{-1}.
  \end{equation*}
\end{lemma}

Combining Lemma~\ref{lem:Bourgain-elimination} and
Lemma~\ref{lem:Bourgain-lacunary} we obtain the following result
about the structure of the scale $\ell$ resonant shifts on the orbit of
sub-exponential length  for an arbitrary $x\in\tor^d$. The result is similar to the
Claim from \cite[p. 694]{Bou07}.
\begin{prop}\label{prop:NDR}
  Let $  \ell\ge 1  $, $ \gamma>0 $, and $ \sigma,\tau $ as in (LDT).
  Let  $ \cN_1,\ldots,\cN_{q-1}\subset \Z $, $ q=2^{2d+1} $, be finite sets with the
  property that
  \begin{equation*}
    |n_i|>(\ell^{C(d)} |n_{i-1}|)^{C'(d)},\text{ if } n_i\in \cN_i
    \text{ and }n_{i-1}\in\cN_{i-1},\ 2\le i\le q-1
  \end{equation*}
  and
  \begin{equation*}
    \max_{n\in \cN_{q-1}} |n| < \exp(c\ell^\sigma),\ c=c(d).
  \end{equation*}
  For any $ \ell\ge \ell_0(V,a,b,\gamma)$  there exists a set  $ \Omega_\ell $, depending on the choice of
  finite sets $ \cN_i $, such that
  \begin{equation*}
    \mes (\Omega_\ell) \le \ell^{C(a,b)} \left( \min_{n\in \cN_1}|n| \right)^{-1}
  \end{equation*}
  and the following statement holds.  For any $ x\in\tor^d $,
  $ \omega\in \tor^d(a,b)\setminus\Omega_\ell $, $ E\in \R $,
  if $ L(\omega,E)>\gamma $ and
  \begin{equation*}
    \log|f_\ell(x,\omega,E)|\le \ell L_\ell(\omega,E)-\ell^{1-\tau/2},
  \end{equation*}
  then there exists $ i\in \{ 1,\ldots,q-1 \} $, depending on
  $ x,\omega,E $, such that
  \begin{equation}\label{eq:i}
    \log|f_\ell(x+(n-1)\omega,\omega,E)|>\ell L_\ell(\omega,E)-\ell^{1-\tau/2},
    \text{ for all $ n\in \cN_i $.}
  \end{equation}
\end{prop}
\begin{proof}
  Let $ \cB_{\ell}, \cS_{\ell} $ be the sets from Lemma \ref{lem:semialgebraic-all}.
  We have that $ \cS_\ell $ is semialgebraic, $ \cB_\ell\subset \cS_\ell $, $ \deg(\cS_\ell)\le \ell^{C(a,b)}  $,
  $ \mes(\cS_\ell(\omega,E))<\exp(-\ell^\sigma) $.
  Let $ \cT_\ell $ be the set of
  \begin{equation*}
    (y,x,\omega,E)\in \T^d\times\T^d\times \T_{\ell^C}^d(a,b)
    \times \R
  \end{equation*}
  such that $ (x+y-\omega,\omega,E)\in \cS_\ell $. Clearly, $ \cT_\ell $ is a
  semialgebraic set and
  \begin{equation*}
    \deg (\cT_\ell)\le \ell^C,\ \mes (\cT_\ell(x,\omega,E))< \exp(-\ell^\sigma).
  \end{equation*}
  By Lemma \ref{lem:Bourgain-elimination} (applied with $ \cA=\cT_\ell $, $ t=(x,\omega,E) $, $ x=y $,
  $ q=d $, $ r=2d+1 $, $ B=\ell^C $, $ \eta=\exp(-\ell^{\sigma}) $; note that $ \cA $ and $ q $ are different in our
  proof), the set
  \begin{equation*}
    \cA:= \{ (y_1,\ldots,y_q) : \cT_\ell(y_1)\cap \ldots\cap \cT_\ell(y_q)\neq \emptyset  \}
  \end{equation*}
  is semialgebraic and satisfies $ \deg (\cA)\le \ell^C $, $ \mes (\cA)< \exp(-c\ell^\sigma) $, $ c=c(d) $.
  The conclusion follows from Lemma~\ref{lem:Bourgain-lacunary} (applied with $ \cA=\cA $, $ q=q $, $ r=d $,
  $ B=\ell^C $, $ \eta=\exp(-c\ell^\sigma) $, $ \cN_i=\cN_i $)
  by letting
  \begin{equation*}
    \Omega_\ell= \{ \omega : (\omega,n_1\omega,\ldots,n_{q-1}\omega)\in \cA \text{ for some } n_i\in \cN_i \}.
  \end{equation*}
\end{proof}

A typical example of how the previous proposition leads to (NDR) intervals is obtained by considering  the sets
\begin{equation*}
  \cN_i = \{ n\in \Z : \ell^{\underline C_i}\le |n|\le \ell^{\overline C_i} \},\ i=1,\ldots,q-1,\ q=2^{2d+1}.
\end{equation*}
We can choose the constants
\begin{equation*}
  1\ll \underline C_1\ll \overline C_1 \ll \ldots \ll \underline C_{q-1}\ll \overline C_{q-1}
\end{equation*}
such that Proposition~\ref{prop:NDR} applies. Then (provided $ \omega\notin \Omega_\ell $ and $ L(\omega,E)>0 $)
we have that for each $ x\in \tor^d $, either
\begin{equation*}
  \log|f_\ell(x,\omega,E)|>\ell L_\ell(\omega,E)-\ell^{1-\tau/2},
\end{equation*}
or there exists $ i\in \{ 1,\ldots,q-1 \} $, depending on $ x $,
such that the interval
$ [-\ell^{\overline C_i},\ell^{\overline C_i}] $ is $ (\ell^{\underline C_i},\ell) $-(NDR) with
respect to $ x,\omega,E $.
The fact that $ i $ depends on $ x $
poses the following problem.  In
Proposition~\ref{prop:NDR-elimination} we will eliminate resonances
between (NDR) intervals of the same size, but at this point we
cannot guarantee that we can get (NDR) intervals of the same size
for any pair of distinct phases, say $ x+n\omega $ and
$ x+m\omega $. We solve this issue in the following corollary.

\begin{cor}\label{cor:same-scale-NDR}
  Let $  \ell\ge 1  $, $ \gamma>0 $, and $ \sigma,\tau $ as in (LDT).
  Let $ \cN_1,\ldots,\cN_{(q-1)^2}\subset \Z $, $ q=2^{2d+1} $, be finite sets with the
  property that
  \begin{equation*}
    |n_i|>(\ell^{C(d)} |n_{i-1}|)^{C'(d)},\text{ if } n_i\in \cN_i
    \text{ and }n_{i-1}\in\cN_{i-1},\ 2\le i\le (q-1)^2
  \end{equation*}
  and
  \begin{equation*}
    \max_{n\in \cN_{(q-1)^2}} |n| < \exp(c\ell^\sigma),\ c=c(d).
  \end{equation*}
  For any $ \ell\ge \ell_0(V,a,b,\gamma)$  there exists a set  $ \Omega_\ell $, depending on the choice of
  finite sets $ \cN_i $, such that
  \begin{equation*}
    \mes (\Omega_\ell) \le \ell^{C(a,b)} \left( \min_{n\in \cN_1}|n| \right)^{-1}
  \end{equation*}
  and the following statement holds.  For any $ x_1,x_2\in\tor^d $,
  $ \omega\in \tor^d(a,b)\setminus\Omega_\ell $, $ E\in \R $,
  if $ L(\omega,E)>\gamma $ and
  \begin{equation*}
    \log|f_\ell(x_j,\omega,E)|\le \ell L_\ell(\omega,E)-\ell^{1-\tau/2},\ j=1,2
  \end{equation*}
  then there exists $ i\in \{ 1,\ldots,(q-1)^2 \} $, depending on
  $ x_1,x_2,\omega,E $, such that
  \begin{equation*}
    \log|f_\ell(x_j+(n-1)\omega,\omega,E)|>\ell L_\ell(\omega,E)-\ell^{1-\tau/2},\ j=1,2,
    \text{ for all $ n\in \cN_i $.}
  \end{equation*}
\end{cor}
\begin{proof}
  Let $
  \Omega_\ell  $ be the union of the sets of exceptional phases obtained by
  applying Proposition~\ref{prop:NDR} with the following choices of
  finite sets:
  \begin{gather}
    \cN_{k(q-1)+1},\ldots,\cN_{(k+1)(q-1)},\ k=0,\ldots,q-2,\label{eq:lacunary-sets}\\
    \bigcup_{j=1}^{q-1} \cN_j,\bigcup_{j=1}^{q-1}
    \cN_{(q-1)+j},\ldots,\bigcup_{j=1}^{q-1}\cN_{(q-2)(q-1)+j}.
    \label{eq:lacunary-sets-union}
  \end{gather}
  The set $
  \Omega_\ell $ clearly satisfies the stated measure bound. Let
  $ \omega\in \tor^d(a,b)\setminus\Omega_\ell
  $. By Proposition~\ref{prop:NDR}, with
  \eqref{eq:lacunary-sets-union}, there exists $
  k\in\{0,\ldots,q-2\} $ such that
  \begin{equation*}
    \log|f_\ell(x_1+(n-1)\omega,\omega,E)|>\ell L_\ell(\omega,E)-\ell^{1-\tau/2}
  \end{equation*}
  for all $  n\in \bigcup_{j=1}^{q-1} \cN_{k(q-1)+j}
  $. By Proposition~\ref{prop:NDR}, with \eqref{eq:lacunary-sets},
  there exists $ j\in \{ 1,\ldots, q-1 \} $ such that
  \begin{equation*}
    \log|f_\ell(x_2+n\omega,\omega,E)|>\ell L_\ell(\omega,E)-\ell^{1-\tau/2}
  \end{equation*}
  for all $  n\in \cN_{k(q-1)+j} $. The conclusion holds with $ i=k(q-1)+j $.
\end{proof}

\section{Factorization under (NDR) Condition}\label{sec:NDRWegner}

In this section we obtain a local factorization for the
Dirichlet determinants with respect to the spectral variable,
via Weierstrass' Preparation Theorem. It is important that the size of the
polydisk on which the factorization holds
is not too small and that the degree of the polynomial is not too large.
With the aid of the  basic tools from \cref{sec:basic-tools}  the factorization at scale $ N $ can only
be obtained on a polydisk of radius $ \exp(-CN^{1-\tau}) $ and with
a polynomial of degree less than $ CN^{1-\tau} $. This is too weak for
our purposes. The main reason for considering (NDR) intervals is that
for them we can get a much better factorization, as is shown in \cref{prop:factorNDREvar}.

First, we show that if $ \Lambda $ is an (NDR) interval, then we have
good control on the number of zeroes of
$ f_\Lambda(z,w,\cdot) $ in a small disk around $ E_0 $.

\begin{lemma}\label{lem:evcount}
  Assume $ x_0\in \tor^d $, $ \omega_0\in \tor^d(a,b) $,
  $ E_0\in \R $, and $ L(\omega_0,E_0)> \gamma >0 $. Let $ K\ge 1 $,
  $ \ell\ge \ell_0(V,a,b,|E_0|,\gamma) $, $ r_0=\exp(-\ell) $.  Assume $ \Lambda $ is
  $ (K,\ell) $-(NDR) with respect to $ x_0,\omega_0,E_0 $.
  There exists $ r_0/2<r<r_0 $
  such that for any $ (z,w)\in \C^d\times \C^d $,
  \begin{equation}\label{eq:evcount-condition}
    |z-x_0|<\frac{c(V)r_0}{(K+1)^2},\ |w-\omega_0|<\frac{c(V)r_0}{|\Lambda|(K+1)^2},
  \end{equation}
  we have
  \begin{gather*}
	\#\{E\in \cD(E_0,r): f_\Lambda(z,w,E)=0\}\le K,\\
    \dist \left( \{E\in \C: f_\Lambda(z,w,E)=0\}, \{|E-E_0|=r\}\right)\ge \frac{r_0}{8(K+1)}.
  \end{gather*}
\end{lemma}
\begin{proof}
  Let $ \underline \Lambda $ be as in \cref{def:NDR} and $ \Lambda':=\Lambda\setminus\underline \Lambda $.
  Let $ I $
  be any of the connected components of $ \Lambda' $. By the definition of (NDR) intervals and the
  covering form of (LDT) (see \cref{lem:Greencoverap1}) we have
  \begin{equation*}
    \dist(E,\spec H_{I}(x_0,\omega_0))>\exp(-\ell)>0\text{ provided }|E-E_0|<\exp(-\ell).
  \end{equation*}
  Therefore, $ H_{\Lambda'}(x_0,\omega_0) $ has no eigenvalues in
  $ (E_0-r_0,E_0+r_0) $. Using
  \cref{lem:interlacing}, it follows that $ H_\Lambda(x_0,\omega_0) $ has at most $ |\underline \Lambda| $
  eigenvalues in
  $ (E_0-r_0,E_0+r_0) $. The condition \cref{eq:evcount-condition}  is such that
  \begin{equation*}
    \norm{H_{\Lambda}(z,w)-H_\Lambda(x_0,\omega_0)}\le C(V)(|z-x_0|+|\Lambda||w-\omega_0|)
    \le \frac{r_0}{32(K+1)^2}.
  \end{equation*}
  So, the conclusion follows by \cref{lem:evcount10}.
\end{proof}

Now we can obtain the main result of this section. Note that when we define polydisks, $ |\cdot| $ will stand
for the maximum norm.
\begin{prop}\label{prop:factorNDREvar}
  Assume $ x_0\in \tor^d $, $ \omega_0\in \tor^d(a,b) $,
  $ E_0\in \R $, and $ L(\omega_0,E_0)> \gamma >0 $. Let $ K\ge 1 $,
  $ \ell\ge \ell_0(V,a,b,|E_0|,\gamma) $, $ r_0=\exp(-\ell) $.  Assume $ \Lambda $ is
  $ (K,\ell) $-(NDR) with respect to $ x_0,\omega_0,E_0 $.
  There exist
  \[
    P(z,w,E) = E^k +a_{k-1} (z,w) E^{k-1} + \cdots + a_0
    (z,w)
  \]
  with $a_j$ analytic in the polydisk
  \[
    \cP:= \{|z-x_0|<c(V)r_0 (K+1)^{-2},
    |w-\omega_0|<c(V)r_0 |\Lambda|^{-1}(K+1)^{-2} \},
  \]
  and an analytic function $g(z,w,E)$, on $ \cP_1 \times \cD(E_0,r) $, $ r_0/2<r<r_0 $, such that:
  \begin{enumerate}
  \item[(a)] $f_\La(z, w,E) = P(z,w,E) g(z,w,E)$ for
    any $(z,w,E)\in \cP \times \cD(E_0,r)$,

  \item[(b)] $g(z,w,E) \ne 0$ for any
    $(z,w,E) \in \cP \times \cD(E_0,r)$,

  \item[(c)] for any $(z,w) \in \cP$, $P(z,w,\cdot)$ has
    no zeros in $\IC \setminus \cD(E_0, r)$,

  \item[(d)] $k\le K$,

  \item[(e)] if $ (x,\omega)\in \cP\cap (\T^{d}\times \T^d(a,b)) $, $ E\in \cD(E_0,r) $, and
    $ \frac{r_0}{16(K+1)}\ge \exp(-|\Lambda|^{\sigma/2}) $,
    then
    \begin{equation}\label{eq:g-estimate}
      \log \bigl | g(x,\omega,E)\bigr|> |\Lambda|L_{|\Lambda|}(\omega,E)- |\Lambda|^{1-\tau/2}.
    \end{equation}
  \end{enumerate}
\end{prop}
\begin{proof} Due to Lemma~\ref{lem:evcount} all conditions needed
  for Lemma~\ref{lem:weier} hold.  This implies the statements (a)-(d).  We just need to verify
  (e).   Note that $ P(x,\omega,E) $ is a product of factors $ E-E_j^{\Lambda}(x,\omega) $, with
  $ E_j^{\Lambda}(x,\omega)\in (E_0-r,E_0+r) $. So, for $ E\in \cD(E_0,r) $ we have
  $ |P(x,\omega,E)|\le (2r)^k<1 $. It follows that
  \begin{equation*}
    \log|g(x,\omega,E)|>\log|f_\Lambda(x,\omega,E)|
  \end{equation*}
  for any $ E\in \cD(E_0,r) $.
  Let $ r-\frac{r_0}{16(K+1)}<r'<r $. By \cref{lem:evcount}, for any $ |E-E_0|=r' $ we have
  \begin{equation*}
    \dist(E,\spec H_N(x,\omega))\ge \frac{r_0}{16(K+1)}\ge \exp(-|\Lambda|^{\sigma/2})
  \end{equation*}
  and hence, using \cref{cor:4.6zeros},
  \begin{equation*}
    \log|f_\Lambda(x,\omega,E)|> |\Lambda|L_{|\Lambda|}(\omega,E)-|\Lambda|^{1-\tau/2}
  \end{equation*}
  and \cref{eq:g-estimate} holds.
  Since the left-hand side of \cref{eq:g-estimate} is harmonic (in $ E $) and the right-hand side is subharmonic,
  it follows that the estimate holds for all $ E\in \cD(E_0,r') $.
  Since this is true for $ r' $ arbitrarily close to $ r $, the conclusion follows.

\end{proof}


\section{Elimination of Double Resonances Using Semialgebraic Sets}\label{sec:elim-SA}


In this section we obtain a result on elimination of double resonances using semialgebraic sets. This
result cannot yield the finite scale localization by itself. Instead, it will only serve as a catalyst for
the sharper result on elimination of double resonances between (NDR) intervals from \cref{sec:elim-NDR}.
The basis for elimination via semialgebraic sets is the following result, see \cite[Prop. 5.1]{BouGolSch02},
\cite[Lem. 9.9]{Bou05}, \cite[(1.5)]{Bou07}.
\begin{lemma}\label{lem:Bourgain-slopes}
  Let $ \cS \subset [0,1]^{d=d_1+d_2} $ be a semialgebraic set of
  degree $ B $ and $ \mes_{d} \cS <\eta $,
  $ \log B \ll \log \frac{1}{\eta} $. We denote
  $ (x,\omega)\in [0,1]^{d_1}\times [0,1]^{d_2} $ the product
  variable.  Fix $ \epsilon > \eta^{\frac{1}{d}} $. Then there is a
  decomposition
  \begin{equation*}
    \cS= \cS_1 \cup \cS_2
  \end{equation*}
  $ \cS_1 $ satisfying
  \begin{equation*}
    \mes_{d_2} (\Proj_\omega \cS_1)<B^C\epsilon
  \end{equation*}
  and $ \cS_2 $ satisfying the transversality property
  \begin{equation*}
    \mes_{d_1}(\cS_2\cap L)<B^C\epsilon^{-1}\eta^{\frac{1}{d}}
  \end{equation*}
  for any $ d_1 $-dimensional hyperplane $ L $ s.t.
  $ \max_{1\le j\le d_2 } |\Proj_L(e_j)|<\frac{1}{100}\epsilon $ (we
  denote by $ e_1,\ldots,e_{d_2} $ the $ \omega $-coordinate
  vectors).
\end{lemma}


We can now prove our result on semialgebraic elimination. In the sections to follow we will use $ \oell $ to denote
the size of a $ (K,\ell) $-(NDR) interval. We also use the $ \oell $ notation in this section because we will only
apply its results to (NDR) intervals.
\begin{prop}\label{prop:double-resonance-elimination}
  Let $ \alpha\in(0,1) $, $ \gamma>0 $, $ \overline \ell \ge 1 $, and $ \sigma $ as in
  (LDT). Let $ \Lambda_0,\Lambda_1 $ be
  intervals in $ \Z $ such that $ \overline \ell/10 \le |\Lambda_0|,|\Lambda_1|\le 10\overline \ell $ and
  $ 0\in \Lambda_0,\Lambda_1 $. For any $ x_0\in \T^d $,
  $ \overline \ell\ge \overline \ell_0(V,a,b,\alpha,\gamma)$,
  $ 1\le t_0\le \exp(c_0\overline \ell^{\alpha\sigma}) $,
  there exists a set
  $ \Omega_{\overline \ell,t_0,x_0}  $, $ \mes(\Omega_{\overline \ell,t_0,x_0})< \overline\ell^{C(a,b)}/t_0 $
  such that the following holds.
  For all $ \omega\in \tor^d(a,b)\setminus \Omega_{\overline \ell,t_0,x_0} $,
  $ E\in \R $, such that $ L(\omega,E)>\gamma $,
  we have
  \begin{equation*}
    \min \left( \norm{(H_{\Lambda_0}(x_0,\omega)-E)^{-1}}, \norm{(H_{\Lambda_1}(x_0+t\omega,\omega)-E)^{-1}}\right)
    \le \exp(\,\overline \ell^\alpha)
  \end{equation*}
  for all $ t_0\le |t|\le \exp(c_0\overline \ell^{\alpha\sigma}) $.
\end{prop}
\begin{proof}
  Fix $ x_0\in \T^d $. Let $ \cB $ be the set of
  $ (x,\omega,E)\in \tor^d\times \tor^d(a,b)\times \R $,
  such that $ L(\omega,E)>\gamma $ and
  \begin{equation*}
    \norm{(H_{\Lambda_0}(x_0,\omega)-E)^{-1}} > \exp(\, \overline \ell^{\alpha}),\text{ and }
    \norm{(H_{\Lambda_1}(x,\omega)-E)^{-1}} > \exp(\,\overline \ell^{\alpha}).
  \end{equation*}
  Let $ \tilde V $ be as in \cref{eq:V-tilde}.  Then $ \cB $ is
  contained in the semialgebraic set $ \tilde \cS $ of
  \begin{equation*}
    (x,\omega,E)\in \tor^d\times\tor^d_{J}(a,b)
    \times [-C(V),C(V)],\ J=\overline \ell^{C(a,b)}
  \end{equation*}
  satisfying
  \begin{gather*}
    \frac{1}{|\Lambda_1|J}\sum_{j=1}^J\log \normhs{\tilde M_{|\Lambda_1|}(x+j\omega,\omega,E)}\ge \frac{\gamma}{2}\\
    \norm{(\tilde H_{\Lambda_0}(x_0,\omega)-E)^{-1}}_{\text{HS}}
    \ge \frac{1}{2}\exp(\,\overline \ell^{\alpha}),\text{ and }
    \norm{(\tilde H_{\Lambda_1}(x,\omega)-E)^{-1}}_{\text{HS}} \ge \frac{1}{2}\exp(\,\overline \ell^{\alpha}).
  \end{gather*}
  We used \cref{eq:E-tilde} (cf. proof of \cref{lem:semialgebraic-all}).
  Clearly the degree of $ \tilde \cS $ is less than $ \overline \ell^C $.
  Furthermore, for $ (x,\omega,E)\in \tilde \cS $ we have $ L_{|\Lambda_1|}(\omega,E)\ge \gamma/4 $,
  \begin{equation*}
    \norm{(H_{\Lambda_0}(x_0,\omega)-E)^{-1}} > \frac{1}{4}\exp(\, \overline \ell^{\alpha}),\text{ and }
    \norm{(H_{\Lambda_1}(x,\omega)-E)^{-1}} > \frac{1}{4}\exp(\,\overline \ell^{\alpha}).
  \end{equation*}
  Applying the Wegner estimate from Proposition~\ref{prop:Wegner} (recall \cref{rem:scale-N-sa-refinement})
  to $ H_{\Lambda_1}(x,\omega) $ with $ E_0\in \spec H_{\Lambda_0}(x_0,\omega) $
  we get
  \begin{equation*}
    \mes(\Proj_{(x,\omega)}\tilde \cS)< \exp(-c\overline \ell^{\alpha\sigma}).
  \end{equation*}
  Set $ \cS:=\Proj_{(x,\omega)}\tilde \cS $.  Let
  $ \cS=\cS_1\cup\cS_2 $ be the decomposition of $ \cS $ afforded by
  Lemma \ref{lem:Bourgain-slopes} with
  $ \epsilon=200/t_0 $. Note that to get the conclusion we just need that
  \begin{equation*}
    \omega\notin \{ \omega : (\{ x_0+t\omega  \},\omega)\in \cS   \}.
  \end{equation*}
  So, we let
  \begin{gather*}
    \Omega_{\overline \ell,t_0,x_0}(\Lambda_0,\Lambda_1)=\Proj_\omega \cS_1\cup
    \left( \bigcup_t \Omega_{\overline \ell,t,x_0}(\Lambda_0,\Lambda_1)\right),\\
    \Omega_{\overline \ell,t,x_0}(\Lambda_0,\Lambda_1):=\{ \omega : (\{ x_0+t\omega \},\omega)\in \cS_2 \},\\
    \Omega_{\overline \ell,t_0,x_0}
    =\bigcup_{\Lambda_0,\Lambda_1} \Omega_{\overline \ell,t_0,x_0}(\Lambda_0,\Lambda_1).
  \end{gather*}
  Note that $ \Proj_\omega \cS_1\supset\{ \omega : (\{ x_0+t\omega  \},\omega)\in \cS_1   \} $ and due to our
  assumptions there are less than $ C\overline \ell ^4 $ possible choices for $ \Lambda_0,\Lambda_1 $.
  We just need to check the estimate on the measure of $ \Omega_{\overline \ell,t_0,x_0} $.
  To this end, note that the set of $ \{x_0+t\omega\} $, with
  $ \omega\in[0,1]^d $, is contained in a union of hyperplanes
  $ L_i $, $ i\le |t|^d $. The hyperplanes $ L_i $
  are parallel to $ (t\omega,\omega) $, $ \omega\in\R^d $, and
  therefore
  \begin{equation*}
    |\Proj_{L_i} e_j|\le \frac{1}{|t|}\le \frac{1}{t_0}<\frac{\epsilon}{100} \text{ for all }  i,j
  \end{equation*}
  ($ e_j $ are as in Lemma \ref{lem:Bourgain-slopes}).  Then by
  Lemma~\ref{lem:Bourgain-slopes},
  \begin{gather*}
    \mes(\Proj_\omega \cS_1)<\overline \ell^C/t_0,\\
	\mes(\Omega_{\overline \ell,t,x_0}(\Lambda_0,\Lambda_1))= \sum_i \mes(\cS_2\cap L_i)
    \lesssim |t|^{d} \overline \ell ^C t_0 \exp(-c\overline \ell^{\alpha\sigma})
    \le \exp(-c'\overline\ell^{\alpha\sigma}),
  \end{gather*}
  and the conclusion follows.
\end{proof}

For the purposes of \cref{sec:elim-NDR} it
will be convenient to eliminate the phase variable from the set of resonant frequencies from
\cref{prop:double-resonance-elimination}.
\begin{cor}\label{cor:double-resonance-elimination}
  We use the same assumptions and notation as in \cref{prop:double-resonance-elimination}.
  For any
  $ \overline \ell\ge \overline\ell_0(V,a,b,\alpha,\gamma)$ and
  $ 1\le t_0\le \exp(c_0\overline \ell^{\alpha\sigma}) $, there exists a set
  $ \Omega_{\overline \ell,t_0} $, $ \mes(\Omega_{\overline \ell,t_0})<\overline \ell^{C(a,b)} t_0^{-\frac{1}{2}} $,
  such that for any $ \omega\notin \Omega_{\overline \ell,t_0} $ there exists a set
  $ \cB_{\overline \ell,t_0,\omega} $,
  $ \mes(\cB_{\overline \ell,t_0,\omega})<\overline \ell^{C(a,b)} t_0^{-\frac{1}{2}} $,
  and the following holds.
  For any $ \omega\in  \tor^d(a,b)\setminus \Omega_{\overline \ell,t_0} $,
  $ x\in \T^d\setminus \cB_{\overline \ell,t_0,\omega} $, $ E\in \R $, such that $ L(\omega,E)>\gamma $,
  we have
  \begin{equation*}
    \min \left( \norm{(H_{\Lambda_0}(x,\omega)-E)^{-1}}, \norm{(H_{\Lambda_1}(x+t\omega,\omega)-E)^{-1}}\right)
    \le \exp(\,\overline \ell^\alpha)
  \end{equation*}
  for all $ t_0\le |t|\le \exp(c_0\overline \ell^{\alpha\sigma}) $.
\end{cor}
\begin{proof}
  Let $ \cB_{\overline\ell,t_0} $ be the set of $ x\in \T^d $, $ \omega\in \T^d(a,b) $
  such that $ \omega\in \Omega_{\overline \ell,t_0,x} $, with $ \Omega_{\overline \ell,t_0,x} $ as in
  \cref{prop:double-resonance-elimination}. Using Chebyshev's inequality we get that there exists a set
  $ \Omega_{\overline \ell,t_0} $,
  $ \mes(\Omega_{\overline \ell,t_0})\le {\mes(\cB_{\overline\ell,t_0})}^{\frac{1}{2}} $
  such that for $ \omega\notin \Omega_{\overline \ell,t_0}$ we have
  $ \mes(\cB_{\overline \ell,t_0,\omega})\le {\mes(\cB_{\overline\ell,t_0})}^{\frac{1}{2}} $, where
  \begin{equation*}
    \cB_{\overline \ell,t_0,\omega}= \{ x : (x,\omega)\in \cB_{\overline\ell,t_0} \}.
  \end{equation*}
  Now the conclusion follows by \cref{prop:double-resonance-elimination}.
\end{proof}

\section{Elimination of Double Resonances under (NDR) Condition}
\label{sec:elim-NDR}

To eliminate double resonances between (NDR) intervals we combine the elimination from
\cref{cor:double-resonance-elimination}  with another tool,
resultants of polynomials (see \cref{sec:resultants}). The biggest problem with the estimate in
\cref{cor:double-resonance-elimination}
is that, due to the weakness of the measure
estimate for the set of resonant phases, we cannot apply it simultaneously to
$ x_0+n\omega $, $ n\in [1,N] $, as is needed for
finite scale localization.
In the next lemma we will obtain a sharper (local) measure estimate via Cartan's estimate applied to the
resultant of the polynomials from the Weierstrass  preparation theorem in the $E$-variable, under
the  (NDR) condition, see  Proposition~\ref{prop:factorNDREvar}.
For Cartan's estimate to be effective, we need a point at which we have a good lower bound on
the modulus of the resultant. This point will come from \cref{cor:double-resonance-elimination}.

\begin{lemma}\label{lem:eliminationa}
  Assume $ x_0\in \T^d $, $ \omega_0\in \T^d(a,b) $, $ E_0\in \R $, and $ L(\omega_0,E_0)>\gamma>0 $.
  Let $ \sigma,\tau $ as in (LDT) and
  \begin{equation*}
    1\le \ell\le \overline \ell \le \exp(\ell),\
    1\le K \le \overline \ell^{\frac{1-\tau}{10}},\quad
    \exp(\ell^2)\le |t|\le \exp(\overline \ell^{\frac{\sigma(1-\tau)}{4}}).
  \end{equation*}
  Let $ \Lambda_j $, $ j=0,1 $, be
  intervals in $ \Z $ such that $ \overline \ell/10 \le |\Lambda_j|\le 10\overline \ell $,
  $ 0\in \Lambda_j $, and $ \Lambda_j $ is $ (K,\ell) $-(NDR) with respect to $ x_j,\omega_0,E_0 $,
  $ x_j:=x_0+jt\omega_0 $, $ j=0,1 $.

  For any $ \ell\ge \ell_0(V,a,b,|E_0|,\gamma) $ there exists a set
  $ \Omega_{\overline \ell} $, $ \mes(\Omega_{\overline \ell})<\exp(-\ell^2/4) $, such that
  if $ \omega_0\notin \Omega_{\overline \ell} $ the following holds.
  There exists a set
  $\cB_{\overline \ell,t}=\cB_{\overline \ell,t}(x_0,\omega_0,E_0)$, such that
  \begin{equation*}
    \mes (\cB_{\overline \ell,t})
    <\exp(-\overline \ell^{\frac{1-\tau}{6d}})/|t|^d,
  \end{equation*}
  and for any $ (x,\omega)\in (\T^d\times\T^d(a,b))\setminus \cB_{\overline \ell,t} $, $ E\in \R $,
  \begin{equation}\label{eq:elimination1}
    \begin{split}
      |x-x_0|<\exp(-4\ell),\quad |\omega-\omega_0|<
      \exp(-4\ell)/|t|,\quad |E-E_0|<\exp(-\ell)/2,
    \end{split}
  \end{equation}
  we have
  \begin{equation}\label{eq:elimination2}
    \max_{j=0,1}\big(\log
    |f_{\La_j}(x+jt\omega,\omega,E)|-|\Lambda_j|L_{|\Lambda_j|}(\omega,E)+2|\Lambda_j|^{1-\tau/2}\big)>0.
  \end{equation}
\end{lemma}
\begin{proof}
  We start by extracting the information we need from the semialgebraic elimination.
  By \cref{prop:log-Hoelder} we can guarantee that $ L(\omega_0,E)>\frac{\gamma}{2} $ for
  any $ |E-E_0|<\exp(-\ell) $ (provided $ \ell $ is large enough).
  Let  $ \Omega_{\overline \ell,t_0} $, $ \cB_{\overline \ell,t_0,\omega_0} $ be the sets
  from \cref{cor:double-resonance-elimination}, with $ \alpha=(1-\tau)/3 $, $ t_0=\exp(\ell^2) $,
  and $ \gamma/2 $ instead of $ \gamma $.
  Since we are assuming $ \overline \ell \le \exp(\ell) $ it follows that
  \begin{equation*}
    \mes(\Omega_{\overline \ell,t_0}), \mes(\cB_{\overline \ell,t_0,\omega_0})<\exp(-\ell^2/4).
  \end{equation*}
  We set $ \Omega_{\overline \ell}:= \Omega_{\overline \ell,t_0}$ and we assume
  $ \omega_0\notin \Omega_{\overline \ell} $. Let $ x_0'\in \T^d $, $ |x_0'-x_0|<\exp(-\ell^2/(5d)) $
  be such that $ x_0'\notin \cB_{\overline \ell,t_0,\omega_0} $. Then we have
  \begin{equation}\label{eq:reference-point}
    \min \left( \norm{(H_{\Lambda_0}(x_0',\omega_0)-E)^{-1}},
      \norm{(H_{\Lambda_1}(x_0'+t\omega_0,\omega_0)-E)^{-1}}\right)
    \le \exp(\,\overline \ell^\frac{1-\tau}{3})
  \end{equation}
  for all $ |E-E_0|<\exp(-\ell) $
  and $ \exp(\ell^2)\le |t|\le \exp(c\overline \ell^{\sigma(1-\tau)/3} ) $. In particular this holds for $ t $
  satisfying our assumptions.

  Next we collect the relevant facts about the factorization of $ f_{\Lambda_j} $.
  Let $ f_{\Lambda_j}=P_jg_j $, $ j=0,1 $, be the factorizations from Proposition~\ref{prop:factorNDREvar}.
  We have that
  \[
    P_j(z,\omega,E) = E^{k_j} +a_{j,k_j-1} (z,\omega) E^{k_j-1} +
    \cdots + a_{j,0} (z,\omega),\ k_j\le K
  \]
  with $a_{j,i}$ analytic on a polydisk containing
  \[
    \cP_j:=\{ (z,w)\in \IC^{2d}: |z-x_j|<\exp(-2\ell),  |w-\omega_0|<\exp(-3\ell)\},
  \]
  and all the zeros of $ P_j(z,w,\cdot) $, $ (z,w)\in \cP_j $, are contained in $ \cD(E_0,\exp(-\ell)) $
  (we used the assumptions that $ K\le \overline \ell^{\frac{1-\tau}{10}} $, $ \overline \ell\le \exp(\ell) $,
  $ |\Lambda_j|\le 10\overline \ell $).
  We also have
  \begin{equation}\label{eq:g_j-lb}
    \log|g_j(x,\omega,E)|>|\Lambda_j|L_{|\Lambda_j|}(\omega,E)-|\Lambda_j|^{1-\tau/2}
  \end{equation}
  for any $ (x,\omega)\in \cP_j\cap (\T^d\times \T^d(a,b)) $ and $ E\in \cD(E_0,\exp(-\ell)/2) $.
  Note that the condition $ \frac{r_0}{16(K+1)}\ge\exp(-|\Lambda_j|^{\sigma/2}) $ needed for part (e) of
  Proposition~\ref{prop:factorNDREvar} is satisfied because our restrictions on $ t $ imply
  $ \overline \ell\ge \ell^{\frac{8}{\sigma(1-\tau)}} $.

  Now we can proceed with the elimination via resultants and Cartan's estimate (recall \cref{sec:Cartan} and
  \cref{sec:resultants}). Set
  \[
    R(z,w)=\Res(P_0(z,w,\cdot),P_1(z+tw,w,\cdot)).
  \]
  Then $ R $ is well-defined on the polydisk
  \begin{equation*}
    \cP:= \{ (z,w)\in \IC^{2d} : |z-x_0'|<\exp(-3\ell), |w-\omega_0|< \exp(-3\ell)/|t| \}.
  \end{equation*}
  By the definition of the resultant (see \eqref{eq:resdef}) and the properties of $ P_j $, we have
  \begin{equation*}
    R(z,w)=\prod (E_i^{\Lambda_0}(z,w)-E_j^{\Lambda_1}(z+tw,w)),
  \end{equation*}
  where the product is only in terms of eigenvalues contained in $ \cD(E_0,\exp(-\ell)) $. This
  implies
  \begin{equation*}
    \sup|R(z,w)|\le (2\exp(-\ell))^{k_1k_2}<1.
  \end{equation*}
  By \cref{eq:reference-point} applied to
  \begin{equation*}
    E\in \left( \spec H_{\Lambda_0}(x_0',\omega_0)\cup \spec
      H_{\Lambda_1}(x_0'+t\omega_0,\omega_0) \right)\cap \cD(E_0,\exp(-\ell))
  \end{equation*}
  we get that
  \begin{equation*}
    \left| E_i^{\Lambda_0}(x_0',\omega_0)-E_j^{\Lambda_1}(x_0'+t\omega_0,\omega_0) \right|
    \ge \exp(-\overline \ell^{\frac{1-\tau}{3}})
  \end{equation*}
  for the pairs of eigenvalues contained in $ \cD(E_0,\exp(-\ell)) $. It follows that
  \begin{equation*}
    |R(x_0',\omega_0)|\ge \exp(-K^2 \overline\ell^{\frac{1-\tau}{3}}).
  \end{equation*}
  Using Cartan's estimate (recall \cref{lem:high_cart}) with
  \begin{equation*}
    m=-K^2 \overline\ell^{\frac{1-\tau}{3}},\ M=0,\ H=\overline \ell^{\frac{1-\tau}{3}}
  \end{equation*}
  we have that
  \begin{equation*}
    |R(x,\omega)|>\exp(-CK^2 \overline\ell^{\frac{2(1-\tau)}{3}})
  \end{equation*}
  for all $ (x,\omega)\in (6^{-1}\cP\cap \R^{2d})\setminus \cB $ with, $ \cB=\cB(\Lambda_0,\Lambda_1) $,
  \begin{equation*}
    \mes(\cB) \le \frac{C\exp(-6d\ell)}{|t|^d}\exp(-H^{\frac{1}{2d}}).
  \end{equation*}
  Note that $ 6^{-1}\cP\cap \R^{2d} $ contains the points $ (x,\omega) $ satisfying \cref{eq:elimination1}.
  We let $ \cB_{\overline \ell,t}=\cup \cB(\Lambda_0,\Lambda_1) $.
  Finally, Lemma~\ref{lem:resbasicprop1} implies
  \begin{equation*}
    \max(|P_0(x,\omega,E)|,|P_1(x+t\omega,\omega,E)|)\ge \exp(-C K^3\overline\ell^{\frac{2(1-\tau)}{3}})
    >\exp(-\overline\ell^{1-\tau}),
  \end{equation*}
  for all $ E $ and $ (x,\omega)\notin \cB_{\overline \ell,t} $,
  and the conclusion follows using the factorization of the determinants and
  \eqref{eq:g_j-lb}.
\end{proof}



We can now use a covering argument to prove a global version of the previous result.

\begin{prop}\label{prop:NDR-elimination}
  Let $ \gamma>0 $, $ \sigma,\tau $ as in (LDT), and
  \begin{equation*}
    1\le \ell\le \overline \ell \le \exp(\ell),\
    1\le K \le \overline \ell^{\frac{1-\tau}{10}},\
    \exp(\ell^2)\le N.
  \end{equation*}
  Let $ \Lambda_j $, $ j=0,1 $, be
  intervals in $ \Z $ such that $ \overline \ell/10 \le |\Lambda_j|\le 10\overline \ell $,
  $ 0\in \Lambda_j $, $ j=0,1 $.

  For $ \ell\ge \ell_0(V,a,b,\gamma) $ there exists $ \alpha=\alpha(a,b) $ such that if
  $ N\le \exp(\overline \ell^\alpha) $,
  there exists a set $ \Omega_{N,\overline \ell} $, $ \mes(\Omega_{N,\overline \ell})<2\exp(-\ell^2/4) $,
  and for every
  $ \omega\in \tor^d(a,b)\setminus \Omega_{N,\overline \ell} $ there exists a set
  $ \cB_{N,\overline \ell,\omega} $, $ \mes(\cB_{N,\overline \ell,\omega})<\exp(-\overline \ell^\alpha) $
  such that the following statement holds.  For any
  $ \omega\in \tor^d(a,b)\setminus \Omega_{N,\overline \ell} $,
  $ x\in \tor^d\setminus \cB_{N,\overline \ell,\omega} $, $ E\in \R $, and
  $ m_0,m_1\in [1,N] $ such that $ |m_0-m_1|\ge \exp(\ell^2) $, we
  have that if $ L(\omega,E)>\gamma $ and $ \Lambda_j $, $ j=0,1 $, are $ (K,\ell,1/2) $-(NDR) intervals with
  respect to $ x+m_j\omega,\omega,E $,
  then
  \begin{equation}\notag
    \max_{j=0,1}\big(\log
    |f_{\La_j}(x+m_j\omega,\omega,E)|-|\Lambda_j|L_{|\Lambda_j|}(\omega,E)+2|\Lambda_j|^{1-\tau/2}\big)>0.
  \end{equation}
\end{prop}
\begin{proof}
  By \cref{lem:LDT-fail}, if the conclusion does not hold, then $ E $ is in a neighborhood
  of
  \begin{equation*}
    \spec H_{\Lambda_0}(x+m_0\omega,\omega)\cup\spec H_{\Lambda_1}(x+m_1\omega_1,\omega_1).
  \end{equation*}
  Therefore, it is enough to prove the result with $ E\in (-C(V),C(V)) $.

  Let $ \Omega_{\overline \ell} $ be the set from Lemma~\ref{lem:eliminationa}
  (with $ \gamma/2 $ instead of $ \gamma $). Let
  \begin{gather*}
    \{ z : |z-x_k|<\exp(-\ell^{3/2}) \},\ k\lesssim \exp(d\ell^{3/2})\\
    \{ \omega : |\omega-\omega_{k'}|<\exp(-\ell^{3/2})/N \},\ k'\lesssim N^d\exp(d\ell^{3/2})\\
    \{ E : |E-E_{k''}|<\exp(-\ell^{3/2}) \},\ k''\lesssim
    C(V)\exp(\ell^{3/2})
  \end{gather*}
  be covers of $ \tor^d $, $ \tor^d(a,b)\setminus \Omega_{\overline \ell} $, and $ (-C(V),C(V)) $,
  respectively. Note that if $ |\omega-\omega_{k'}|<\exp(-\ell^{3/2})/N $, $ |E-E_{k''}|<\exp(-\ell^{3/2}) $, and
  $ L(\omega,E)>\gamma>0 $, then, by \cref{prop:log-Hoelder}, $ L(\omega_{k'},E_{k''})>\gamma/2 $, provided
  $ \ell $ is large enough. So, if $ L(\omega_{k'},E_{k''})>\gamma/2 $ we let
  $ \cB_{\overline \ell,t}(x_k,\omega_{k'},E_{k''}) $   be the set from Lemma~\ref{lem:eliminationa}
  (with $ \gamma/2 $ instead of $ \gamma $). Otherwise we let
  $ \cB_{\overline \ell,t}(x_k,\omega_{k'},E_{k''}) $ be the empty set. Let
  \begin{equation*}
    \cB_{N,\overline \ell}= \bigcup_{k,k',k'',t,m} S_m(\cB_{\overline \ell,t}(x_k,\omega_{k'},E_{k''})),
  \end{equation*}
  where $ S_m(x,\omega)=(x-m\omega,\omega) $, $ m\in [1,N] $, $ |t|\in [\exp(\ell^2),N] $.
  Then we have
  \begin{equation*}
    \mes (\cB_{N,\overline \ell})
    \le CN^{d+2}\exp((2d+1)\ell^{3/2})\exp(-\overline \ell^{\frac{\sigma(1-\tau)}{6d}}).
  \end{equation*}
  Note that, due to the restriction required on $ |t| $ by \cref{lem:eliminationa} we need to have
  $ N\le \exp(\overline \ell^{\frac{\sigma(1-\tau)}{4}})$. We choose $ \alpha=\alpha(\sigma,\tau)=\alpha(a,b) $
  such that if $ N\le \exp(\overline \ell^{\alpha}) $, the previous restriction is satisfied and
  \begin{equation*}
    \mes(\cB_{N,\overline \ell})<\exp(-2\overline \ell^{\alpha}).
  \end{equation*}
  Using Chebyshev's inequality we get that there exists a set
  $ \tilde \Omega_{N,\overline \ell} $,
  $ \mes(\tilde \Omega_{N,\overline \ell})<\exp(-\overline \ell^{\alpha})  $
  such that for $ \omega\notin \tilde \Omega_{N,\overline \ell}$ we have
  $ \mes(\cB_{N,\overline \ell,\omega})<\exp(-\overline \ell^{\alpha}) $, where
  \begin{equation*}
    \cB_{N,\overline \ell,\omega}= \{ x : (x,\omega)\in \cB_{N,\overline \ell} \}.
  \end{equation*}
  We set $ \Omega_{N,\overline \ell}=\Omega_{\overline \ell}\cup \tilde \Omega_{N,\overline \ell} $
  and we claim that the conclusion follows with our choice of sets.
  Indeed, let $ x,\omega,E $, $ m_j $, $ \Lambda_j $ as in the assumptions. There exist $ x_k,\omega_{k'} $,
  $ E_{k''} $ such that
  \begin{equation*}
    |x+m_0\omega-x_k|<\exp(-\ell^{3/2}), \ |\omega-\omega_{k'}|<\exp(-\ell^{3/2})/N,\ |E-E_{k''}|<\exp(-\ell^{3/2})
  \end{equation*}
  (recall that there's nothing to check if $ E\notin (-C(V),C(V)) $). Since $ L(\omega,E)>\gamma $, we
  have $ L(\omega_{k'},E_{k''})>\gamma/2 $, and since $ \Lambda_j $, $ j=0,1 $, are $ (K,\ell,1/2) $-(NDR) with
  respect to $ x+m_j\omega,\omega,E $, it follows (recall \cref{rem:NDR-stability}) that
  $ \Lambda_j $, $ j=0,1 $, are $ (K,\ell) $-(NDR) with respect to
  \begin{equation*}
    x_k+jt\omega_{k'},\omega_{k'},E_{k''},\ t=m_1-m_0.
  \end{equation*}
  The choice of our exceptional sets guarantees
  $ (x+m_0\omega,\omega)\notin \cB_{\overline \ell,t}(x_k,\omega_{k'},E_{k''}) $ and
  the conclusion follows from \cref{lem:eliminationa}.
\end{proof}

\section{Finite and Full Scale Localization: Proofs of Theorems A, B, C}\label{sec:localization}

Combining elimination of resonances under the (NDR) condition with the covering form of (LDT) yields the
following result on elimination of resonances at a given scale, as needed for obtaining \cref{thm:A} via
\cref{prop:NDRloc1}.
\begin{lemma}\label{lem:elimination}
  Let $ \epsilon\in(0,1/5) $, $ \gamma>0 $, and $ \sigma,\tau $ as in (LDT).
  For $ N\ge N_0(V,a,b,\gamma,\epsilon) $, there exists
  a set $ \Omega_N $,
  \begin{equation*}
    \mes(\Omega_N)<\exp(-(\log N)^{\epsilon\sigma}),
  \end{equation*}
  such that for
  $ \omega\in \T^d(a,b)\setminus \Omega_N  $ there exists a set $ \cB_{N,\omega} $,
  \begin{equation*}
    \mes(\cB_{N,\omega})<\exp(-\exp((\log N)^{\epsilon\sigma}))
  \end{equation*}
  and the following holds
  for any $ \omega\in \T^d(a,b)\setminus \Omega_N $, $ x\in \T^d\setminus \cB_{N,\omega} $,
  and  $ E\in \R $ such that $ L(\omega,E)>\gamma $.
  If $ m_0\in [1,N] $ is such that
  \begin{equation*}
    \log|f_\Lambda(x,\omega,E)|\le |\Lambda|L_{|\Lambda|}(\omega,E)-|\Lambda|^{1-\tau/4}
  \end{equation*}
  for all intervals $ \Lambda\subset [1,N] $ satisfying $     \dist(m_0,[1,N]\setminus \Lambda)>|\Lambda|/100 $,
  $ |\Lambda|\ge (\log N)^{\epsilon} $,
  then for any $ \oell\ge \exp((\log N)^{3\epsilon\sigma/2}) $ we have
  \begin{equation}\label{eq:no-long-range-resonance}
    \log |f_{\oell}(x+(m-1)\omega,\omega,E)|>\oell L_{\oell}(\omega,E)-\oell^{1-\tau/2}
  \end{equation}
  for any  $ m\in [1,N-\oell+1] $ such that $ |m-m_0|\ge \oell+2\exp(\ell^2) $,
  $ \ell=\lceil (\log N)^{2\epsilon}\rceil  $.
\end{lemma}
\begin{proof}
  With  $ q=2^{2d+1} $, let
  \begin{equation*}
    1\ll \underline  C_1\ll \overline C_1\ll \ldots \ll \underline C_{(q-1)^2}\ll \overline C_{(q-1)^2}
  \end{equation*}
  be constants such that the sets
  \begin{equation*}
    \cN_k= \{ n\in \Z : \exp(\underline C_k\ell^{\sigma/2})\le |n| \le \exp(\overline C_k \ell^{\sigma/2}) \}
  \end{equation*}
  satisfy the assumptions of Corollary~\ref{cor:same-scale-NDR}.
  Let
  \begin{equation*}
    \overline \ell_k=\lfloor\exp(\overline C_k \ell^{\sigma/2})\rfloor,
    \ K_k=\lceil\exp(\underline C_k \ell^{\sigma/2})\rceil.
  \end{equation*}
  We choose the constants $ \underline C_k,\overline C_k $ such that we also have
  $ K_k\le \overline \ell_k^{\frac{1-\tau}{10}} $ as needed for \cref{prop:NDR-elimination}.
  The conclusion will follow by
  choosing
  \begin{equation*}
    \Omega_N:=\Omega_{\ell} \cup \left( \bigcup_k \Omega_{N,\overline \ell_k} \right),\quad
    \cB_{N,\omega}:= \bigcup_k \cB_{N,\overline \ell_k,\omega},
  \end{equation*}
  where $ \Omega_{\ell} $ is the set from
  Corollary~\ref{cor:same-scale-NDR} and
  $ \Omega_{N,\overline \ell_k}, \cB_{N,\overline \ell_k,\omega}$ are the sets from
  Proposition~\ref{prop:NDR-elimination}.
  Note that
  \begin{gather*}
	\mes(\Omega_N)<\ell^{C}\exp(-\underline C_1 \ell^{1/2})+2(q-1)^2\exp(-\ell^2/4)<\exp(-\ell^{1/2})
    <\exp(-(\log N)^{\epsilon\sigma}),\\
    \mes(\cB_{N,\omega})<(q-1)^2\exp(-\overline \ell_1^{\alpha})<\exp(-\exp(\ell^{1/2}))<
    \exp(-\exp((\log N)^{\epsilon\sigma})),
  \end{gather*}
  provided we choose $ \underline C_1,\overline C_1 $ large enough.

  Let $ \Lambda_0=[m_0',m_0'']\subset [1,N] $, $ |\Lambda_0|=\ell $, be an interval such that
  $ \dist(m_0,[1,N]\setminus \Lambda_0)>|\Lambda_0|/100 $.
  Then by our assumptions
  \begin{equation*}
    \log|f_{\Lambda_0}(x,\omega,E)|=\log|f_{\ell}(x+(m_0'-1)\omega,\omega,E)|
    \le \ell L_\ell(\omega,E)-\ell^{1-\tau/2}.
  \end{equation*}
  By Corollary~\ref{cor:same-scale-NDR}, for each
  $ m'\in [1,N] $ we either have
  \begin{equation*}
    \log|f_\ell(x+(m'-1)\omega,\omega,E)|>\ell L_\ell(\omega,E)-\ell^{1-\tau/2}
  \end{equation*}
  or there exists $ k=k(m_0',m') $ such that $ [-\overline \ell_k,\overline \ell_k] $ is
  $ (K_k,\ell,1/2) $-(NDR) with
  respect to both $ x+(m_0'-1)\omega,\omega,E $ and $ x+(m'-1)\omega,\omega,E $.
  By our assumptions we must also have
  \begin{equation*}
    \log|f_\Lambda(x+(m_0'-1)\omega,\omega,E)|
    \le |\Lambda|L_{|\Lambda|}(\omega,E)-2|\Lambda|^{1-\tau/2},
  \end{equation*}
  where $ \Lambda $ is any of the intervals
  $ [-\overline \ell_k,\overline \ell_k]\cap([1,N]-m_0'+1) $.
  Note that by the definition of (NDR) intervals, if $ [-\overline \ell_k,\overline \ell_k] $ is
  (NDR) with respect to $ x+(m_0'-1)\omega,\omega,E $, then
  $ [-\overline \ell_k,\overline \ell_k]\cap([1,N]-m_0'+1) $ is also (NDR) with respect to
  $ x+(m_0'-1)\omega,\omega,E $. It
  follows from Proposition~\ref{prop:NDR-elimination} that for any
  $ m'\in [1,N] $ such that $ |m'-m_0'|\ge \exp(\ell^2) $, we have
  \begin{equation}\label{eq:lambda-m-lde}
    \log|f_{\Lambda(m')}(x+(m'-1)\omega,\omega,E)|
    > |\Lambda(m')|L_{|\Lambda(m')|}(\omega,E)-2|\Lambda(m')|^{1-\tau/2},
  \end{equation}
  where $ \Lambda(m') $ is either $ [1,\ell] $ or
  $ [-\overline \ell_k,\overline \ell_k] $, with $ k=k(m_0',m') $ as above.
  In addition, \cref{prop:NDR-elimination} guarantees that if
  $ \Lambda(m')=[-\overline \ell_k,\overline \ell_k] $, then  \eqref{eq:lambda-m-lde} also
  holds for any subintervals of $ [-\overline \ell_k,\overline \ell_k] $ with length
  greater than $   \overline \ell_k/10 $.

  Take $ m\in [1,N-\oell+1] $ such that $ |m-m_0|\ge \oell+2\exp(\ell^2) $.
  This choice of $ m $ guarantees that
  \begin{equation*}
    [m,m+\oell-1]\subset [1,N]\setminus \{ m' : |m'-m_0'|\ge \exp(\ell^2) \}.
  \end{equation*}
  So \eqref{eq:no-long-range-resonance} follows from \cref{eq:lambda-m-lde} and
  the covering form of (LDT) (see  Lemma~\ref{lem:Greencoverap1}). Indeed, each point of $ [m,m+\overline \ell-1] $
  is covered by an interval of the form
  $ \Lambda(m')\cap [m,m+\overline \ell-1] $ for some $ m'\in [m-1,m+\overline \ell-1-\ell] $,
  on which the large deviations estimate holds.  Note that we wanted
  to make sure that $ m'+[1,\ell]\subset [m,m+\overline \ell-1] $, because if
  a large deviations estimate holds on $ m'+[1,\ell] $, it
  does not necessarily hold on $ (m'+[1,\ell])\cap[m,m+\overline \ell-1] $. On
  the other hand, we already noted that when a large deviations estimate holds
  on $ m'+[-\overline \ell_k,\overline \ell_k] $ it also holds on
  $ (m'+[-\overline \ell_k,\overline \ell_k])\cap[m,m+\overline \ell-1] $.  Finally note that the lower bound
  on $ \oell $ is such that
  \begin{equation*}
    |\Lambda(m')|\le \oell_{(q-1)^2}\le \oell^{\sigma/2},
  \end{equation*}
  as needed for the covering form of (LDT).
\end{proof}

The following propositions  are more detailed versions of \cref{thm:A} and \cref{thm:B}.
\begin{prop}\label{prop:A}
  Let $ \epsilon\in(0,1/5) $, $ \gamma>0 $, $ \sigma $ as in (LDT), and $ \Omega_N $, $ \cB_{N,\omega} $
  as in \cref{lem:elimination}.
  For any $ N\ge N_0(V,a,b,\gamma,\epsilon) $, $ \omega_0\in \T^d(a,b)\setminus \Omega_N $,
  $ x_0\in \T^d\setminus \cB_{N,\omega} $,
  and any eigenvalue $ E_0=E_j^{(N)}(x_0,\omega_0) $, such that $ L(\omega_0,E_0)>\gamma $,
  there exists an interval $ I=I (x_0, \omega_0,E_0)\subset [1,N] $,
  \begin{equation*}
    |I|<\exp((\log N)^{5\epsilon}),
  \end{equation*}
  such that for any $ (x,\omega)\in \T^d\times \T^d(a,b) $, with
  $$|x - x_0|,\ |\omega - \omega_0|< \exp(-\exp((\log N)^{2\epsilon\sigma})),$$ we have
  \begin{equation*}
    \left| \psi_j^{(N)}(x,\omega;m) \right|< \exp\left(-\frac{\gamma}{4}\dist(m,I)\right),
  \end{equation*}
  provided  $ \dist(m,I)> \exp((\log N)^{2\epsilon\sigma})$.
\end{prop}
\begin{proof}
  Let  $ \psi $  a choice of normalized
  eigenvector for $ E_0 $ and let $ m_0 $ be such that
  \begin{equation*}
    |\psi(m_0)|=\max_{m\in [1,N]}|\psi(m)|.
  \end{equation*}
  Then $ m_0 $ satisfies the assumptions of \cref{lem:elimination}, because otherwise we can use Poisson's formula
  \cref{eq:poissonC}
  and \cref{lem:Green} to contradict the maximality of $ |\psi(m_0)| $. So, if we let
  \begin{equation*}
    \oell= \lceil \exp((\log N)^{3\epsilon\sigma/2})\rceil,
  \end{equation*}
  by \cref{lem:elimination},
  \cref{eq:no-long-range-resonance} holds for $ m\in [1,N]\setminus I $,
  $ I=[1,N]\cap[m_0-\oell-2\exp(\ell^2),m_0+\oell+2\exp(\ell^2)] $,
  $ \ell=\lceil (\log N)^{2\epsilon}\rceil $.
  The conclusion follows immediately from \cref{prop:NDRloc1} (with $ \oell $ instead of $ \ell $).
\end{proof}

\begin{prop}\label{prop:B}
  Let $ \epsilon\in(0,1/5) $, $ \gamma>0 $, and $ \Omega_N $, $ \cB_{N,\omega} $ as in \cref{prop:A}.
  For any $ N\ge N_0(V,a,b,\gamma,\epsilon) $, $ \omega_0\in \T^d(a,b)\setminus \Omega_N $,
  $ x_0\in \T^d\setminus \cB_{N,\omega} $,
  and any eigenvalue $ E_0=E_j^{(N)}(x_0,\omega_0) $, such that $ L(\omega_0,E_0)>\gamma $, if we take
  $ I=I(x_0,\omega_0,E_0) $ as in \cref{prop:A}, then
  \begin{equation*}
    |E_k^{(N)}(x,\omega)-E_{j}^{(N)}(x,\omega)|>\exp(-C(V)|I|),
  \end{equation*}
  for any $ k\neq j $ and any $ (x,\omega)\in \T^d\times \T^d(a,b) $, with
  $$|x - x_0|,\ |\omega - \omega_0|< \exp(-\exp((\log N)^{2\epsilon\sigma})).$$
\end{prop}
\begin{proof}
  From the proof of \cref{prop:A} we know that
  \eqref{eq:no-long-range-resonance} holds
  for $ m\in [1,N]\setminus I $,
  $ I=[1,N]\cap[m_0-\oell-2\exp(\ell^2),m_0+\oell+2\exp(\ell^2)] $,
  $ \ell=\lceil (\log N)^{2\epsilon}\rceil $.
  The conclusion
  follows from Proposition~\ref{prop:Ej_NDRsep} (with $ \overline \ell $ instead of $ \ell $). Since $ E_0 $
  is restricted to the spectrum, we can choose the constant $ C $ from Proposition~\ref{prop:Ej_NDRsep}
  independent of $ E_0 $.
\end{proof}

Next we establish two auxilliary results needed for the proof of \cref{thm:C}.
\begin{lemma}\label{lem:C}
  Let $ \epsilon\in (0,1/5) $, $ \gamma>0 $ and $ \hat \Omega_N $, $ \hat\cB_{N,\omega} $ as in \cref{thm:C}.
  For
  $ N\ge N_0(V,a,b,\gamma,\epsilon) $, $ \omega\in \T^d(a,b)\setminus \hat\Omega_{N} $,
  $ x\in\T^d\setminus \hat \cB_{N,\omega} $,
  if
  \begin{equation*}
    L(\omega,E_j^{[-N,N]}(x,\omega))>\frac{3\gamma}{2}
    \text{ and }I=I(x,\omega,E_j^{[-N,N]}(x,\omega))\subset [-N/2,N/2]
  \end{equation*}
  then for any $ N\le N'\le \exp((\log N)^{\frac{1}{10\epsilon}}) $ there exists
  $ j_{N'} $,  such that
  \begin{gather}
    \label{eq:EN'}
	\left| E_{j_{N'}}^{[-N',N']}(x,\omega)-E_{j}^{[-N,N]}(x,\omega) \right|
    <\exp \left( -\frac{\gamma}{9}N \right), \\
    \label{eq:psiN'}
    \norm{\psi_{j_{N'}}^{[-N',N']}(x,\omega)-\psi_j^{[-N,N]}(x,\omega)}
    < \exp \left( -\frac{\gamma}{9}N \right).
  \end{gather}
  Furthermore, $ E_{j_{N'}}^{[-N',N']}(x,\omega) $ is localized,
  $ I(x,\omega,E_{j_{N'}}^{[-N',N']}(x,\omega))\subset [-3N/4,3N/4] $, and
  \begin{equation*}
    \left| \psi_{j_{N'}}^{[-N',N']}(x,\omega;n) \right|<\exp \left( -\frac{\gamma}{17}\dist(n,I) \right),
    \ N\le |n|\le N'.
  \end{equation*}
\end{lemma}
\begin{proof}
  By \cref{thm:A},
  \begin{equation*}
    \left| \psi_j^{[-N,N]}(x,\omega;\pm N) \right|<\exp
    \left( -\frac{\gamma}{4} \left( \frac{N}{2}-1 \right)\right).
  \end{equation*}
  Then
  \begin{equation*}
    \norm{(H_{[-N',N']}(x,\omega)-E_j^{[-N,N]}(x,\omega))\psi_j^{[-N,N]}(x,\omega)}<
    2\exp\left( -\frac{\gamma}{4} \left( \frac{N}{2}-1 \right)\right)
  \end{equation*}
  and applying \cref{lem:eigenvector-stability} (with the aid of \cref{thm:B}) we get that there exists $ j_{N'} $
  such that \cref{eq:EN'} and \cref{eq:psiN'} hold.
  Invoking \cref{prop:log-Hoelder} we have  $$ L(\omega,E_{j_{N'}}^{[-N',N']}(x,\omega))>\gamma $$ and therefore
  $ E_{j_{N'}}^{[-N',N']}(x,\omega) $ is localized. Using \cref{eq:psiN'} and $ I\subset [-N/2,N/2] $ we get
  \begin{equation*}
    \sum_{|n|\le 5N/9} \left| \psi_{j_{N'}}^{[-N',N']}(x,\omega;n) \right|^2>1/2.
  \end{equation*}
  It follows from \cref{thm:A} (arguing by contradiction) that
  \begin{equation*}
    I(x,\omega,E_{j_{N'}}^{[-N',N']}(x,\omega))\subset [-3N/4,3N/4],
  \end{equation*}
  and
  \begin{align*}
    \left| \psi_{j_{N'}}^{[-N',N']}(x,\omega;n) \right|&<\exp \left( -\frac{\gamma}{4}\dist(n,[-3N/4,3N/4]) \right)\\
    &<\exp \left( -\frac{\gamma}{17}\dist(n,I) \right),
  \end{align*}
  for any   $N\le |n|\le N'.$
\end{proof}

\begin{lemma}\label{lem:Cb}
  Let $ \epsilon\in (0,1/5) $, $ \gamma>0 $ and $ \hat \Omega_N $, $ \hat\cB_{N,\omega} $ as in \cref{thm:C}.
  The following statement holds for any
  $ N_0\ge C(V,a,b,\gamma,\epsilon) $, $ \omega\in \T^d(a,b)\setminus \hat\Omega_{N_0} $,
  $ x\in\T^d\setminus \hat \cB_{N_0,\omega} $.
  If $ E,\psi $ is a generalized eigenpair for $ H(x,\omega) $, that is
  \begin{equation*}
    H(x,\omega)\psi=E\psi,\quad \psi\not \equiv 0,\quad |\psi(m)|<C(\psi)(|m|+1),\ m\in \Z,
  \end{equation*}
  and $ L(\omega,E)>\gamma  $, then
  \begin{equation*}
    |\psi(m)|<\exp \left( -\frac{\gamma}{4}|m| \right),\ |m|\ge C(\psi,V,a,b,\gamma).
  \end{equation*}
\end{lemma}
\begin{proof}
  There exists $ C_0(\psi,V,a,b,\gamma) $ such that for any interval $ \Lambda\subset\Z $ satisfying
  \begin{equation*}
    |\Lambda|\ge C_0\text{ and }\dist(0,\Z\setminus \Lambda)>|\Lambda|/100
  \end{equation*}
  we have
  \begin{equation*}
    \log|f_\Lambda(x,\omega,E)|\le |\Lambda|L_{|\Lambda|}(\omega,E)-|\Lambda|^{1-\tau/4}.
  \end{equation*}
  Otherwise, we could use Poisson's formula and \cref{lem:Green} to show that $ \psi\equiv 0 $. So, for $ N $
  large enough we can apply \cref{lem:elimination} with $ m_0=0 $ to get that \cref{eq:no-long-range-resonance}
  holds for any
  \begin{gather*}
	m\in [-N,N-\oell+1]\setminus I,\ I=[-\oell-2\exp(\ell^2),\oell+2\exp(\ell^2)],\\
    \oell=\lfloor \exp((\log N)^{2\epsilon\sigma}) \rfloor,\   \ell=\lceil (\log (2N+1))^{2\epsilon} \rceil.
  \end{gather*}
  Let
  \begin{equation*}
    J=\begin{cases}
      [\lfloor m/4 \rfloor, 2m] &, m>0\\
      [2m, \lceil m/4 \rceil] &, m<0.
    \end{cases}
  \end{equation*}
  Take $ m $ such that $ \exp((\log N)^{5\epsilon})\le |m|\le N/2 $.
  Then $ J\subset [-N,N]\setminus I $, and we can use
  \cref{eq:no-long-range-resonance} and the covering form of (LDT) to get
  \begin{equation*}
    \log|f_J(x,\omega,E)|>|J|L_{|J|}(\omega,E)-|J|^{1-\tau/2}.
  \end{equation*}
  Poisson's formula and \cref{lem:Green} imply
  \begin{equation*}
    |\psi(m)|<2C(\psi)(2|m|+1)\exp \left( -\frac{\gamma}{2}\frac{3|m|}{4}+C|m|^{1-\tau} \right)
    <\exp \left( -\frac{\gamma}{4}|m| \right),
  \end{equation*}
  provided $ N $ is large enough. This concludes the proof.
\end{proof}

\begin{proof}[Proof of \cref{thm:C}]
  (a) Let $ N_k=N^{2^k} $. Iterating \cref{lem:C} (and using \cref{prop:log-Hoelder}), we get that there
  exist $ j_k $, $ k\ge 0 $, $ j_0=j $,
  such that
  \begin{gather*}
	\left| E_{j_{k+1}}^{[-N_{k+1},N_{k+1}]}(x,\omega)-E_{j_k}^{[-N_k,N_k]}(x,\omega) \right|
    <\exp \left( -\frac{\gamma}{9}N_k \right), \\
    \norm{\psi_{j_{k+1}}^{[-N_{k+1},N_{k+1}]}(x,\omega)-\psi_{j_{k}}^{[-N_k,N_k]}(x,\omega)}
    < \exp \left( -\frac{\gamma}{9}N_k \right),\\
    L(\omega,E_{j_{k+1}}^{[-N_{k+1},N_{k+1}]}(x,\omega))>\frac{3\gamma}{2},\\
    I_k=I(x,\omega,E_{j_{k+1}}^{[-N_{k+1},N_{k+1}]}(x,\omega))\subset [-3N_{k}/4,3N_{k}/4],\\
    \left| \psi_{j_{k+1}}^{[-N_{k+1},N_{k+1}]}(x,\omega;n) \right|
    <\exp \left( -\frac{\gamma}{17}\dist(n,I_k) \right)
    <\exp \left( -\frac{\gamma}{18}\dist(n,I) \right),
  \end{gather*}
  for all $ N_{k}\le |n|\le N_{k+1}. $
  It follows that for any $ k'\ge k\ge  0 $,
  \begin{gather*}
	\left| E_{j_{k'}}^{[-N_{k'},N_{k'}]}(x,\omega)-E_{j_k}^{[-N_k,N_k]}(x,\omega) \right|
    <\sum_{i=k}^{k'-1}\exp \left( -\frac{\gamma}{9}N_i \right)
    <\exp\left( -\frac{\gamma}{10}N_k \right), \\
    \norm{\psi_{j_{k'}}^{[-N_{k'},N_{k'}]}(x,\omega)-\psi_{j_k}^{[-N_k,N_k]}(x,\omega)}
    <\sum_{i=k}^{k'-1}\exp \left( -\frac{\gamma}{9}N_i \right)
    <\exp\left( -\frac{\gamma}{10}N_k \right).
  \end{gather*}
  This proves \cref{eq:jk}, \cref{eq:jk'k}, and the existence of the limits $ E(x,\omega)$, $ \psi(x,\omega) $.
  The fact that $ \norm{\psi}=1 $ follows immediately.

  We already established that \cref{eq:psik} holds for $ N_{k-1}\le |n|\le N_k $. For $ 1\le k'\le k-2 $ we have
 \begin{equation} \begin{aligned}\label{eq:psikk'}
    \left| \psi_{j_k}^{[-N_k,N_k]}(x,\omega;n) \right|
    &<\left| \psi_{j_{k'}}^{[-N_{k'},N_{k'}]}(x,\omega;n) \right|
    +\sum_{i=k'}^{k-1}\exp \left( -\frac{\gamma}{9}N_{i} \right)\\
   & <\exp \left( -\frac{\gamma}{20}\dist(n,I) \right),
  \end{aligned}\end{equation}
  for $ N_{k'-1}\le |n|\le N_{k'} $. This shows \cref{eq:psik} holds for $ N\le |n|\le N_k $.
  By \cref{thm:A},
  \begin{equation*}
    \left| \psi_{j}^{[-N,N]}(x,\omega;n) \right|<\exp \left( -\frac{\gamma}{17}\dist(n,I) \right),\
    3N/4\le |n|\le N.
  \end{equation*}
  Using the reasoning of \cref{eq:psikk'} again we conclude that \cref{eq:psik} holds as stated.

  Taking $ k'\to \infty $ in \cref{eq:jk'k} we have
  \begin{equation*}
    \norm{\psi(x,\omega)-\psi_{j_k}^{[-N_k,N_k]}(x,\omega)}<\exp \left( -\frac{\gamma}{10} N_k \right)
  \end{equation*}
  and \cref{eq:psi} follows in the same way as \cref{eq:psik}.

  (b) It is well-known that to get purely pure point spectrum it is enough to show that generalized eigenvectors
  decay exponentially. We have this by \cref{lem:Cb}, so we just need to argue that all eigenpairs can be obtained
  as limits from (a). Let $ E,\psi $ be an eigenpair for $ H(x,\omega) $, $ E\in [E',E''] $,
  $ L(\omega,E)>3\gamma $. By \cref{lem:Cb},
  \begin{equation}\label{eq:psi-decay}
    |\psi(m)|<\exp \left( -\frac{\gamma}{4}|m| \right),\ |m|\ge C(\psi,V,a,b,\gamma).
  \end{equation}
  Since
  \begin{equation*}
    \norm{(H_{[-N,N]}(x,\omega)-E)\psi}^2=|\psi(-N-1)|^2+|\psi(N+1)|^2,
  \end{equation*}
  it follows from \cref{lem:eigenvector-stability} that for $ N $ large enough there exists $ j $ such that
  \begin{equation}\label{eq:psi-psiN}
    \left| E_j^{[-N,N]}(x,\omega)-E \right|<\exp \left( -\frac{\gamma}{5}N \right),\quad
    \norm{\psi_j^{[-N,N]}(x,\omega)-\psi}<\exp\left( -\frac{\gamma}{5}N \right).
  \end{equation}
  By \cref{prop:log-Hoelder}, $ L(\omega,E_j^{[-N,N]}(x,\omega))>2\gamma $ and hence $  E_j^{[-N,N]}(x,\omega) $
  is localized. For $ N $ large enough, \cref{eq:psi-decay} and \cref{eq:psi-psiN} imply that
  $ I=I(x,\omega,E_j^{[-N,N]}(x,\omega))\subset [-N/2,N/2] $. Let $ j_k $ be the sequence from (a). By the same
  reasoning as above, there exist $ j_k' $ such that
  \begin{equation}\label{eq:psi-psiNk}
    \left| E_{j_k'}^{[-N_k,N_k]}(x,\omega)-E \right|<\exp \left( -\frac{\gamma}{5}N_k \right),\quad
    \norm{\psi_{j_k'}^{[-N_k,N_k]}(x,\omega)-\psi}<\exp\left( -\frac{\gamma}{5}N_k \right).
  \end{equation}
  We just need to argue that $ j_k'=j_k $ and that the choice of normalized eigenvector is the same as in (a).
  This follows by induction. We just check the initial step, the general
  step being analogous. We have that
  \begin{equation*}
    \left| E_{j_1'}^{[-N_1,N_1]}(x,\omega)-E_{j_1}^{[-N_1,N_1]}(x,\omega) \right|
    \lesssim \exp \left( -\frac{\gamma}{10} N\right),
  \end{equation*}
  so \cref{thm:B} implies that $ j_1=j_1' $. If the choice of normalized eigenvector in \cref{eq:psi-psiNk} is not
  the same as in (a) we would have
  \begin{align*}
    &\norm{\psi_{j_1}^{[-N_1,N_1]}(x,\omega)-\psi_{j}^{[-N,N]}(x,\omega)}<\exp \left( -\frac{\gamma}{10} N \right)\\
    &\norm{-\psi_{j_1}^{[-N_1,N_1]}(x,\omega)-\psi}<\exp\left( -\frac{\gamma}{5} N_1 \right)
  \end{align*}
  which would imply
  \begin{equation*}
    \norm{\psi_{j_1}^{[-N_1,N_1]}(x,\omega)}\lesssim \left( -\frac{\gamma}{10} N \right),
  \end{equation*}
  contradicting the fact that $ \psi_{j_1}^{[-N_1,N_1]}(x,\omega) $ is normalized.
\end{proof}

\begin{remark}
  It is possible to prove part (b) of \cref{thm:C} in a direct manner by showing that the eigenvectors from part
  (a) form an orthonormal basis for the spectral subspace correspoding to  $ [E',E''] $. This argument is
  straightforward in the case when the Lyapunov exponent is positive on $ \cS_\omega $, see
  \cite[Prop.\ 13.1]{GolSch11}, but in general it is more complicated than the above indirect argument.
\end{remark}

\section{Stabilization of the Spectrum: Proofs of Theorems D, E}\label{sec:Thm-D-E}

We prove the relations \cref{eq:S-SN} and \cref{eq:SN-S} from \cref{thm:D} separately.
The relation \cref{eq:S-SN} is an immediate consequence of the following lemma.
\begin{lemma}\label{lemma:full-to-restricted}
  Assume $ \omega\in \T^d(a,b) $ and $ L(\omega,E)>\gamma>0 $ for all $ E\in [E',E''] $. Let $ k_0\ge 0 $,
  $ s>1 $ integers, and
  $ \sigma $ as in (LDT).
  Then for all $ N\ge N_0(V,a,b,\gamma) $ we have
  \begin{equation*}
    \cS_\omega\cap [E',E''] \subset \fS_{N,\omega}(s,k_0,\rho),
  \end{equation*}
  for any $ \rho\in \R^{k_0+1} $ such that $ \rho_{k}\ge \exp(-(N^{(k)})^{\sigma/2}) $, $ k=0,\ldots,k_0 $.
\end{lemma}
\begin{proof}
  Take $ E\in[E',E'']\setminus \fS_{N,\omega}(s,k_0,\rho)$. We just need to show $ E\notin \cS_\omega $.
  For any $ x\in \T^d $ there exists
  \begin{equation*}
    I=I(x)\in \{ [-N^{(k)},N^{(k)}]: k=0,\ldots,k_0 \}
  \end{equation*}
  such that
  \begin{equation*}
    \dist(E,\spec H_{I}(x,\omega))\ge\exp(-|I|^{\sigma/2}),
  \end{equation*}
  and by the spectral form of (LDT) (see \cref{cor:4.6zeros})
  \begin{equation*}
    \log|f_{I}(x,\omega,E)|>|I|L_{|I|}(\omega,E)-|I|^{1-\tau/2}.
  \end{equation*}
  Given $ \oN$ set
   \[
   \Lambda_{\oN}=\bigcup_{m\in[-\oN,\oN]} I(m\omega).
   \]
   From the covering form of (LDT)
  (see \cref{lem:Greencoverap1}) we get
  \begin{equation*}
    \dist(E,\spec H_{\Lambda_{\overline{N}}}(0,\omega))\ge \exp(-N).
  \end{equation*}
  Since $ \oN $ can be chosen arbitrarily large,  Lemma~\ref{lem:elemspec1} implies
  that $ \dist(E,\cS_\omega)\ge\exp(-N) $. In particular, $ E\notin \cS_\omega $ and the conclusion follows.
\end{proof}

We will prove relation \cref{eq:SN-S} from \cref{thm:D} in \cref{prop:SN-S},
but first we establish some auxilliary results. We start with an application of Bourgain's elimination of
multiple resonances. The next lemma is the reason for the choice of parameters $ s,k_0,\rho $ in
\cref{thm:D}.

\begin{lemma}\label{lem:SN-NDR}
  Let $ x\in \T^d $, $ E\in \R $, $ \gamma>0 $,  $ A\ge 1 $, $ s> 1 $, $ N\ge 1 $, $ N^{(k)}=N^{s^k} $, $ \sigma $
  as in (LDT).
  For any $ s\ge s_0(a,b,A) $, $ N\ge N_0(V,a,b,\gamma,A,s) $  there exists a
  set $ \Omega_N=\Omega_N(s) $, $ \mes(\Omega_N)<N^{-A} $, such that the following holds.
  If $ \omega\in \T^d\setminus \Omega_N $, $ L(\omega,E)>\gamma $, and
  \begin{gather*}
	\dist(E,\spec H_{[-N,N]}(x,\omega))<\exp(-N^{\sigma/2}),\\
    \dist(E,\spec H_{[-N^{(k)},N^{(k)}]}(x,\omega))<\exp \left( -\frac{\gamma}{10}N \right)
    ,\ k=1,\ldots,k_0,\ k_0=2^{2d+1}-1,
  \end{gather*}
  then  there exist $ k=k(x)\in \{ 1,\ldots,k_0 \}  $, $ j=j(x) $,
  such that
  \begin{gather}
    \label{eq:Nk-E}|E-E_j^{[-N^{(k)},N^{(k)}]}(x,\omega)|<\exp\left( -\frac{\gamma}{10}N \right),\\
    \label{eq:Nk-psi}|\psi_j^{[-N^{(k)},N^{(k)}]}(x',\omega,\pm N^{(k)})|
    <\exp \left( -\frac{\gamma}{8}N^{(k)} \right),\
     |x'-x|<\exp(-(\log N)^{3/\sigma}).
  \end{gather}
\end{lemma}
\begin{proof}
  Let $ s $ be a sufficiently large integer, so that the sets
  \begin{equation*}
    \cN_{k}=\{n\in \Z: (N^{(k)})^{1/2}\le |n|\le 2N^{(k)}\},\ k=1,\ldots,k_0
  \end{equation*}
  satisfy the assumptions of \cref{prop:NDR}, with $ \ell=\lfloor (\log N)^{2/\sigma} \rfloor $, for $ N $ large
  enough. We let $ \Omega_N $ be the set of resonant frequencies from \cref{prop:NDR}. We have
  \begin{equation}\label{eq:N(1)-A}
    \mes(\Omega_N)\le \ell^C (N^{(1)})^{-1/2}<N^{-A},
  \end{equation}
  provided $ s $ is large enough. Since
  \begin{equation*}
    \dist(E,\spec H_{[-N,N]}(x,\omega))<\exp(-N^{\sigma/2})<\exp(-\ell),
  \end{equation*}
  it follows from the covering form of (LDT) (see \cref{lem:Greencoverap1}) that there exist $ m_0\in[-N,N] $ such
  that
  \begin{equation*}
    \log|f_\ell(x+m_0\omega,\omega,E)|\le \ell L_\ell(\omega,E)-\ell^{1-\tau/2}.
  \end{equation*}
  Applying \cref{prop:NDR}, we get that there exists $ k $ such that
  \begin{equation*}
    \log|f_\ell(x+(m-1)\omega,\omega,E)|>\ell L_\ell(\omega,E)-\ell^{1-\tau/2}\text{ for any }
    2(N^{(k)})^{1/2}\le |m|\le N^{(k)}
  \end{equation*}
  (note that we chose $ \cN_k $ so that we get the above estimate after we take into account the shift by
  $ m_0\omega $).
  By the hypothesis, there exists $ E_j^{[-N^{(k)},N^{(k)}]}(x,\omega) $ so that
  \begin{equation*}
    |E-E_j^{[-N^{(k)},N^{(k)}]}(x,\omega)|<\exp\left( -\frac{\gamma}{10}N \right).
  \end{equation*}
  The localization estimate \cref{eq:Nk-psi} follows immediately from  \cref{prop:NDRloc1} by noting that
  \begin{equation*}
    |E-E_j^{[-N^{(k)},N^{(k)}]}(x',\omega)|<\exp\left( -\frac{\gamma}{10}N \right)+C|x'-x|<\exp(-\ell),    
  \end{equation*}
  provided $ |x'-x|<\exp(-(\log N)^{3/\sigma})<\exp(-\ell^{3/2}) $. 
\end{proof}

The next result is a technical tool for \cref{lem:SN1-SN2}. It's peculiar statement stems from the
following issue. The function $ \mes(\fS_N(s,k_0,\rho)) $ can have jump-discontinuities in $ \rho $, and we have no
control over the size of the jumps. This forces us to work with $ \mes((\fS_N(s,k_0,\rho))^{(\rho')}) $
(recall \cref{eq:rho-neighborhood}) which is Lipschitz continuous in $ \rho' $.
\begin{lemma}\label{lem:SN-fattening}
  Let $ N\ge 1 $, $ \omega\in \T^d $, $ s>1 $, $ k_0\ge 1 $, $ \rho\in \R^{k_0+1} $,
  $ \rho'>0 $,
  $ \epsilon\in \R^{k_0+1} $,
  $ \epsilon_k=\exp(-N^{3/2}) $. For any $ N\ge N_0(V,a,b) $ we have
  \begin{equation*}
    \mes \left( (\fS_{N,\omega}(s,k_0,N^{-1/2}(\rho-\epsilon)-\epsilon))^{(\rho')}\setminus
      \fS_{N,\omega}(s,k_0,\rho) \right)< (N^{(k_0)})^{C(a,b)} \rho'.
  \end{equation*}
\end{lemma}
\begin{proof}
  Take $ \tilde V $ as in \cref{eq:V-tilde} and consider the set of $ (x,E) $ such that
  \begin{equation*}
    \normhs{(\tilde H_{[-N^{(k)},N^{(k)}]}(x,\omega)-E)^{-1}}> \rho_k^{-1},\ k=0,\ldots,k_0.
  \end{equation*}
  This set is semialgebraic of degree less than $ (N^{(k_0)})^C $. We let $ \tilde \fS_{N,\omega}(s,k_0,\rho) $ be
  its projection onto the $ E $-axis. By the Tarski-Seidenberg principle, $ \tilde \fS_{N,\omega}(s,k_0,\rho) $ is
  also semialgebraic of degree less than $ (N^{(k_0)})^{C'} $ and hence it has less than $ (N^{(k_0)})^{C''} $
  connected components (see \cite[Ch.\ 9]{Bou05}).
  Note that
  \begin{equation*}
    |E_j^{[-N^{(k)},N^{(k)}]}(x,\omega)-\tilde E_j^{[-N^{(k)},N^{(k)}]}(x,\omega)|<\exp(-N^{3/2}),
    \ k=0,\ldots,k_0
  \end{equation*}
  (recall \cref{eq:E-tilde}) and therefore
  \begin{equation*}
    \fS_{N,\omega}(s,k_0,\rho)\subset \tilde \fS_{N,\omega}(s,k_0,\rho+\epsilon),\quad
    \tilde \fS_{N,\omega}(s,k_0,\rho)\subset \fS_{N,\omega}(s,k_0,N^{1/2}\rho+\epsilon)
  \end{equation*}
  (recall that $ \norm{\cdot}\le \normhs{\cdot}\le \sqrt{N} \norm{\cdot} $ on $ \C^{N\times N} $). It follows that
  \begin{multline*}
    (\fS_{N,\omega}(s,k_0,N^{-1/2}(\rho-\epsilon)-\epsilon))^{(\rho')}\setminus
    \fS_{N,\omega}(s,k_0,\rho)\\
    \subset (\tilde \fS_{N,\omega}(s,k_0,N^{-1/2}(\rho-\epsilon)))^{(\rho')}
    \setminus \tilde \fS_{N,\omega}(s,k_0,N^{-1/2}(\rho-\epsilon)).
  \end{multline*}
  Since $ \tilde \fS_{N,\omega} $ has less than $ (N^{(k_0)})^{C''} $ components,
  \begin{equation*}
    \mes \left( (\tilde \fS_{N,\omega}(s,k_0,N^{-1/2}(\rho-\epsilon)))^{(\rho')}
      \setminus \tilde \fS_{N,\omega}(s,k_0,N^{-1/2}(\rho-\epsilon)) \right)\lesssim (N^{(k_0)})^{C''}\rho'
  \end{equation*}
  and we are done.
\end{proof}

The following is a finite volume version of the second part of \cref{thm:D} and it will imply the full scale
statement.
\begin{lemma}\label{lem:SN1-SN2}
  Let $ \gamma>0 $,  $ A\ge 1 $, $ \sigma $
  as in (LDT).
  For any $ s\ge s_0(a,b,A) $, $ N\ge N_0(V,a,b,\gamma,A,s) $, there exists a
  set $ \Omega_N=\Omega_N(s) $, $ \mes(\Omega_N)<N^{-A} $, such that the following holds with
  \begin{gather*}
    \fS_{N,\omega}=\fS_{N,\omega}(s,k_0,\rho_N),\ k_0=2^{2d+1}-1,\\
    \rho_{N,0}=\exp(-N^{\sigma/2}),\ \rho_{N,k}=\exp\left( -\frac{\gamma}{10}N \right),k=1,\ldots,k_0.
  \end{gather*}
  If $ \omega\in \T^d\setminus \Omega_N $ and $ L(\omega,E)>\gamma $ for all $ E\in [E',E''] $ and
  $ \oN=\lfloor \exp(N^{1/2})\rfloor $, then
  \begin{equation*}
    \mes \left( (\fS_{N,\omega}\cap [E',E''])\setminus  \fS_{\oN,\omega}\right)
    <\exp\left( -\frac{\gamma}{15}N \right).
  \end{equation*}
\end{lemma}
\begin{proof}
  Let $ \Omega_N' $ be the set from \cref{lem:SN-NDR} with $ A+1 $ instead of $ A $.
  Let $ \Omega_{\oN} $, $ \cB_{\oN,\omega} $, be the sets
  from \cref{thm:A} with $ \epsilon=1/10 $, $ [-\oN,\oN] $ instead of $ [1,\oN] $,  and $ \gamma/2 $ instead of
  $ \gamma $. The conclusion will hold
  with $ \Omega_N=\Omega_N'\cup \Omega_{\oN} $. Clearly the measure estimate for $ \Omega_N $ is satisfied.

  Take $ E\in \fS_{N,\omega}\cap [E',E''] $. Then by \cref{lem:SN-NDR},
  there exist $ x\in \T^d $, $ k=k(x) $, $ j=j(x) $, such that \eqref{eq:Nk-E} and \eqref{eq:Nk-psi} hold. Since
  \begin{equation*}
    \mes(\cB_{\oN,\omega})<\exp(-\exp((\log \oN)^{\sigma/10})),
  \end{equation*}
  there exists $ x'\notin \cB_{\oN,\omega} $,
  \begin{equation*}
    |x'-x|<\exp(-\exp((\log \oN)^{\sigma/20}))<\exp(-(\log N)^{3/\sigma}).
  \end{equation*}
  Using \cref{cor:localglob1}, we have that there exists $ E_i^{[-\oN,\oN]}(x',\omega) $ such that
  \begin{gather*}
	|E_i^{[-\oN,\oN]}(x',\omega)-E_j^{[-N^{(k)},N^{(k)}]}(x',\omega)|
    \lesssim \exp \left( -\frac{\gamma}{8}N^{(k)} \right),\\
    \max_{n\in [-N^{(k)},N^{(k)}]} |\psi_i^{[-\oN,\oN]}(x',\omega;n)|>(2N^{(k)}\oN)^{-1/2}.
  \end{gather*}
  Note that
  \begin{multline*}
    |E-E_i^{[-\oN,\oN]}(x',\omega)|\lesssim \exp\left( -\frac{\gamma}{10}N \right)
    +C\exp(-\exp((\log \oN)^{\sigma/20}))
    +\exp \left( -\frac{\gamma}{8}N^{(k)} \right)\\
    < \exp\left( -\frac{\gamma}{11}N \right),
  \end{multline*}
  so by \cref{prop:log-Hoelder}, $ L(E_i^{[-\oN,\oN]}(x',\omega),\omega)>\gamma/2 $, provided $ N $ is large
  enough.
  Since $ x'\notin \cB_{\oN,\omega} $, by \cref{thm:A}, there exists $ I=I(E_i^{[-\oN,\oN]}(x',\omega)) $ such that
  $ |I|<\exp((\log \oN)^{1/2}) $ and
  \begin{equation*}
    |\psi_i^{[-\oN,\oN]}(x',\omega;n)|<\exp \left( -\frac{\gamma}{4} \dist(n,I) \right),
  \end{equation*}
  provided $ \dist(n,I)>\exp((\log \oN)^{\sigma/2}) $. It follows that
  \begin{equation*}
    \dist([-N^{(k)},N^{(k)}],I)\le \exp((\log \oN)^{\sigma/2})
  \end{equation*}
  and in particular
  \begin{equation*}
    |\psi_i^{[-\oN,\oN]}(x',\omega;\pm \oN)|<\exp \left( -\frac{\gamma}{8} \oN \right).
  \end{equation*}
  Using \cref{lem:eigenvalue-stabilization} we get that
  \begin{equation*}
    E_i^{[-\oN,\oN]}(x',\omega)\in \fS_{\oN,\omega}(s,k_0,\rho'),
    \ \rho'_k=4\exp \left( -\frac{\gamma}{8} \oN \right).
  \end{equation*}
  and therefore
  \begin{equation*}
    E\in \left( \fS_{\oN,\omega}(s,k_0,\rho') \right)^{(\rho'')},\ \rho''=\exp\left( -\frac{\gamma}{11}N \right).
  \end{equation*}
  Thus we showed
  \begin{equation*}
    \fS_{N,\omega}\cap [E',E'']\subset \left( \fS_{\oN,\omega}(s,k_0,\rho') \right)^{(\rho'')}.
  \end{equation*}
  Then
  \begin{multline*}
    (\fS_{N,\omega}\cap [E',E''])\setminus \fS_{\oN,\omega}\subset
    \left( \fS_{\oN,\omega}(s,k_0,\rho') \right)^{(\rho'')}\setminus \fS_{\oN,\omega}\\
    \subset
    \left( \fS_{\oN,\omega}(s,k_0,\oN^{-1/2}(\rho_{\oN}-\epsilon)-\epsilon) \right)^{(\rho'')}
    \setminus \fS_{\oN,\omega},\ \epsilon_k=\exp(-\oN^{3/2})
  \end{multline*}
  and the conclusion follows from \cref{lem:SN-fattening}.
\end{proof}

We can now prove the second part of \cref{thm:D}.
\begin{prop}\label{prop:SN-S}
  Let $ \gamma>0 $,  $ A\ge 1 $.
  For any $ s\ge s_0(a,b,A) $, $ N\ge N_0(V,a,b,\gamma,A,s) $,  there exists a
  set $ \Omega_N=\Omega_N(s) $, $ \mes(\Omega_N)<N^{-A} $, such that the following holds with
  $ \fS_{N,\omega} $ as in \cref{lem:SN1-SN2}.
  If $ \omega\in \T^d\setminus \Omega_N $ and $ L(\omega,E)>\gamma $ for all $ E\in [E',E''] $  then
  \begin{equation*}
    \mes \left( (\fS_{N,\omega}\cap [E',E''])\setminus  \cS_\omega \right)<\exp\left( -\frac{\gamma}{20}N \right).
  \end{equation*}
\end{prop}
\begin{proof}
  Let $ N_j=\lfloor \exp(N_{j-1}^{1/2}) \rfloor $, $ N_0=N $. Let $ \Omega_{N_j}' $ be the set from
  \cref{lem:SN1-SN2} with $ A+1 $ instead of $ A $. The conclusion will hold with
  $ \Omega_N=\bigcup_j \Omega_{N_j} $. Clearly the measure estimate for $ \Omega_N $ is satisfied.

  Using \cref{lem:eigenvalue-stabilization} it follows from \cref{lem:SN-NDR} that if
  $ E\in \fS_{N,\omega}\cap [E',E''] $, then $ \dist(E,\cS_\omega)<\exp(-\gamma N/11) $ for $ N $ large enough.
  Then we have
  \begin{equation*}
     \bigcap_{j\ge 1} \fS_{N_j,\omega}\cap [E',E'']\subset \cS_\omega
  \end{equation*}
  and
  \begin{multline*}
	(\fS_{N,\omega}\cap [E',E''])\setminus  \cS_\omega
    \subset (\fS_{N,\omega}\cap [E',E''])\setminus \left(\bigcap_{j\ge 1} \fS_{N_j,\omega}\cap [E',E'']\right)\\
    =(\fS_{N,\omega}\cap [E',E''])\setminus \bigcap_{j\ge 1} \fS_{N_j,\omega}
    \subset \bigcup_{j\ge 0} \left( (\fS_{N_j,\omega}\cap [E',E''])\setminus\fS_{N_{j+1},\omega} \right).
  \end{multline*}
  So the conclusion follows from \cref{lem:SN1-SN2}.
\end{proof}

\cref{thm:E} is an immediate consequence of the following proposition.
\begin{prop}\label{prop:pre-homogeneous}
  Let $ \gamma>0 $,  $ A\ge 1 $, $ \sigma $
  as in (LDT).
  For any $ N\ge N_0(V,a,b,\gamma,A) $ there exists a
  set $ \Omega_N $, $ \mes(\Omega_N)<N^{-A} $, such that the following holds.
  If $ \omega\in \T^d\setminus \Omega_N $, $ E_0\in \cS_\omega $ and $ L(\omega,E_0)>\gamma $, then
  \begin{equation*}
    \mes \left( (E_0-\delta_N,E_0+\delta_N) \cap \cS_\omega \right)> \frac{2}{3}\delta_N,
    \ \delta_N=\exp(-(\log N)^{4/\sigma^2}).
  \end{equation*}
\end{prop}
\begin{proof}
   Let $ \Omega_N' $ be the exceptional set from \cref{lem:SN1-SN2}, with $ A+1 $ instead of $ A $, $ s=s_0 $, and
  $ \gamma/2 $  instead of $ \gamma $. The conclusion will hold with
  $ \Omega_N= \bigcup_{k=0}^{k_0} \Omega_{N^{(k)}}' $, $ k_0=2^{2d+1}-1 $.

  Let $ \fS_{N,\omega} $ as in \cref{lem:SN1-SN2}.
  By \cref{lemma:full-to-restricted}, $ E_0\in \fS_{N,\omega} $, and by  \cref{lem:SN-NDR},
  there exist $ x $, $ k=k(x) $, $ j=j(x) $, such that
  \begin{gather*}
    |E_0-E_j^{[-N^{(k)},N^{(k)}]}(x,\omega)|<\exp\left( -\frac{\gamma}{10}N \right),\\
    |\psi_j^{[-N^{(k)},N^{(k)}]}(x',\omega,\pm N^{(k)})|
    <\exp \left( -\frac{\gamma}{8}N^{(k)} \right),
    \text{ for any }|x'-x|<\exp(-(\log N)^{3/\sigma}).
  \end{gather*}
  Let $ S= \{ x' : |x'-x|<\exp(-(\log N)^{3/\sigma}) \} $. Using \cref{lem:eigenvalue-stabilization} we have
  \begin{equation*}
    E_j^{[-N^{(k)},N^{(k)}]}(S,\omega)\subset \fS_{N^{(k)},\omega}.
  \end{equation*}
  Let $ [E',E'']=[E_0-N^{-1},E_0+N^{-1})] $. By \cref{prop:log-Hoelder}, $ L(\omega,E)>\gamma/2 $ for all
  $ E\in [E',E''] $, provided $ N $ is large enough. We have
  $ E_j^{[-N^{(k)},N^{(k)}]}(S,\omega)\subset [E',E''] $, so \cref{lem:Wegner-graphs} yields
  \begin{equation*}
    \mes(E_j^{[-N^{(k)},N^{(k)}]}(S,\omega))>\exp(-(\log N)^{4/\sigma^2}).
  \end{equation*}
  By \cref{prop:SN-S},
  \begin{equation*}
    \mes(E_j^{[-N^{(k)},N^{(k)}]}(S,\omega)\setminus \cS_\omega)<\exp\left( -\frac{\gamma}{20}N \right).
  \end{equation*}
  Let
  \begin{equation*}
    I=(E_0-\delta_N,E_0+\delta_N)\cap E_j^{[-N^{(k)},N^{(k)}]}(S,\omega).
  \end{equation*}
  Then
  \begin{multline*}
    \mes \left( (E_0-\delta_N,E_0+\delta_N) \cap \cS_\omega \right)\ge
    \mes(I\cap \cS_\omega)> \mes(I)-\exp\left( -\frac{\gamma}{20}N \right)\\
    >\exp(-(\log N)^{4/\sigma^2})-2\exp\left( -\frac{\gamma}{20}N \right)>\frac{2}{3}\delta_N.
  \end{multline*}
\end{proof}

\begin{proof}[Proof of \cref{thm:E}]
  Let $ \Omega_N' $ be the set from \cref{prop:pre-homogeneous} with $ A=3 $. The conclusion
  follows immediately from \cref{prop:pre-homogeneous},
  by letting $ \Omega_N=\bigcup_{N'\ge N}\Omega_{N'}' $ and $ \delta_0=\exp(-(\log N)^{4/\sigma^2})  $
  (provided $ N $ is large enough).
\end{proof}

\def\cprime{$'$}
\providecommand{\bysame}{\leavevmode\hbox to3em{\hrulefill}\thinspace}
\providecommand{\MR}{\relax\ifhmode\unskip\space\fi MR }
\providecommand{\MRhref}[2]{%
  \href{http://www.ams.org/mathscinet-getitem?mr=#1}{#2}
}
\providecommand{\href}[2]{#2}

\end{document}